%% file: Bilimits_representable_final.tex
\documentclass{amsart}

\input{macros}
\usepackage[font=small,format=hang,labelfont={sf,bf}]{caption}

\DeclareRobustCommand{\SkipTocEntry}[5]{}

\newcommand{\comp}{\ast}

\newcommand{\A}{\EuScript{A}}
\newcommand{\B}{\EuScript{B}}
\newcommand{\C}{\EuScript{C}}
\newcommand{\D}{\EuScript{D}}
\newcommand{\E}{\EuScript{E}}
\newcommand{\J}{\EuScript{J}}

\newcommand{\mJ}{\bar{\J}}

\newcommand{\Bx}{\trbis{\B}{x}}
\newcommand{\Bxcart}{(\Bx)_{\text{cart}}}
\newcommand{\Cl}{\trbis{\C}{\ell}}
\newcommand{\CF}{\trbis{\C}{F}}
\newcommand{\mCF}{\C^{/F}}

\newcommand{\mCFcart}{(\mCF)_{\text{cart}}}
\newcommand{\CFmod}{\overline{\trbis{\C}{F^\triangleleft}}}

\DeclareMathOperator{\Di}{D}
\newcommand{\final}{\Di_0}

\newcommand{\car}{\mathrm{cart}}

\title{Bilimits are bifinal objects}
\author{Andrea Gagna}
\address{Institute of Mathematics, Czech Academy of Sciences\\ \v{Z}itn\'a 25 \\115 67   Praha 1\\ Czech Republic}
\email{gagna@math.cas.cz}
\urladdr{https://sites.google.com/view/andreagagna/home}
\author{Yonatan Harpaz}
\address{Institut Galilée\\ Université Paris 13\\ 99 avenue Jean-Baptiste Clément\\ 93430 Villeta-neuse\\ France}
\email{harpaz@math.univ-paris13.fr}
\urladdr{https://www.math.univ-paris13.fr/~harpaz}
\author{Edoardo Lanari}
\address{Institute of Mathematics, Czech Academy of Sciences\\ \v{Z}itn\'a 25 \\115 67   Praha 1\\ Czech Republic}
\email{edoardo.lanari.el@gmail.com}
\urladdr{https://edolana.github.io/}
\subjclass[2020]{18D30, 18D70, 18N10}
\begin{document}
%
% Abstract
%
\begin{abstract}
	We prove that a (lax) bilimit of a \(2\)-functor is characterized by the existence of a limiting contraction in the \(2\)-category of (lax) cones over the diagram. We also investigate the notion of bifinal object and prove that a (lax) bilimit is a limiting bifinal object in the category of cones. Everything is developed in the context of marked 2-categories, so that the machinery can be applied to different levels of laxity, including pseudo-limits.
\end{abstract}

\maketitle

\tableofcontents

\section*{Introduction}

The theory of limits and colimits sits at the core of category theory and
it has become an essential tool to express natural constructions of interest to many areas of mathematics.
As the formalism of category theory matured, it became clear that some important
phenomena are better understood when framed in a 2-categorical setting.
It was therefore natural to pursue a generalization of the useful concept
of (co)limits in such a context.
The notions of 2-limit and 2-colimit were first formulated independently
by Auderset in~\cite{Auderset}, where
the Eilenberg--Moore and the Kleisli category of a monad are recovered as a 2-limit and 2-colimit,
and by Borceux--Kelly~\cite{BorceuxKelly}, who introduced the notion of enriched limits and colimits
and so in particular limits and colimits enriched in the cartesian closed category of small categories.
These notions were further studied and developed by Street~\cite{StreetLimits}, Kelly~\cite{KellyElementary,KellyBasic}
and Lack~\cite{LackCompanion}, who also introduced and investigated the lax and weighted versions.\\

More recently two papers by clingman and Moser appeared in the literature, namely~\cite{ClingmanMoserLimitsDifferent} and~\cite{ClingmanMoserBiinitial},
where they investigate whether the well-known result that limits are terminal cones extends to the 2-dimensional framework.
The first one proves that the answer is negative, \ie~terminal cones are no longer enough to capture the correct universal property, no matter what flavour of slice category one uses.
In the second paper, the authors leverage on results from double-category theory on representability of \(\Cat\)-valued functors to show that being terminal still captures the notion of limit, provided one is willing to work with an alternative 2-category than that of cones, there denoted by \(\mathbf{mor}(F)\) for a given diagram \(F\).

\addtocontents{toc}{\SkipTocEntry}
\subsection*{Motivations}

The main goal of this paper is to clarify, with its main result (Theorem~\ref{thm:bilimits_are_bifinal}), that a natural characterization of lax bilimits in terms of cones is still possible. The use of a marking on the domain of the 2-functor of which we
want to study the bilimit addresses in a fundamental way all the possible
levels of laxity of the bilimit (pseudo, lax or anything in between). We remark
that this further level of generality is necessary from a technical viewpoint to coherently interpolate laxity from pseudo to lax in all the 2-categorical constructions.
For example, a similar technique is also needed in
the development of flat pseudo-functors, and therefore in the theory of 2-categorical
filteredness as well as in the theory of 2-topoi, carried out by  Descotte, Dubuc and Szyld in~\cite{DescotteDubucSzyldSigmaLimits}.

Another reason that drove us to write this paper was filling the gap in the literature as concerns final objects in 2-category theory.
Although some results on final 2-functors already appeared in~\cite{GarciaSternThmA} by Abell\'{a}n Garc\'{i}a and Stern
(and in the \((\infty,2)\)-categorical context in~\cite{GagnaHarpazLanariLaxLimits} by the authors),
we consider establishing the connections between final objects, contractions and bilimits long overdue.

Finally, further motivation comes from the realm of homotopy theory and in particular of \((\infty, 2)\)-category theory.
On the one hand, we wanted to build some low-dimensional intuition based on the \((\infty, 2)\)-categorical treatment
of (lax, weighted) (co)limits we developed in~\cite{GagnaHarpazLanariLaxLimits}. On the other hand, we are convinced that classical 2-category theory and \((\infty, 2)\)-category theory can benefit
significantly from each other techniques and results. Indeed, the results of this paper have been written with a fibrational point of view
typical of weak higher categories. From this perspective, the use of marked edges is all the more natural.
Conversely, 2-category theory encompasses many structures and results that would be useful to generalize to
\((\infty, 2)\)-categories, as they are needed in derived algebraic geometry. Consider for instance the
theory of 2-(co)filtered 2-categories and 2-Ind/2-Pro constructions.
These are worked out by Descotte and Dubuc in~\cite{DescotteDubuc2Pro} for 2-categories
and their \((\infty, 2)\)-categorical counter-part is a useful tool in geometry, see for instance~\cite[\S A.3]{PortaSalaCatHall}
of Porta and Sala.
Profiting to a greater extend of such an interactive relationship will be the subject of further investigations.

\addtocontents{toc}{\SkipTocEntry}
\subsection*{Technical summary}
In this work, we insist on keeping the category of cones as the object of interest, but we claim that the notion of ``terminality'' is not the correct one to consider (hence the no-go theorem of~\cite{ClingmanMoserLimitsDifferent}). Instead, we focus our attention on bifinal objects, and we make use of the formalism of marked \(2\)-categories and contractions to characterize (lax) bilimits as bifinal objects in the 2-category of cones. The reason why this is not needed in the 1-dimensional case is that in such context final objects coincide with terminal ones.
That is to say, an object \(c\) of a category \(C\) is terminal if and only if the inclusion functor \(\{c\} \hookrightarrow C\) is a final functor
or equivalently if the projection \(\trbis{C}{c} \to C\) has a section mapping \(c\) to the identity \(1_c\).
It is therefore natural to investigate the meaning of finality of a 2-functor \(\{c\} \hookrightarrow \C\), with \(\C\) a 2-category with a set of marked 1-cells,
and to inspect well-behaved sections of the projection 2-functor \(\trbis{\C}{c} \to \C\) for a pseudo/lax version of the slice 2-category.
By doing so, one is led to consider conditions of \emph{local terminality}, by which we mean
choices of terminal objects in the hom-categories \(\C(a, c)\), for all objects \(a\) of \(\C\).
Exploiting this point of view, we show that the appropriate 2-categorical generalizations of these characterizations of terminality
still coincide and are equivalent to a ``coherent'' notion of local terminality known as \emph{contraction}.

An archetypical example of terminal object in category theory is given by the identity arrow \(1_x\),
thought as an object of \(\trbis{C}{x}\) for \(C\) a category and \(x\) an object of \(C\).
Armed with the understanding of 2-categorical terminal conditions as contractions or as suitable final 2-functor from a singleton,
one can examine the (lax) 2-slice \(\trbis{\C}{x}\) for a 2-category \(\C\) and an object \(x\) of \(\C\).
This is the 2-categorical equivalent of the slice, but where the triangles of \(\C\) with tip \(x\)
defining the 1-cells \(\trbis{\C}{x}\) are filled with a 2-cell of \(\C\).
A first interesting observation concerns the local terminal objects: for an arrow \(f \colon a \to x\) of \(\C\),
the terminal objects of the hom-category \(\trbis{\C}{x}(f, 1_x)\) are precisely the cartesian edges of the
projections \(\trbis{\C}{x} \to \C\). These cartesian edges are simply the triangles \(\C\) with edges
\(f\) and \(1_x\)
such that the 2-cell filling it is invertible. Said otherwise, such a cartesian edge is the data of
a 1-cell \(g \colon a \to x\) of \(\C\) together with an invertible 2-cell between \(f\) and \(g\).
But there are potentially other cartesian edges of \(\trbis{\C}{x}\)
from \(f\colon a \to x\) to another arrow \(g \colon b \to x\) of \(\C\),
which are just 1-cells \(h \colon a \to b\) together with an invertible 2-cell
between \(f\) and \(hg\). Adopting a marking of these cartesian edges,
which acts as a book-keeping device, we show that \(1_x\) satisfies a finer finality property with respect to
all the cartesian edges. In particular, we find that \(1_x\) is \emph{biterminal} in the sub-2-category
where the objects are the same as in \(\trbis{\C}{x}\) and the hom-categories are the full sub-categories
spanned by the cartesian edges. So a more classical terminality property can be obtained,
if we are willing to restrict to these triangles filled with invertible 2-cells.
But this biterminal property \emph{does not} characterise the stronger local terminality in the
full (lax) 2-slice.

Once we turn to (lax) bilimits, keeping track of all the cartesian edges proves to be the right approach.
Given 2-functor \(F \colon \J \to \C\) and an object \(\ell\) of \(\C\),
a (lax) 2-cone over \(F\) with tip \(\ell\) is a (lax) 2-natural transformation from the constant 2-functor on~\(\ell\) to \(F\). This is the 2-categorical equivalent of a cone over \(F\), where the triangles
constituting the 2-cone are filled with 2-cells, which can be invertible or not according
to the laxity prescribed. The (lax) 2-cones over \(F\) can be canonically organised to form
a 2-category of (lax) 2-cones \(\trbis{C}{F}\). Again this is similar to the classical
category of cones over a functor, but now every triangle is filled with a 2-cell of \(\C\).
A pair \((\ell, \lambda)\) is said to be a (lax) bilimit of the functor \(F\) if the representable
hom-category \(\C(x, \ell)\) is equivalent to the category of (lax) 2-cones over \(F\) with tip \(x\);
that is, the pair \((\ell, \lambda)\) 2-represents the (lax) 2-cones over \(F\).
The main result of the paper (Theorem~\ref{thm:bilimits_are_bifinal}, together with Proposition~\ref{prop: limiting iff final}) states that the pair \((\ell, \lambda)\)
is a (lax) bilimit of \(F\) if and only if the pair \((\ell, \lambda)\), thought as an object of the
2-category \(\trbis{\C}{F}\) of (lax) 2-cones over \(F\), is final with respect to the class of cartesian
edges of the projection \(\trbis{\C}{F} \to \C\). 
The local terminality that we get is of the following kind: If \((\ell, \lambda)\) is a (lax) bilimit
of the 2-functor \(F \colon \J \to \C\) and \((x, \alpha)\) is another (lax) 2-cone over \(F\),
then every morphism of (lax) 2-cones \((f, \sigma)\) from \((x, \alpha)\) to \((\ell, \lambda)\),
\ie any object of the hom-category \(\trbis{\C}{F}\bigl((x, \alpha), (\ell, \lambda)\bigr)\),
with \(\sigma\) an invertible ``2-cell'' is terminal; here \(f \colon x \to \ell\) is a morphism of \(\C\),
\(\alpha\) is a (lax) cone of $x$ over \(F\) and \(\sigma\)
is a modification from \(\lambda\) precomposed (component-wise) by \(f\) to \(\alpha\).
Said otherwise, any such morphism of (lax) cones over \(F\), morally a triangle over \(F\),
filled by an invertible 2-cell is terminal in the appropriate hom-category.
These morphisms of (lax) 2-cones are precisely the cartesian edges of the projection
\(\trbis{\C}{F} \to \C\) with target the (lax) 2-cone \((\ell, \lambda)\).
In particular, if we restrict to the hom-categories spanned by the cartesian edges we get that \((\ell, \lambda)\)
is biterminal.

\medskip

\addtocontents{toc}{\SkipTocEntry}
\subsection*{Structure of the article}
The paper is structured as follows. After a preliminary section where we fix the notation,  we briefly recall the necessary background on 2-categories and relevant constructions, namely joins, slices and the Grothendieck construction for fibrations of 2-categories.

We then move on to Section~\ref{sec:bilimits}, where we introduce lax marked bilimits and contractions (following Descotte, Dubuc and Szlyd~\cite{DescotteDubucSzyldSigmaLimits}), and we prove in Proposition \ref{prop:E-final-iff-E-contraction} that final objects can be characterized in several ways, one of which involves contractions.

Next, Section~\ref{sec:car-edges} is an investigation on representable fibrations and the properties of the corresponding representing objects. By looking at the cartesian morphisms of the projection \(p \colon \mCF \to \C\) we are able to highlight the difference from the simpler case of ordinary 1-categories, drawing a parallelism in Proposition \ref{prop:biterminal_Bx}.
Moreover, we carve out a full subcategory of the category of maps over a given cone on a diagram, that we prove in Proposition~\ref{prop:mod_slice_trivial_fibration_over_point} to be equivalent to the slice over the tip of the given cone. Note that this is a low-tech version of Corollary 5.1.6 from~\cite{GagnaHarpazLanariLaxLimits}.

Finally, in the last section we blend all together, culminating in the main result recorded as Theorem \ref{thm:bilimits_are_bifinal}, which characterizes (lax, marked) bilimits as \emph{limiting bifinal} cones. Furthermore, we prove that such bilimits are also terminal in the appropriate subcategory of cones obtained by restricting to cartesian morphisms between them.

\addtocontents{toc}{\SkipTocEntry}
\section*{Acknowledgements}
The first and third author gratefully acknowledge the support of Praemium Academiae of M.~Markl and RVO:67985840.
The authors also thank the anonymous reviewer for the precise and detailed recommendations.

\section{Preliminaries}

\subsection{Basic notions and notations}

\begin{parag}
	Recall that a \(2\)-category \(\C\) is composed by a class \(\Ob(\C)\) of objects or
	0-cells, a class of \(1\)-cells, a class of \(2\)-cells, identities for objects and 1-cells,
	source and target functions	for 1-cells and 2-cells, together with a ``horizontal''
	composition \(\comp_0\) for 1-cells and 2-cells and a ``vertical'' composition \(\comp_1\)
	for 2-cells.
	The identity of an object or 1-cell \(x\) will be denoted by \(1_x\).
	We will denote the composition of \(1\)-cells just by juxtaposition.
	Given a 1-cell \(f \colon x \to y\) and a 2-cell \(\alpha \colon g\to g' \colon y \to z\),
	we will denote by \(\alpha \comp_0 f\) the whiskering, by which we mean the composition \(\alpha \comp_0 1_f\).
	For any pair of objects \(x, y\), we will denote by \(\C(x, y)\)
	the category whose objects are 1-cells of \(\C\) from \(x\) to \(y\)
	and the morphisms are the 2-cells between these.

	The category of small 2-categories and strict 2-functors will be denoted by \(\nCat{2}\),
	while its full subcategory spanned by 1-categories will be denoted by \(\Cat\).
	The 2\nbd-cat\-e\-gory of small 2-categories, pseudo-functors and pseudo-natural transformations will
	be denoted by~\(\nCat{2}_{\text{ps}}\). By pseudo-functor we will always mean normal pseudo-functors, \ie
	unit preserving.
\end{parag}

\begin{example}
	The canonical example of a 2-category is given by \(\Cat\), where natural transformations play the role of
2-cells.
\end{example}

\begin{parag}
	We will denote by
	\[
		\Di_0 = \bullet
		\quad , \quad
		\Di_1 = 0 \to 1
		\quad , \quad
		\Di_2 =
			\begin{tikzcd}
				0 \ar[r, bend left, ""{swap, name=s}] \ar[r, bend right, ""{name=t}] & 1
				\ar[Rightarrow, from=s, to=t]
			\end{tikzcd}
	\]
	the 2-categories corepresenting, respectively, the free-living object, 1-cell and 2-cell in \(\nCat{2}\).
\end{parag}

\begin{parag}
	Given a 2-category \(\C\), one can consider the 2-category
	\(\C^\op\) obtained from \(\C\) by reversing all the 1-cells,
	\ie formally swapping source and target. By reversing the 2-cells of \(\C\), we get another 
	2-category which we denote by \(\C_{\co}\). Thus, for every pair of objects \(x, y\), we have \(\C(x, y)^\op = \C_{\co}(x,y)\).
	Combining these two dualities, we get a 2-category \(\C^\op_{\co}\), 
	where both 1-cells and 2-cells of \(\C\) have been reversed.
\end{parag}

\begin{parag}
\label{par: bieq}
	A 1-cell \(f \colon x \to y\) of a 2-category \(\C\) is an \ndef{equivalence}
	if we can find a 1-cell \(g \colon y \to x\) together with invertible
	2-cells \(gf \to 1_x\) and \(fg \to 1_y\).
	A \ndef{biequivalence} of 2-categories is a 2-functor \(F \colon \C \to \D\)
	such that
	\[
		\C(x, y) \to \D(Fx, Fy)
	\]
	is an equivalence of categories for every pair of objects \(x, y\) of \(\C\),
	and for every object \(z\) of \(\D\) we can find an object \(c\) of \(\C\)
	equipped with an equivalence \(Fc \to z\).
	Biequivalences of 2-categories are the weak equivalences
	for a model category structure of 2-categories established by Lack~\cite{LackModel2Cat, LackModelBicat}.
	In particular, we shall use the fact that the class of biequivalences
	enjoys the 2-out-of-3 property.
	Furthermore, in order to check whether a 2-functor \(F \colon \C \to \D\)
	is a biequivalence it is sufficient to check that:
	\begin{itemize}
		\item the 2-functor \(F\) is surjective on objects,
		\item the 2-functor \(F\) is \emph{full on 1-cells}, \ie lifts 1-cells with prescribed lifts of the boundary \(0\)-cells (\ie objects),
		\item the 2-functor \(F\) is \emph{fully faithful} on 2-cells, \ie it uniquely lifts 2-cells with prescribed lifts of the boundary \(1\)-cells. 
	\end{itemize}
	In fact, such a 2-functor is a trivial fibration in the above-mentioned model category structure. In detail, the first condition corresponds to a lifting property of the form:
	\[\begin{tikzcd}
		\emptyset \ar[d] \ar[r] & \C \ar[d,"F"]\\
		\Di_0 \ar[r] & \D .
	\end{tikzcd}\]
	The second one is encoded by the following lifting property:
	\[\begin{tikzcd}
		\Di_0 \coprod \Di_0 \ar[d,"(s{,}t)"{swap}] \ar[r] & \C \ar[d,"F"]\\
		\Di_1 \ar[r] & \D .
	\end{tikzcd}\]
	Here we used the globular notation, so that \(s, t \colon \Di_0 \to \Di_1\) are the functors mapping the unique object of \(\Di_0\) to \(0\), resp.~\(1\).
	Finally, the third condition merges together two lifting properties, depicted below:
	\[ \begin{tikzcd}
		{\displaystyle \Di_1 \coprod_{D_0 \coprod D_0} \Di_1 }\ar[d,"(s{,}t)"{swap}] \ar[r] & \C \ar[d,"F"] & & & {\displaystyle \Di_2 \coprod_{D_1 \coprod D_1} \Di_2} \ar[d,"(\mathrm{Id}{,}\mathrm{Id})"{swap}] \ar[r] & \C \ar[d,"F"]\\
		\Di_2 \ar[r] & \D & & & \Di_2 \ar[r] & \D .
	\end{tikzcd}\]

	An analogous characterisation holds if \(F \colon \C \to \D\) is instead a pseudo-functor (see~\cite{LackModelBicat}).
\end{parag}

\begin{define}
	A \ndef{marked \(2\)-category} is a pair \((\C, E)\) where
	\(\C\) is a \(2\)-category and \(E\) is a class of \(1\)-cells in 
	\(\C\) containing the identities.
	A \ndef{marked 2-functor} between marked 2-categories
	\(F \colon (\C, E_\C) \to (\D, E_\D)\) is a 2-functor
	\(F \colon \C \to \D\) that maps every marked 1-cell in \(E_\C\)
	to a marked 1-cell in \(E_\D\), \ie \(F(E_\C) \subseteq E_\D\).
	We denote by \(\nCat{2}^+\) the category of marked 2-categories and marked 2-functors.

	Whenever we will consider a marked 2-category, we will only mention the non-trivial
	marked \(1\)-cells, \ie those which are not identities.
\end{define}

\begin{parag}
	We say that an object \(c\) of a 2-category \(\C\) is \ndef{quasi-terminal},
	or that \(\C\) \ndef{admits~\(c\) as a quasi-terminal object}, if for
	any object \(x\) of \(\C\) the category \(\C(x, c)\) has a terminal object.
	The notion of quasi-terminal object was introduced in~\cite{JayLocalAdjunction} by Jay
	under the name of locally terminal object. 
	The term \emph{quasi-terminal} appears also in~\cite{AraMaltsiCondE}, where Ara and Maltsiniotis extend this
	to strict \(n\)-categories, for \(1 \leq n \leq \infty\).
\end{parag}

\begin{parag}
	Let \(\bar{\J} = (\J, E)\) be a marked 2-category, \(\C\) be
	a 2-category and \(F, G \colon \J \to \C\) be 2-functors.
	A \ndef{lax \(E\)-natural transformation} \(\alpha \colon F \to G\) consists of the following data:
	\begin{itemize}
		\item a 1-cell \(\alpha_i \colon Fi \to Gi\) of \(\C\) for any \(i\) in \(\Ob(\J)\),

		\item a 2-cell \(\alpha_k \colon G(k) \comp_0 \alpha_i \to \alpha_j \comp_0 F(k)\) of \(\C\),
			that we depict as follows
			\[
				\begin{tikzcd}
					Fi \ar[r, "F(k)"] \ar[d, "\alpha_i"']	& Fj \ar[d, "\alpha_j"]	\\
					Gi \ar[r, "G(k)"']					& Gj,
					\ar[Rightarrow, from=2-1, to=1-2, shorten <=3mm, shorten >=2.5mm, "\alpha_k"']
				\end{tikzcd}
			\] 
			for every 1-cell \(k \colon i \to j\) of \(\J\),
 	\end{itemize}
 	satisfying the following conditions:
 	\begin{description}
 		\item[identity] we have \(\alpha_{1_i} = 1_{\alpha_i}\) for all objects \(i\) of \(\J\);
 		\item[marking] the 2-cell \(\alpha_k\) is invertible for all \(k\) in \(E\);

 		\item[compositions] we have \(\alpha_{lk} = (\alpha_l \comp_0 Fk) \comp_1 (Gl \comp_0 \alpha_k)\),
 			for all \(k \colon i \to i'\) and \(l \colon i' \to i''\) 1-cells in \(\J\),

 		\item[compatibility] we have \(\alpha_l \comp_1 (G\delta \comp_0 \alpha_i) =
 			(\alpha_j \comp_0 F\delta) \comp_1 \alpha_k\), for every 2-cell
 			\(\delta \colon k \to l \colon i \to j\) of \(\J\).
 	\end{description}

 	We will denote by \([\J, \C]_E\) the \pdftwo-category of \pdftwo-functors
 	from \(\J\) to \(\C\), lax \(E\)-natural transformations and modifications.
\end{parag}

\begin{rem}
	In the notation of the previous paragraph,
	if \(E\) consists only of the identity 1-cells of \(\J\) (resp.~of all the 1-cells of \(\J\)), we recover the notion of \emph{lax} natural transformation (resp.~\emph{pseudo} natural transformation).
\end{rem}

\begin{rem}
	The notion of $E$-natural transformation was introduced in~\cite[Definition~2.1.1]{DescotteDubucSzyldSigmaLimits}
	by Descotte, Dubuc and Szyld, where they are called $\sigma$-natural transformation and the marking is denoted by $\Sigma$.
\end{rem}

\subsection{Join and slices of \pdftwo-categories}
\label{sec:slices}

\begin{parag}
	The notions of join and slice were generalized by Ara and Maltsiniotis~\cite{AraMaltsiniotisJoin}
	to the category of strict \(\infty\)-categories (also known as \(\omega\)-categories).
	By truncating, their notion also provides a definition for strict \(n\)-categories
	(see~\cite[Ch.~8]{AraMaltsiniotisJoin}).
	In fact, there are at least two sensible notions of join that one can give for
	strict higher categories, due to the choice of the variance of the higher cells:
	the lax and the oplax join
	(cf.~\cite[Remark~6.37]{AraMaltsiniotisJoin}). The two joins collapse to the
	classical \(1\)-categorical join once we truncate with respect to the higher cells.
	This choice of variance is important
	in relation to the kind of slice one intends to consider. We shall mainly consider
	slices \emph{over} a \(2\)-functor, and in this case the lax variance enjoys
	better formal properties. We start by recalling the general formalism that allows us
	to get the generalised (op)lax slices.
\end{parag}

\begin{parag}
	Given two 2-categories \(\A\) and \(\B\), their lax join \(\A \ast \B\) is the following 2-category:
	\begin{itemize}
		\item the objects are those of \(A \coprod B\), that we denote by $a \star \varnothing$ and \(\varnothing \star b\), for \(a \in \Ob(\A)\) and \(b \in \Ob(\B)\);
		
		\item to the 1-cells of the coproduct 2-category \(A \coprod B\), that we denote by $f \star \varnothing$ and \(\varnothing \star g\),
		for any \(f \colon a \to a'\) in \(\A\) and any \(g \colon b \to b'\) in \(\B\),  we add 
		 a 1-cell \(a \star b \colon a \star \varnothing \to \varnothing \star b\) for every pair of objects \((a, b)\) in \(\Ob(\A)\times \Ob(\B)\),
		and then closing by composition in the obvious way;
		
		\item to the 2-cells of the coproduct 2-category \(A \coprod B\), that we denote by $\alpha \star \varnothing$ and \(\varnothing \star \beta\),
		for any \(\alpha \colon f \Rightarrow f'\) in \(\A\) and every \(\beta \colon g \Rightarrow g'\) in \(\B\),  we add 
		2-cells
		\[
			\begin{tikzcd}[column sep=small, row sep=small]
				& {\varnothing \star b} \\
				{a \star \varnothing} \\
				& {\varnothing\star  b'}
				\arrow["{\varnothing\star g}", from=1-2, to=3-2]
				\arrow["{a \star b}", from=2-1, to=1-2]
				\arrow["{a\star b'}"{swap}, ""{anchor=center,name=ab}, from=2-1, to=3-2]
				\arrow["{a \star g}"{description}, shorten <=3pt, shorten >= 1pt, Rightarrow, from=1-2, to=ab]
			\end{tikzcd}
			\quad\text{and}\quad
			\begin{tikzcd}[column sep=small, row sep=small]
				a \star \varnothing \ar[rd, "a\star b", ""{anchor=center, name=ab}]	& 	\\
																	& \varnothing \star b 	\\
				a' \star \varnothing \ar[ru, "a'\star b"{swap}]		&
				\ar[from=1-1, to=3-1, "f\star \varnothing"{swap}]
				\ar[from=3-1, to=ab, Rightarrow, shorten <= 3pt, shorten >= 1pt, "f \star b"{description}]
			\end{tikzcd}
		\]
		for any \(f\colon a \to a'\) in \(\A\) and \(g \colon b \to b'\) in \(\B\).
		We close by composition in the obvious way, but imposing the following relations.

		\item For any \(a\) in \(\Ob(\A)\) and \(b\) in \(\Ob(\B)\), the 2-cells \(a \star 1_b\) and \(1_a \star b\)
		are identities;
		
		\item for any \(\alpha \colon f \Rightarrow f'\) in \(\A\) and \(b\) in \(\Ob(\B)\), we impose
		\[
			f' \star b \ast_1 (a'\star b \ast_0 \alpha \star \varnothing ) = f \star b; 
		\]
		
		\item for any \(a\) in \(\Ob(\A)\) and \(\beta \colon g \Rightarrow g'\) in \(\B\), we impose
		\[
			a \star g' \ast_1 (\varnothing \star \beta \ast_0 a \star b) = a \star g;
		\]

		\item for any \(f \colon a \to a'\) in \(\A\) and any \(g \colon b \to b'\) in \(\B\), we impose
		\begin{nscenter}
			\begin{tikzpicture}[scale=1.5]
				\square{
					/square/label/.cd,
					0 = {$a \star \varnothing$}, 1 = {$a' \star \varnothing$}, 2 = {$\varnothing \star b$}, 3 = {$\varnothing \star b'$},
					01 = {$f \star \varnothing$}, 12 = {$a\star b$}, 23 = {$\varnothing \star g$}, 03 = {$a \star b'$},
					02= {$a \star b$}, 13 = {$a' \star b'$},
					012 = {$f \star b$}, 023 = {$a \star g$}, 013 = {$f \star b'$}, 123 = {$a' \star g$},
					/square/arrowstyle/.cd,
					012={Leftarrow}, 023={Leftarrow}, 013={Leftarrow}, 123={Leftarrow},
					0123 = {equal}
				}
			\end{tikzpicture}
		\end{nscenter}
	\end{itemize}
	This gives a monoidal category structure
	on \(\nCat{2}\), the category of small 2-categories, whose unit is the empty 2-category.
	This monoidal structure is \emph{not} closed, but it is locally closed in the sense that the functors
	\begin{align*}
		\nCat{2} 	&\to \overslice{\nCat{2}}{\A}	 & \nCat{2} & \to \overslice{\nCat{2}}{\B} \\
		\B 		& \mapsto (\A \ast \B, i_{\A} \colon \A \to \A \ast \B )	&
		\A 		& \mapsto (\A \ast \B, i_{\B} \colon \B \to \A \ast \B),
	\end{align*}
	both have a right adjoint, where \(\begin{tikzcd}\A \ar[r, "i_{\A}"] & \A\ast \B & \B \ar[l, "i_{\B}"']\end{tikzcd}\) are the
	canonical inclusions.
	These adjoint functors are denoted by
	\begin{align*}
		\overslice{\nCat{2}}{\A}		&\to  \nCat{2} 		& \overslice{\nCat{2}}{\B} &\to \nCat{2}\\
		(\C, u \colon \A \to \C) 	& \mapsto \overslice{\C}{u}	&
		(\C, v \colon \B \to \C)	& \mapsto \trbis{\C}{v}
	\end{align*}
	and called the (generalized) \ndef{lax slice} functors.
	The 2-categories \(\overslice{C}{u}\) and \(\trbis{C}{v}\) correspond
	to the 2-category of lax cocones under \(u\) and the 2-category of lax cones over \(v\), respectively.
	If \(\A\) and \(\B\) are both the terminal (2-)category \(\Di_0\) and \(u\) and \(v\)
	correspond to the object \(c\) of \(C\), then the generalized lax slices over and under the object \(c\)
	will simply be denoted by \(\overslice{\C}{c}\) and \(\trbis{\C}{c}\).
\end{parag}

\begin{parag}
 \label{parag:lax_slice_under_point}
	We now provide an explicit description of the \(2\)-categorical lax slice over a point.
	Let \(\C\) be a \(2\)-category and \(c\) be an object of \(\C\).
	\begin{itemize}
		\item The objects of \(\trbis{\C}{c}\) are pairs \((x, \alpha)\), with \(x\) in \(\Ob(\C)\)
	and \(\alpha \colon x \to c\) a \(1\)-cell of~\(\C\).
	They correspond to \(2\)-functors
	\(\Di_0 \ast \Di_0 \to \C\) and \(\Di_0 \ast \Di_0 \cong \Di_1\).

		\item A \(1\)-cell of \(\trbis{\C}{c}\) from \((x, \alpha)\) to \((y, \beta)\) is
			given by a \(2\)-functor \(\Di_1 \ast \Di_0 \to \C\) such that its restriction
			to \(\{0\} \ast \Di_0\) is \((x, \alpha)\) (resp.~to \(\{1\}\ast \Di_0\) is \((y, \beta)\)).
			Explicitly, this means that such a \(1\)-cell is a pair \((f, \gamma)\), where
			\(f \colon x \to y\) is a \(1\)-cell and \(\gamma \colon \beta f \to \alpha\)
			is a \(2\)-cell of \(\C\), that we depict by
			\[
				\begin{tikzcd}[column sep = tiny]
					x \ar[rr, "f"] \ar[rd, "\alpha"{swap, name=a}] 	&	& y \ar[ld, "\beta"{pos=0.4}] \\
													& c &
					\ar[Rightarrow, from=1-3, to=a, shorten <=3mm, shorten >=2mm, "\gamma"{swap}]
				\end{tikzcd}
			\]

		\item A \(2\)-cell of \(\trbis{\C}{c}\) from \((f, \gamma)\) to \((g, \delta)\) is
			given by a \(2\)-functor \(\Di_2 \ast \Di_0 \to \C\) such that the appropriate restrictions
			to \(\Di_1 \ast \Di_0\) are \((f, \gamma)\) to \((g, \delta)\).
			Explicitly, this means that such a \(2\)-cell is given by a \(2\)-cell \(\Xi\colon f \to g\) of \(\C\)
			satisfying 
			\[
				\delta \comp_1 (\beta \comp_0 \Xi) = \gamma,
			\]
			that we can depict by
			\[
				\begin{tikzcd}[column sep=small, row sep=3.2em]
					x \ar[rr, bend left, "f", ""{swap, name=s}]
						\ar[rr, bend right, ""{name=t}, "g"{above right = 0pt and 2pt}]
						\ar[rd, "\alpha"{swap, name=a}]	&	& y \ar[ld, "\beta"{pos=0.4}, ""{swap, pos=0.1, name=b}]	\\
												& c &
					\ar[Rightarrow, from=s, to=t, "\Xi"{swap}]
					\ar[Rightarrow, from=b, to=a, shorten <=3mm, shorten >=3mm, "\delta"]
				\end{tikzcd}
				=
				\begin{tikzcd}[column sep=small, row sep=3.2em]
					x \ar[rr, bend left, "f"] \ar[rd, "\alpha"{swap, name=a}]
						&	& y \ar[ld, "\beta"{pos=0.4}, ""{swap, pos=0.1, name=b}]	\\
						& c &
					\ar[Rightarrow, from=b, to=a, shorten <=3mm, shorten >=3mm, "\gamma"']
				\end{tikzcd}
			\]
	\end{itemize}
	Compositions of cells is inherited by \(\C\).

	The 2-category \(\overslice{\C}{c}\) admits a similar description,
	and it is canonically isomorphic to the 2-category \(\bigl(\trbis{\C^\op}{c}\bigr)^\op\).
\end{parag}

\begin{parag}
	Based on the description of the slice over an object of a \(2\)-category,
	we provide the description of the lax slice of a \(2\)-category \(\C\)
	over a \(2\)-functor \(F \colon \J \to \C\).
	\begin{itemize}
		\item The objects of \(\trbis{\C}{F}\) are pairs \((x, \alpha)\), where \(x\) is an object of \(\C\), and \(\alpha\) consists of a family of \(1\)-cells 
			\((\alpha_i \colon x \to Fi)_{i \in \Ob(\J)}\) of \(\C\), and a family of \(2\)-cells
			\[
				\begin{tikzcd}[column sep=tiny]
											& x \ar[ld, "\alpha_i"'] \ar[rd, "\alpha_j"{name=t}]	&	\\
						Fi\phantom{'} \ar[rr, "F(k)"']	&														& Fj
						\ar[Rightarrow, from=2-1, to=t, shorten <= 2mm, shorten >= 3mm, "\alpha_k"{swap}]
					\end{tikzcd}
			\] indexed by \(1\)-cells \(k \colon i \to j\) of \(\J\), such that for every \(2\)-cell \(\Lambda \colon k \to k'\) in \(\J\) the following relation is satisfied
			\begin{equation}
			\label{eq:coherence_point-vs-2ell}
				\alpha_k = \alpha_{k'} \comp_1 (F\Lambda \comp_0 \alpha_i).
			\end{equation}
			Furthermore, this assignment has to respect identities and be functorial.
			More precisely, \(\alpha_{1_i} = 1_{\alpha_i}\) for all objects \(i\) of \(\J\)
			and given another \(1\)-cell \(h \colon i' \to i''\) 
			of \(\J\), we have
			\[
				\alpha_{hk} = \alpha_h \comp_1 (Fh \comp_0 \alpha_k).
			\]
			They correspond to \(2\)-functors \(\Di_0 \ast \J \to \C\)
			such that the restriction to the \(2\)-functor
			\(\emptyset \ast \J \to \C\) is
			\(F\), \ie \ndef{lax cones over \(F\)}.

		\item A \(1\)-cell of \(\trbis{\C}{F}\) from \((x, \alpha)\) to \((y, \beta)\) is
			given by a \(2\)-functor \(\Di_1 \ast \J \to \C\) such that its restriction
			to \(\{0\} \ast \Di_0\) is \((x, \alpha)\) (resp.~to \(\{1\}\ast \Di_0\) is \((y, \beta)\)).
			It is a \ndef{lax morphism of lax cones over \(F\)}. 
			Explicitly, this means that such a \(1\)-cell is a pair \((f, \mu)\), where
			\(f \colon x \to y\) is a \(1\)-cell of \(\C\) and \(\mu_i \colon \beta_i f \to \alpha_i\)
			is a \(2\)-cell of \(\C\) for all \(i\) in \(\Ob(\J)\), that we depict by
			\[
				\begin{tikzcd}[column sep = tiny]
					x \ar[rr, "f"] \ar[rd, "\alpha_i"{swap, name=a}] 	&	& y \ar[ld, "\beta_i"{pos=0.4}] \\
													& Fi &
					\ar[Rightarrow, from=1-3, to=a, shorten <=3mm, shorten >=2mm, "\mu_i"{swap}]
				\end{tikzcd}
			\]
			That is, for every object \(i\) of \(\J\) the pair \((f, \mu_i)\) is a \(1\)-cell
			in \(\trbis{\C}{Fi}\), which corresponds to the restriction
			\(\Di_1 \ast \{i\} \to \C\).
			Moreover, the set of \(2\)-cells \(\mu\) must satisfy the following relations:
			\begin{equation}
			\label{eq:coherence_1cell-vs-1cell}
				\alpha_k \comp_1 (Fk \comp_0 \mu_i) = \mu_j \comp_1 (\beta_k \comp_0 f) 
			\end{equation}
			for every \(1\)-cell \(k \colon i \to j\) of \(\J\),
			that we can depict as the following commutative diagram
			\begin{nscenter}
		 		\begin{tikzpicture}
					\square{
						/square/label/.cd,
							0={$x$}, 1={$y$}, 2={$Fi$}, 3={$Fj$},
							01={$f$}, 12={$\beta_i$}, 23={$Fk$}, 02={$\alpha_i$},
							13={$\beta_j$}, 03={$\alpha_j$},
							012={${\mu}_i$}, 023={$\alpha_k$},
							013={${\mu}_j$}, 123={$\beta_k$},
						/square/arrowstyle/.cd,
							012={Leftarrow}, 023={Leftarrow},
							013={Leftarrow}, 123={Leftarrow},
							0123 = {equal}
					}
			 	\end{tikzpicture}
			\end{nscenter}
			in \(\C\) and corresponds to \(\Di_1 \ast \{k \colon i \to j\} \to \C\).

		\item A \(2\)-cell of \(\trbis{\C}{F}\) from \((f, \mu)\) to \((g, \delta)\) is
			given by a \(2\)-functor \(\Di_2 \ast \J \to \C\) whose restrictions
			to \(\Di_1 \ast \J\) are \((f, \mu)\) to \((g, \delta)\).
			Explicitly, this means that such a \(2\)-cell is given by a \(2\)-cell \(\Xi \colon f \to g\) of \(\C\)
			satisfying \(\delta_i \comp_1 (\beta_i \comp_0 \Xi) = \mu_i\) for all \(i\) in \(\Ob(\J)\),
			\ie it provides a \(2\)-cell from \((f, \mu_i)\) to \((g, \delta_i)\) in the slices of
			\(\C\) over the objects \(Fi\).
	\end{itemize}
\end{parag}

\begin{rem}
 \label{rem:lax_cones_are_lax_transformations}
	From the explicit definition, it is straightforward to check that an object in $\trbis{\C}{F}$ corresponds
	to a pair \((x, \alpha)\), where $x$ is an object of \(\C\) and \(\alpha \colon \Delta x \to F\) is a lax
	natural trasformation, where \(\Delta x\) is the constant 2-functor on \(x\)
	from \(\J\) to~\(\C\). A 1-cell of \(\trbis{\C}{F}\) corresponds to a pair \((f, \mu)\),
	where \(f \colon x \to y\) is a 1-cell of \(\C\) and \(\mu \colon \beta \cdot \Delta f \to \alpha\)
	is a modification.
\end{rem}

\begin{rem}
 \label{rem:1-cells_of_CF}
	For a \(1\)-cell \((f, \mu) \colon (x, \alpha)\to (y, \beta)\) of \(\trbis{\C}{F}\)
	we have not imposed any equations to be satisfied for a \(2\)-cell \(\Gamma\colon k \to k'\)
	of \(\J\). This is because they are implied by the relations already present.
	To be more precise, let us consider \(\Di_1 \ast \Di_2\) as a strict \(\infty\)-category,
	in fact a strict \(4\)-category, and let us describe its \(2\)-categorical truncation in detail.

	There are four ``generating'' \(3\)-cells in \(\Di_1 \ast \Di_2\) (generating here is meant
	in the sense of polygraphs or computads, see for instance \cite[1.4]{AraMaltsiniotisJoin}).
	These are sent via \((f, \mu) \colon \Di_1 \ast \Di_2 \to \C\) to the identities witnessing:
	\begin{itemize}
		\item the relation~\eqref{eq:coherence_point-vs-2ell} for \(x\) and \(\Lambda\), that we denote
		by \(x \ast \Lambda\);
		\item  the relation~\eqref{eq:coherence_point-vs-2ell} for \(y\) and \(\Lambda\), that we denote
		by \(y \ast \Lambda\);
		\item the relation~\eqref{eq:coherence_1cell-vs-1cell} for \(f\) and \(k\), that we denote by \(f \ast k\);
		\item the relation~\eqref{eq:coherence_1cell-vs-1cell} for \(f\) and \(k'\), that we denote by \(f \ast k'\).
	\end{itemize}
	These are all relations that we have by assumption on \((x, \alpha)\) and \((y, \beta)\) or
	by the requirement for \(1\)-cells of \(\J\) that \((f, \mu)\) must satisfy.
	The unique \(4\)-cell of \(\Di_1 \ast \Di_2\) has as source (resp.~target) an appropriate whiskered composition
	of the preimages of \(x \ast \Lambda\) and \(f \ast k\) (resp.~of \(f \ast k'\) and \(y \ast \Lambda\)).
	Since truncation identifies source and target \(2\)-cells of every \(3\)-cell,
	we see that the condition for (what we can denote by) \(f \ast \Lambda\) is already satisfied.

	A similar argument explains why a \(2\)-cell \(\Xi\) of \(\trbis{\C}{F}\)
	needs no further coherence constraints than those corresponding to objects of \(\J\), the ones with respect to
	\(1\)-cells and \(2\)-cells of \(\J\) being automatically satisfied.
\end{rem}

\begin{parag}\label{parag:marked-slice}
	Let \(\bar{\J} = (\J, E)\) be a marked 2-category and \(F \colon \J \to \C\)
	a 2-functor. As with the natural transformations, we can use the marking \(E\)
	in order to impose invertibility conditions on the slice \(\trbis{\C}{F}\).
	We denote by \(\C^{/F}_E\), or simply by \(\C^{/F}\) when \(E\) is clear
	from the context, the full sub-2-category of \(\trbis{\C}{F}\)
	spanned by the objects \((x, \alpha)\) for which the 2-cell \(\alpha_k \colon Fk \ast_0 \alpha_i \to \alpha_j\)
	is invertible for all \(k\colon i \to j\) in \(E\).

	If \(E\) consists of just the identity 1-cells of \(\J\), the \(2\)-categories \(\C^{/F}_E\) and \(\trbis{\C}{F}\) coincide.
	If \(E\) consists of all the 1-cells of \(\J\) then we get the \emph{pseudo slice of \(\C\)
	over \(F\)}. 
\end{parag}

\begin{parag}
 \label{parag:marked_join}
	For \(\bar{\A} = (\A, E_{\A})\) and \( \bar{\B} = (\B, E_{\B})\) two marked categories,
	it is possible to define their \ndef{marked join} \(\bar{\A} \ast \bar{\B}\).
	This is the 2-category \(\A \ast \B\) where we formally invert all
	the 2-cells of the kind
	\[
		\begin{tikzcd}[row sep = 0.7em]
																			& \varnothing \star b \ar[dd, "\varnothing \star g"] \\
			a \star \varnothing \ar[ur, "{a \star b}"] \ar[dr, "{a \star b'}"{swap, name=t}]	& 			\\
																			& b'
				\ar[Rightarrow, from=1-2, to=t, shorten <= 2.5mm, shorten >= 3.5mm, "{a \star g}"]
		\end{tikzcd}
		\quad  \text{and}\quad
		\begin{tikzcd}[row sep = 0.7em]
			a \star \varnothing \ar[rd, "{a \star b}"{name=t}] \ar[dd, "f \star \varnothing"{swap}]	&	\\
														& b \\
			a' \star \varnothing \ar[ru, "{a' \star b}"{swap}] 			&
				\ar[Rightarrow, from=3-1, to=t, shorten <= 2.5mm, shorten >= 3.5mm, "{f \star b}"]
		\end{tikzcd}
	\]
	for \(f\) an element of \(E_{\A}\) and \(g\) an element of \(E_{\B}\).
	This operation defines a functor
	\(\nCat{2}^+ \times \nCat{2}^+ \to \nCat{2}\), which agrees with the standard join
	on minimally marked 2-categories.
	The functors
	\begin{align*}
		\nCat{2}^+ 	& \to \overslice{\nCat{2}}{\A}	 & \nCat{2}^+ & \to \overslice{\nCat{2}}{\B} \\
		\bar{\B} 			& \mapsto (\bar{\A} \ast \bar{\B}, i_{\A} \colon \A \to \bar{\A} \ast \bar{\B} )	&
		\bar{\A} 			& \mapsto (\bar{\A} \ast \bar{\B}, i_{\B} \colon \B \to \bar{\A} \ast \bar{\B}),
	\end{align*}
	both have a right adjoint functor. These adjoint functors are denoted by
	\begin{align*}
		\overslice{\nCat{2}}{\A}		&\to  \nCat{2}^+ 		& \overslice{\nCat{2}}{\B} &\to \nCat{2}^+\\
		(\C, u \colon \A \to \C) 	& \mapsto \C^{u/}_E	&
		(\C, v \colon \B \to \C)	& \mapsto \C^{/v}_F
	\end{align*}
	and we will often omit the marking in the notation of the slice,
	when it is clear from the context, coherently with the notation
	of the previous paragraph.
	Notice that, by adjunction, a 1-cell \((f, \mu)\) of \(\C^{u/}_E\)
	(resp.~of \(\C^{/v}_F\)) is marked if and only if it corresponds to
	a 2-functor \(\bar{\A} \ast \Di_1^\sharp \to \C\)
	(resp.~to a 2-functor \(\Di_1^\sharp \ast \bar{\B} \to \C\)),
	where \(\Di_1^\sharp\) is the category \(\Di_1\) with its unique non-trivial 1-cell
	marked. More explicitly, this means that for every object \(a\) of \(\bar{\A}\),
	(resp.~\(b\) of \(\bar{\B}\)) the triangle
	\[
		\begin{tikzcd}[row sep = 0.7em]
																			& x \ar[dd, "f"] \\
			u(a) \ar[ur] \ar[dr, ""{swap, name=t}]	& 			\\
																			& y
				\ar[Rightarrow, from=1-2, to=t, shorten <= 2.5mm, shorten >= 3.5mm, "\mu_f", "\simeq"']
		\end{tikzcd}
		\quad \text{, resp.} \quad
		\begin{tikzcd}[row sep = 0.7em]
			x \ar[rd, ""{name=t}] \ar[dd, "f"{swap}]	&		\\
																& v(b) 	\\
			y \ar[ru, ""{swap}] 						&
				\ar[Rightarrow, from=3-1, to=t, shorten <= 3mm, shorten >= 3.5mm, "\mu_f", "\simeq"']
		\end{tikzcd},
	\]
	is filled with an invertible 2-cells \(\mu_f\).
\end{parag}

We stress the fact that we consider the output of the marked join as a bare 2-category, without marking. Even though it inherits a natural marking from the two inputs, it will play no role in what follows.

\begin{parag}
	Following the notation of the previous paragraph, let \(\bar{\A}\) be the terminal marked 2-category \(\Di_0\),
	so that a marked 2-functor \(\Di_0 \to \C\) simply corresponds to an object \(c\) of \(\C\).
	The marked 2-category \(\C^{c/}\) (resp.~\(\C^{/c}\)) has \(\overslice{\C}{c}\) (resp.~\(\trbis{\C}{c}\)) as underlying 2-category (cf.~\ref{parag:lax_slice_under_point})
	and the marked edges are the triangles filled with an invertible 2-cell.
	These marked 2-categories appear in~\cite{GarciaSternThmA} where Abell\'{a}n Garc\'{i}a and Stern denote them by \(\C_{c\swarrow}^{\dagger}\).
\end{parag}

\subsection{Bilimits and bicolimits}

In this section we recall the notions of \(E\)-bilimit of a \(2\)-functor
with source a marked 2-category \((\J, E)\).

\begin{parag}
 \label{parag:lax_cone}
	Let \(\bar{\J} = (\J, E)\) be a marked 2-category, \(\C\) a \(2\)-category and \(F \colon \J \to \C\) a \(2\)\nbd-func\-tor.
	For a given object \(x\) of \(\C\), we denote by \(\Delta x\) the constant 2-functor on \(x\)
	from \(\J\) to \(\C\).
	Given a 2-functor \(F \colon \J \to \C\), an \ndef{\(E\)-lax \(F\)-cone} \(\alpha\) with vertex \(x\),
	that we shall often denote by \((x, \alpha)\),
	is an object of \([\J, \C]_E(\Delta x, F)\). This means that for every object \(i\) of \(\J\)
	we have a \(1\)-cell \(\alpha_i \colon x \to Fi\) of \(\C\), for every \(1\)-cell \(k \colon i \to j\)
	of \(\J\) we have a \(2\)-cell
	\[
		\begin{tikzcd}
			x \ar[d, "\alpha_i"'] \ar[r, equal]	& x \ar[d, "\alpha_j"]	\\
			Fi \ar[r, "Fk"']					& Fj
			\ar[Rightarrow, from=2-1, to=1-2, shorten <=3mm, shorten >=3mm, "\alpha_k"]
		\end{tikzcd}
	\]
	invertible whenever \(k\) is in \(E\). Moreover,
	this assignment has to respect the identities and be functorial. Finally, for every \(2\)-cell \(\Gamma \colon k \to k'\)
	of \(\J\) we have the relation
	\[
		\alpha_k = \alpha_{k'} \comp_1 (F\Gamma \comp_0 \alpha_i).
	\]
	Hence, an \(E\)-lax \(F\)-cone with vertex \(x\) corresponds to an object \((x, \alpha)\)
	of the 2\nbd-cat\-e\-gory \(\mCF\),
	that is the \(E\)-lax slice of \(\C\) over \(F\).
\end{parag}

\begin{rem}\label{rem:cones}
	Notice that the assignment of a \(2\)-functor \(F' \colon \final \ast \bar{\J} \to \C\)
	corresponds to the data of the \(2\)-functor \(F\) together with
	an \(E\)-lax cone \(\lambda \colon \Delta x \to F\) over \(F\), where \(x\)
	is the image via \(F'\) of the unique object of \(\Di_0\).

	The \(2\)-category \(\final \ast \bar{\J}\) will be denoted by
	\(\mJ^\triangleleft\), and we will use \(\mJ^\triangleright\) for the dual case of cocones.
\end{rem}

\begin{parag}
	Let \(F \colon \J \to \C\) be a 2-functor, with \(\mJ = (\J, E)\) a marked 2-category, 
	and \(\alpha \colon \Delta x \to F\) an \(E\)-lax \(F\)-cone.
	For every object \(z\) of \(\C\) there is a canonical functor
	\[
		\alpha^\ast = \alpha \cdot \Delta(-) \colon \C(z, x) \to [\J, \C]_E(\Delta z, F)
	\]
	given by post-composition with \(\alpha\). More explicitly, for \(f \colon z \to x\) a \(1\)-cell of \(\C\)
	we define \(\alpha \cdot f\) to be \(\alpha_i f \colon z \to Fi\) for every \(i\) in \(\Ob(\J)\) and
	\(\alpha_k \comp_0 f\) for every \(1\)-cell \(k \colon i \to j\) of \(\J\).
	A \(2\)-cell \(\zeta \colon f \to g \colon z \to x\) defines an evident modification
	\(\alpha \cdot \zeta \colon \alpha \cdot f \to \alpha \cdot g\).
\end{parag}

\begin{define}
	For a 2-functor \(F \colon \J \to \C\), with \(\mJ = (\J, E)\) a marked 2-category,
	an \(E\)-lax \(F\)-cone \((\ell, \lambda)\)
	is said	to be an \ndef{\(E\)-bilimit cone} if the canonical functor
	\begin{equation}
	\label{eq:bilimit}
		\lambda^\ast = \lambda \cdot \Delta(-) \colon \C(x, \ell) \to [\J, \C]_E(\Delta x, F)
	\end{equation}
	is an equivalence of categories for every \(x\) in \(\Ob(\C)\).
	We also say that \((\ell, \lambda)\) is the \ndef{\(E\)-bilimit of \(F\)}.

	The dual definition gives the notion of \ndef{\(E\)-bicolimit of \(F\)}, which is just
	the \(E^\op\)-bilimit of the 2-functor \(F^\op \colon \J^\op \to \C^\op\).
\end{define}

\begin{rem}
	The notions of $E$-cones and $E$-bilimits were introduced in~\cite[Definition~2.4.3]{DescotteDubucSzyldSigmaLimits},
	and called $\sigma$-cones and $\sigma$-colimits by Descotte, Dubuc and Szyld where the marking $E$ is there denoted by $\Sigma$,
	and emerge naturally in the study of flat pseudo-functors.
\end{rem}

\begin{rem}\label{rem:weighted}
	The case of weighted \(E\)-bilimits can be recovered from that of \(E\)-bilimits,
	as proven in~\cite[Theorem~2.4.10]{DescotteDubucSzyldSigmaLimits}. In fact,
	given a 2-functor \(F \colon \J \to \C\) and a weight \(W \colon \J \to \Cat\),
	if we denote by \(p \colon \mathcal{E}l_W \to \J\) the 2-fibration associated with the 2-functor \(W\)
	(cf.~Proposition~\ref{prop:universal_fibration}), then the $E$-bilimit of \(F\)
	weighted by \(W\) is precisely the $E_W$-bilimit of \(F\circ p \colon \mathcal{E}l_W \to \C\),
	where the marking \(E_W\) is described in~\cite[Definition~2.4.8]{DescotteDubucSzyldSigmaLimits}.
\end{rem}

\subsection{Grothendieck construction for \pdftwo-categories}
\label{sec:grothendieck}

Fibrations of 2-categories were initially introduced by Hermida in his paper \cite{HermidaFib}, but his definition is not powerful enough to obtain a Grothendieck construction for such fibrations. This notion was later perfected by Buckley, who gave the correct definition in \cite{BuckleyFibred} and proved the corresponding (un)straightening theorem. In what follows we present a concise summary of the main results which are relevant for our treatment. We start by recalling the standard notion of cartesian edge and cartesian fibration for the 1-categorical setting.

\begin{define}
	Let \(p \colon E \to B\) be a functor between categories.
	An arrow \(f \colon x \to y\) of \(E\) is \(p\)-\emph{cartesian} if the square
	\[
		\begin{tikzcd}
			E(a,x) \ar[rr,"f\circ-"] \ar[d,"p_{a,x}"{swap}]& & E(a,y)\ar[d,"p_{a,y}"]\\
			B(pa,px)\ar[rr,"p(f)\circ -"]& & B(pa,py)
		\end{tikzcd}
	\]
	is a pullback square of sets for any \(a\) in \(\Ob(E)\).
	The functor \(p \colon E \to B\) is a \emph{cartesian}, or \emph{Grothendieck}, fibration if
	for any object \(e\in E\) and any arrow \(f\colon b \to p(e)\) in \(B\)
	there exists a \(p\)-cartesian arrow \(h\colon a \to  e\) in \(E\) with \(p(h)=f\).
\end{define}

Now we recall the notion of cartesian 1-cell and 2-cell for the 2-categorical setting.

\begin{define}
	Let \(p\colon \E \to \B\) be a \(2\)-functor between 2-categories. 

	\begin{itemize}
		\item A 1-cell \(f\colon x \to y\) in \(\E\) is \(p\)-\emph{cartesian} if the following square is a pullback of categories for every \(a\) in \(\Ob(\E)\):
		\[\begin{tikzcd}
		\E(a,x) \ar[rr,"f\circ-"] \ar[d,"p_{a,x}"{swap}]& & \E(a,y)\ar[d,"p_{a,y}"]\\
		\B(pa,px)\ar[rr,"p(f)\circ -"]& &\B(pa,py)
		\end{tikzcd}\]
		\item A 2-cell \(\alpha\colon f \Rightarrow g\colon x \to y\) in \(\E\) is \(p\)-cartesian if it is a \(p_{x,y}\)-cartesian 1-cell, with \(p_{x,y}\colon \E(x,y)\to \B(px,py)\).
	\end{itemize}
\end{define}

\begin{parag}
 \label{parag:cartesian_edge}
	We spell out explicitly the pullback defining a \(p\)-cartesian \(1\)-cell \(f \colon x \to y\),
	since we will need it in the following sections.
	The pullback gives existence and uniqueness property both at the level of \(1\)-cells
	and of \(2\)-cells.
	\begin{description}
		\item[\(1\)-cells] for every \(1\)-cell \(g \colon a \to y\) of \(\E\) and every
		\(1\)-cell \(k \colon pa \to px\) in \(\B\) such that \(pf\comp_0 k = pg\), there is
		a unique \(1\)-cell \(\bar k \colon a \to x\) such that \(f\bar k = g\) and \(p\bar{k} = k\).

		\item[\(2\)-cells] for every \(2\)-cell \(\alpha \colon g \to g' \colon a \to y\) of \(\E\) and every
		\(2\)-cell \(\tau \colon k \to h \colon pa \to px\) in \(\B\) such that \(pf \comp_0 \tau = p\alpha\),
		there is a unique \(2\)-cell \(\beta \colon \bar{k} \to \bar{h}\)
		such that \(f \comp_0 \beta = \alpha\) and \(p\beta = \tau\). 
	\end{description}
\end{parag}

\begin{parag}
The notion of cartesian fibration for \(2\)-categories amounts to the existence of enough cartesian lifts, as in the \(1\)-dimensional case, but it also requires an additional property: cartesian \(2\)-cells must be closed under horizontal composition. Note that, by definition and by the obvious fact that cartesian \(1\)-cells are closed under composition, cartesian \(2\)-cells are automatically closed under vertical composition.
\end{parag}

\begin{define}
	A \(2\)-functor between \(2\)-categories \(p\colon \E \to \B\) is called a \(2\)-\emph{fibration} if it satisfies the following properties:
	\begin{enumerate}
		\item\label{item:fibration-1} for every object \(e\in \E\) and every 1-cell \(f\colon b \to p(e)\) in \(\B\) there exists a \(p\)-cartesian 1-cell \(h\colon a \to  e\) in \(\E\) with \(p(h)=f\).
		\item\label{item:fibration-2} for every pair of objects \(x,y\) in \(\E\), the map \(p_{x,y}\colon \E(x,y)\to \B(px,py)\) is a Cartesian fibration of categories.
		\item\label{item:fibration-3} cartesian 2-cells are closed under horizontal composition, \ie for every triple of objects \((x,y,z)\) in \(\E\), the functor \(\circ_{x,y,z}\colon \E(y,z)\times \E(x,y)\to \E(x,z)\) sends \(p_{y,z}\times p_{x,y}\)-cartesian 1-cells to \(p_{x,z}\) ones.
	\end{enumerate} 
In this case we call \(\E\) the \ndef{total} \(2\)-category of the fibration \(p\), and \(B\) is said to be the \ndef{base} \(2\)-category.
\end{define}

\begin{rem}
	Observe that condition~\eqref{item:fibration-3} can be rephrased by requiring that given \(1\)-cells in \(\E\) of the form \(f\colon w\to x\) and \(g\colon y \to z\), the whiskering functors \(-\circ f\colon \E(x,y)\to \E(w,y)\) and \(g\circ - \colon \E(x,y)\to \E(x,z)\) preserve cartesian \(1\)-cells. This follows from the fact that horizontal composition can be obtained from vertical composition and whiskerings.
\end{rem}

\begin{define}\label{def:sub-car}
Given a cartesian 2-fibration \(p\colon \E \to \B\) we will denote by \(\E_{\car}\) the sub-2-category of \(\E\) spanned by all objects and the \(p\)-cartesian edges between them.
More precisely, \(\E_{\car}\) is the \(2\)-category whose objects
are the elements of \(\Ob(\E)\), and such that, for every pair
of objects \(x\) and \(y\),
the hom-category \(\E_{\car}\bigl(x, y\bigr)\)
is the full subcategory of \(\E\bigl(x,y\bigr)\)
spanned by \(p\)-cartesian edges. 
\end{define}

The motivation for introducing 2-fibrations is that they are a convenient way to encode functors into \(\nCat{2}\) or \(\nCat{2}_{\text{ps}}\).
More precisely, we have the following result, which is a combination of Theorems~2.2.11 and~3.3.12 in~\cite{BuckleyFibred}.

\begin{thm}
\label{S/U}
	There exists a biequivalence of \(2\)-categories between \(2\mathbf{Fib}_s(\B)_{\text{ps}}\) and \([\B^{op}_{co},\nCat{2}_{\text{ps}}]_{\text{ps}}\), the former being the \(2\)-category of \(2\)-fibrations over \(\B\) equipped with a choice of cartesian lifts compatible with composition, pseudo-functors preserving 1-cartesian and 2-cartesian cells
	and making the obvious triangle commute up-to-isomorphism and pseudo-natural transformations satisfying the obvious property,
	while the latter is the \(2\)-category of (strict) \(2\)-functors into \(\nCat{2}_{\text{ps}}\), pseudo-natural transformations and modifications.
\end{thm}

\begin{rem}
 \label{rem:S/U}
	In the proof of Theorem~\ref{S/U} (specifically in the proof of~\cite[Theorem~2.2.11]{BuckleyFibred}), Buckley exhibits a ``Grothendieck construction''
	2-functor \([\B^{op}_{co},\nCat{2}] \to 2\mathbf{Fib}_s(\B)\) (in fact, a 3-functor, but we will not need this additional structure)
	which is surjective on objects up-to-\emph{isomorphism}. In particular, such a
	biequivalence of \(2\)-categories preserves and reflects equivalences
	(cf.~\ref{par: bieq}); the equivalences of \(2\mathbf{Fib}_s(\B)\) are simply the
	cartesian preserving pseudo-functors that are biequivalences of 2-categories,
	while the equivalences of \([\B^{op}_{co},\nCat{2}]\) are the pseudo-natural transformations
	that are object-wise biequivalences of 2-categories.

	The reason why we need to introduce pseudo-functoriality and pseudo-naturality is that
	the quasi-inverse of a strict 2-functor \(F \colon \A \to \B\) that is a biequivalence
	is \emph{not} in general a strict 2-functor, but instead a pseudo-functor (see~\cite[Example~3.1]{LackModel2Cat}).
	In turn, via this biequivalence this also reflects the fact that given a 2-natural transformation
	between two 2-functors \(F, G \colon \A \to \nCat{2}\) that is object-wise a biequivalence, then
	the transformation built up using the object-wise quasi-inverses is \emph{not} a strict 2-natural transformation in general but instead a pseudo-natural one,
	even if the quasi-inverses are all strict 2-functors. This phenomenon is already evident for
	2-natural transforformations between 2-functors with values in the 2-subcategory \(\Cat\) of \(\nCat{2}\)
	that are object-wise equivalences; in fact, this will be the case of interest for us.
\end{rem}

\begin{notate}
	By considering fibrations over \(\B^{\mathrm{op}}\),  \(\B_\mathrm{co}\) or \(\B^{\mathrm{op}}_{\mathrm{co}}\) we obtain four different variants of 2-fibrations, corresponding to the four possible types of variance for \(2\)-functors \(\B \to \nCat{2}\). Instead of using four different names, we will adopt the name \ndef{fibration} for each of these cases, and specify the variance when needed.
\end{notate}

\begin{parag}
	In the same paper~\cite{BuckleyFibred}, Buckley proves several weakening of
	Theorem~\ref{S/U}, by looking at fibrations without a choice of lifts (which correspond to \emph{pseudofunctors}) and fibrations of bicategories. We content ourself with the strict case as this is the level of generality needed for this work.
\end{parag}

\begin{rem}
	If every \(2\)-cell in the total \(2\)-category \(\E\) of a fibration \(p\colon \E \to \B\) is \(p\)-cartesian,
	then the fibers \(\E_x\) are \((2,1)\)-categories for every \(x \in \B\),
	\ie \(2\)-categories where every \(2\)-cell is invertible.
	Furthermore, there is at most a \(2\)-cell between each pair of
	\(1\)-cells \(f,g\) in \(\E_x\), so we can view these fibers as \(1\)-categories,
	by quotienting out these invertible \(2\)-cells. We shall call \emph{1-fibrations}
	this class of fibrations.
\end{rem}

\begin{example}
	One of the canonical examples of a 1-fibration is given by slice projections. Given a \(2\)-category \(\B\) and an object \(x \in \B\), it is easy to see that the projection \(p \colon \trbis{\B}{x} \to \B\) from the lax slice of \(\B\) over \(x\) to \(\B\) is a 1-fibration. It corresponds to the hom-2-functor \(\B(-,x) \colon \B^{\op} \to \Cat\). In particular, every \(2\)-cell in \(\trbis{\B}{x}\) is \(p\)-cartesian, which is the reason why the associated \(2\)-functor factors through the inclusion \(\Cat \hookrightarrow \nCat{2}\).
\end{example}

We will only be using the part of Theorem~\ref{S/U} that deals with 1-fibrations,
that we now record.

\begin{cor}
	There exists a biequivalence of \(2\)-categories between \(1\mathbf{Fib}_s(\B)_{\text{ps}}\) and \([\B^{op}_{co},\Cat]_{\text{ps}}\), the former being the \(2\)-category of \(1\)-fibrations equipped with a choice of cartesian lifts compatible with composition, pseudo-functors preserving 1-cartesian cells
	and making the obvious triangle commute up-to-isomorphism and pseudo-natural transformations satisfying the obvious property,
	while the latter is the \(2\)-category of (strict) \(2\)-functors into \(\Cat\),
	pseudo-natural transformations and modifications.
\end{cor}

Given the explicit description of the biequivalence in Theorem~\ref{S/U}
provided by Buckley in his paper, we can detail
the action on objects of the biequivalence of the previous corollary.

\begin{prop}
	\label{prop:universal_fibration}
	The following facts hold true:
	\begin{enumerate}
		\item The fibration corresponding to the identity \(2\)-functor \(\Id \colon \Cat \to \Cat\) is given by the forgetful functor
		\[
			\mathbb{U}\colon \overslice{\Cat}{\Di_0} \to \Cat
		\]
		from the lax slice of \(\Cat\) under the terminal category to \(\Cat\).
		\item The fibration \(p\colon \E \to \B\) associated with a \(2\)-functor \(F\colon \B \to \Cat\) is obtained by forming the pullback displayed below.
		\[
			\begin{tikzcd}
				\E \ar[d,"p"{swap}] \ar[r] & \overslice{\Cat}{\Di_0} \ar[d,"\mathbb{U}"] \\
				\B \ar[r,"F"] & \Cat .
			\end{tikzcd}
		\]
	\end{enumerate}
\end{prop}

	\begin{example}
	 \label{example:slice_classifies_cones}
		The slice fibration \(\mCF \to \C\) classifies the 2-functor
		\[
			[\J, \C]_E(\Delta -, F) \colon \C^\op \to \Cat.
		\]
		Indeed, thanks to Proposition~\ref{prop:universal_fibration}
		we have to compute the pullback
		\[
			\begin{tikzcd}[column sep=5em]
				\E \ar[d,"p"{swap}] \ar[r] & \bigl(\overslice{\Cat}{\Di_0}\bigr)^\op \ar[d,"\mathbb{U}"] \\
				\C \ar[r,"{[\J, \C]_E(\Delta -{,} F)}"{swap}] & \Cat^\op .
			\end{tikzcd}
		\]

		An object of \(\E\) is a pair \((x, \alpha)\), where \(x\) is an object of \(\C\)
		and \(\alpha\) an object of \([\J, \C]_E(\Delta x, F)\). That is, \((x, \alpha)\)
		is a \(E\)-lax \(F\)-cone with vertex \(x\), which by what we observed in paragraph~\ref{parag:lax_cone}
		corresponds precisely to an object \((x, \alpha)\) in \(\mCF\).

		A 1-cell in \(\E\) from \((x, \alpha)\) to \((y, \beta)\) is a pair \((f, \mu)\),
		where \(f \colon x \to y\) is a 1-cell in \(\C\) and \(\mu \colon \beta \cdot \Delta f \to \alpha\)
		is a modification. For every object \(i\) in \(\J\), this corresponds
		to a commutative diagram
		\[
			\begin{tikzcd}[column sep=tiny]
					x \ar[rr, "f"] \ar[rd, "\alpha_i"{swap, name=t}]	&	& y \ar[ld, "\beta_i"]	\\
																	& Fi &
					\ar[Rightarrow, from=1-3, to=t, shorten <=3mm, shorten >=3mm, "\mu_i"']
			\end{tikzcd}
		\]
		in \(\C\), \ie the pair \((f, \mu)\) is a 1-cell of \(\mCF\).

		A 2-cell in \(\E\) from \((f, \mu)\) to \((f', \mu')\) corresponds to
		a 2-cell \(\Xi\) of \(\C\) which is a modification from \(\mu\) to \(\mu'\).
		That is, for every \(i\) in \(\J\) we have a commutative diagram
		\[
				\begin{tikzcd}[column sep=small, row sep=3.2em]
					x \ar[rr, bend left, "f", ""{swap, name=s}] \ar[rr, bend right, ""{name=t}, "f'"{above right = 0pt and 3pt}]
						\ar[rd, "\alpha_i"{swap, name=a}]	&		&
								y \ar[ld, "\beta_i"{pos=0.4}, ""{swap, pos=0.1, name=b}]	\\
															& Fi 	&
					\ar[Rightarrow, from=s, to=t, "\Xi"{swap}]
					\ar[Rightarrow, from=b, to=a, shorten <=3mm, shorten >=3mm, "\mu_i"]
				\end{tikzcd}
				=
				\begin{tikzcd}[column sep=small, row sep=3.2em]
					x \ar[rr, bend left, "f"] \ar[rd, "\alpha_i"{swap, name=a}]
						&		& y \ar[ld, "\beta_i"{pos=0.4}, ""{swap, pos=0.1, name=b}]	\\
						& Fi 	&
					\ar[Rightarrow, from=b, to=a, shorten <=3mm, shorten >=3mm, "\mu'_i"']
				\end{tikzcd}
			\]
		This is precisely a 2-cell \(\Xi \colon (f, \mu) \to (f', \mu')\) in \(\mCF\).
	\end{example}

\section{Bifinality}
\label{sec:bilimits}

In this section we motivate and provide the definition of contraction and
bifinal object and we study some of their basic properties.

\subsection{Contractions and local terminality}
\label{sec:contractions}

\begin{parag}
	Recall that a category \(C\) has a terminal object \(c\) if and only if
	the canonical projection functor \(\trbis{C}{c} \to C\) has a section
	mapping \(c\) to \(1_c\). This is straightforward, since such a section
	provides for each object \(x\) of \(C\) a morphism \(\gamma_x \colon x \to c\) and moreover
	for every morphism \(f \colon x \to c\) the triangle
	\[
		\begin{tikzcd}[column sep=tiny]
			x \ar[rd, "\gamma_x"'] \ar[rr, "f"]	&	& c \ar[ld, equal]	\\
												& c &
		\end{tikzcd}
	\]
	of \(C\) must commute.
\end{parag}

\begin{define}
\label{contractions and bifinal objects}
	Let \(\C\) be a 2-category and \(c\) an object of \(\C\).
	A \ndef{contraction on \(\C\) with center \(c\)} is section \(\gamma \colon \C \to \trbis{\C}{c}\) of
	the canonical projection \(\trbis{\C}{c} \to \C\), \ie
	a 2-functor \(\C \to \trbis{\C}{c}\) making the triangle
	\[
		\begin{tikzcd}
			\C \ar[r] \ar[rd, equal]	& \trbis{\C}{c} \ar[d]	\\
										& \C
		\end{tikzcd}
	\]
	of 2-categories commute, such that:
	\begin{itemize}
	\item
	the image of the center \(c\) is equal to \(1_c\); and 
	\item
	for every \(x\) in \(\Ob(\C)\) the \(2\)-cell \(\gamma(\gamma(x))\) (that we will denote by \(\gamma^2_{x}\)) is the identity of $\gamma_x$.
	\end{itemize}
	We will also say that the object \(c\) is \ndef{bifinal} if
	there is a contraction on \(\C\) having \(c\) as center.
\end{define}

\begin{parag}
\label{parag:contraction_explicit}
	Let us spell out the content of the definition of contraction. Suppose we have such a contraction~\(\gamma\)
	on \(\C\) with center \(c\).
	\begin{itemize}
		\item For every \(x\) in \(\Ob(\C)\) we have a \(1\)-cell \(\gamma_x \colon x \to c\), such that \(\gamma_c = 1_c\).

		\item For every \(1\)-cell \(f \colon x \to y\) of \(\C\) we have a \(2\)-cell 
			\[
				\begin{tikzcd}[column sep=tiny]
					x \ar[rr, "f"] \ar[rd, "\gamma_x"{swap, name=t}]	&	& y \ar[ld, "\gamma_y"]	\\
																	& c &
					\ar[Rightarrow, from=1-3, to=t, shorten <=3mm, shorten >=3mm, "\gamma_f"']
				\end{tikzcd}
			\]
			such that \(\gam_{\gam_x} = \Id_{\gam_x}\). 
		\item For every composable pair of 1-cells \(f,g\) the relation \(\gam_{g \comp_0 f} = \gam_{f} \comp_1 (\gam_{g} \comp_0 f)\) holds.
		\item For every \(2\)-cell \(\alpha \colon f \to g\) of \(\C\)
			the relation
			\[
				\gamma_f = \gamma_g \comp_1 (\gamma_y \comp_0 \alpha)
			\]
			holds.
	\end{itemize}
	From this explicit description it is easy to see that a contraction on \(\C\)
	with center~\(c\) also corresponds to the datum of an object \(\gamma\) in \([\C, \C](1_{\C}, \Delta c)\)
	where \(\gamma_c = 1_c\) and such that the collection of the 1-cells of the form \(\gamma_x\)
	are mapped by \(\gamma\) to identities. A contraction uniquely determines
	a set \(E_c\) of 1-cells consisting of (the identities together with) the 1-cells of the form \(\gamma_x\).
	With this at hand, it is clear that such a contraction determines an object of
	\([\C, \C]_{E_c}(1_{\C}, \Delta c)\).
\end{parag}

\begin{parag}
	The notion of contraction in the more general case
	of strict \(\infty\)-categories is due to Ara and Maltsiniotis~\cite{AraMaltsiCondE}
	and it also appears under the name of ``initial/terminal structure'' in an
	unpublished text of Burroni~\cite{BurroniOrientals}.
	Beware that what we call here contraction for simplicity is in fact
	the dual of the notion given in~\cite{AraMaltsiCondE}:
	our notion of contraction corresponds to what they call
	\ndef{dual contraction}.
	Notice also that their definition is equivalent but stated differently.
	In fact, they use the notion of (lax) Gray
	tensor product in the definition of contraction.
	This is a closed monoidal structure that was first introduced by Gray on 2-categories~\cite{GrayFormal} and was later
	generalized to strict higher categories by Al-Agl and Steiner~\cite{AlAglSteiner}
	and alternatively by Crans in his thesis~\cite{CransThesis}.

	With the Gray tensor product \(\otimes \colon \nCat{2} \times \nCat{2} \to \nCat{2}\)
	at our disposal, we can define a contraction
	on \(\C\) with center \(c\) to be a 2-functor
	\[
		\gamma \colon \Di_1 \otimes\, \C \to \C
	\]
	such that
	\begin{itemize}
		\item the restriction to \(\gamma_{|\{0\} \times \C}\) is the identity on \(\C\)
		and the restriction to \(\gamma_{|\{1\} \times \C}\) is constant on \(c\),

		\item  the image of \(\bigl(\{0 \to 1\}, \{c\}\bigr)\) is the identity of \(c\), and

		\item for every \(x\) in \(\Ob(\C)\) the image of \(\gamma_x = \gamma\bigl(\{0 \to 1\}, x\bigr)\)
		is the identity \(2\)-cell.
	\end{itemize}
\end{parag}

\begin{rem}
\label{parag:apriori-differ}
	A priori, one might want to weaken the notion of contraction, according to
	the homotopy coherent approach  
	of (weak) higher categories.
	Indeed, for a contraction \(\gamma\) on \(\C\) with center \(c\),
	we might request the 
	2-cell \(\gamma_{\gamma_x}\) to be an invertible 2-cell,
	while we require the stronger condition of being the identity of \(\gamma_x\).
	However, in our 2-categorical framework this stronger condition follows from the
	relations we impose on 2-cells. More precisely, the 2-cell
	\(\gamma_{\gamma_x} \colon \gamma_x \to \gamma_x \colon x \to c\)
	must satisfy the equality
	\[
		\gamma_{\gamma_x} = \gamma_{\gamma_x} \comp_1 \gamma_{\gamma_x}.
	\]
	As \(\gamma_{\gamma_x}\) is an invertible 2-cell, this implies that
	it must be the identity 2-cell of \(\gamma_x\).
\end{rem}

\begin{parag}
 \label{parag:uniqueness_contractions}
	If \(c\) is a bifinal object of a 2-category \(\C\), then the
	contractions we can associate to it are essentially unique.
	Indeed, let \(\alpha\) and \(\beta\) be two contractions on \(\C\)
	with center~\(c\). For every object \(x\) of \(\C\) we have two
	\(1\)-cells \(\alpha_x, \beta_x \colon x \to c\).
	Applying \(\alpha\) to \(\beta_x\) and \(\beta\) to \(\alpha_x\)
	we get two \(2\)-cells
	\[
		\begin{tikzcd}[column sep=1em, row sep=2em]
			x \ar[rr, "\beta_x"] \ar[rd, "\alpha_x"{swap, name=t}]	&	& c \ar[ld, equal]	\\
																	& c &
					\ar[Rightarrow, from=1-3, to=t, shorten <=3mm, shorten >=3mm, "\alpha_{\beta_x}"']
		\end{tikzcd}
		\quad\text{and}\quad
		\begin{tikzcd}[column sep=1em, row sep=2em]
			x \ar[rr, "\alpha_x"] \ar[rd, "\beta_x"{swap, name=t}]	&	& c \ar[ld, equal]	\\
																	& c &
					\ar[Rightarrow, from=1-3, to=t, shorten <=3mm, shorten >=3mm, "\beta_{\alpha_x}"']
		\end{tikzcd}
	\]
	of \(\C\). Finally, applying once more \(\alpha\) and \(\beta\) to \(\beta_{\alpha_x}\) and
	\(\alpha_{\beta_x}\) respectively we get the relations
	\[
		\alpha_{\beta_x} \comp_1 \beta_{\alpha_x} = 1_{\alpha_x}
		\quad \text{and} \quad
		\beta_{\alpha_x} \comp_1 \alpha_{\beta_x} = 1_{\beta_x}.
	\]
	Hence, \(\alpha_{\beta_x}\) and \(\beta_{\alpha_x}\) are inverses of one another.
	Similarly, for every \(1\)-cell \(f \colon x \to y\) of \(\C\) one gets
	\begin{equation}
	\label{eq: compatibility alpha-beta}
		\alpha_{\beta_x} \comp_1 \beta_f =
		\alpha_f \comp_1 (\alpha_{\beta_y} \comp_0 f) 
	\end{equation}
	so that \(\alpha_f\) and \(\beta_f\) can be obtained one from the other
	via compositions with invertible 2-cells.
	In particular, equation \eqref{eq: compatibility alpha-beta} provides a pseudo natural transformation
	\(\rho \colon \alpha \to \beta \colon \C \to \trbis{\C}{c}\)
	defined by \(\rho_x = (1_x, \alpha_{\beta_x})\) and for which the square
	\[
		\begin{tikzcd}
			\alpha_x \ar[r, "\alpha_f"] \ar[d, "\rho_x"']	& \alpha_y \ar[d, "\rho_y"] \\
			\beta_x \ar[r, "\beta_f"]						& \beta_y
		\end{tikzcd}
	\]
	in \(\trbis{\C}{c}\) commutes for every 1-cell \(f \colon x \to y\) of \(\C\).
\end{parag}

The above paragraph shows that if a contraction with center \(c\) exists then it is essentially unique. The following result concerns the existence of such a contraction, relating it to the existence of an appropriate supply of terminal arrows with target~\(c\).

\begin{prop}\label{prop:contraction-terminality}
Let \(\C\) be a 2-category and \(c \in \Ob(\C)\) an object. Let \(f_a\colon a \to c\) be a choice of a morphism from each object \(a \in \Ob(\C)\) to \(c\). Then the following are equivalent:
\begin{enumerate}
\item
There exists a contraction \(\gamma\) with center \(c\) such that \(\gamma_a = f_a\) for every \(a\).
\item
Every \(f_a\) is terminal in \(\C(a,c)\) and \(f_c = 1_c\). 
\end{enumerate}
\end{prop}
\begin{proof}
	Assume given a contraction \(\gamma\) with center \(c\) and \(\gamma_a = f_a\) for every \(a\), so that in particular \(f_c=1_c\) by definition. 	Then for every \(g \colon a \to c\) we have
	a 2-cell \(\gamma_g \colon g \to \gamma_a = f_a\). We claim that this is the unique 2-cell from \(g\) to \(\gamma_a\). Indeed, Given another \(2\)-cell \(\beta\colon g \to \gamma_a\), we can apply \(\gamma\) to it, resulting in the equality \(\gamma_g = \gamma_{\gamma_a}\comp_1{\beta}\). Since \(\gamma_{\gamma_a}\) is an identity 
	\(2\)-cell, we get that \(\beta\) and \(\gamma_g\) must coincide. 

	On the other hand, suppose that \(f_c=1_c\) and that each \(f_a\) is terminal in \(\C(a,c)\). We define a contraction \(\gamma\) as follows. On objects we define \(\gamma_a = f_a\) and on 1-cells
	\(g \colon a \to b\) we set \(\gamma_g\) to be the unique 2-cell \(\gamma_b \comp_0 g \to f_a\).
	We have to check that this assignment defines a contraction.
	For every pair \(x \xrightarrow{f} y \xrightarrow{g} z\) of composable 1-cells of \(\C\),
	we have two parallel 2-cells \(\gamma_{g \comp_0 f}\) and \(\gamma_f \comp_1 (\gamma_g \comp_0 f)\) from \(\gamma_z \comp_0 g \comp_0 f\) to \(\gamma_x\).
	As the target \(\gamma_x\) is a terminal object in \(\C(x, c)\),
	it follows that these 2-cells must be equal. For the same reason, for every \(2\)-cell \(\Xi \colon f \to g\) of \(\C\)
			the relation
			\[
				\gamma_f = \gamma_g \comp_1 (\gamma_y \comp_0 \Xi)
			\]
			holds.
	Finally, \(\gamma_{\gamma_a} \colon  \gamma_a \Rightarrow \gamma_a\) is a 2-cell from a terminal 1-cell to itself and is hence the identity. 
\end{proof}

\begin{rem}
	Proposition~\ref{prop:contraction-terminality} implies in particular that if there exists a contraction with center \(c\) then \(c\) is quasi-terminal.
	This fact 
	is proven in much greater generality (for strict \(\infty\)-categories)
	by Ara and Maltsiniotis in~\cite[Proposition~B.13]{AraMaltsiCondE}.
\end{rem}

\subsection{Marked bifinality and final 2-functors}

	In category theory, a functor \(u \colon J \to C\) is final if for
	every functor \(F \colon C \to D\) the colimit \(\varinjlim F\) exists
	if and only if the colimit \(\varinjlim Fu\) does and, whenever they exist,
	the canonical morphism \(\varinjlim Fu \to \varinjlim F\) is an isomorphism.
	We now give the appropriate 2-categorical generalization.

	\begin{define}
	\label{def:final-2-functor}
	A 2-functor \(u \colon (\C,E_{\C}) \to (\D,E_{\D})\) of marked 2-categories
	is said to be \ndef{final} if
	for	every 2-functor \(F \colon \D \to \E\) with \(\E\) a 2-category, and for every \(x \in \E\), the induced map
\[	[\D, \E]_{E_{\D}}(F,\Delta x) \to [\C, \E]_{E_{\C}}(Fu,\Delta x) \]
is an equivalence of categories.
	\end{define}

	\begin{rem}
	\label{rem:final-2-functor}
	It follows directly from the 2-categorical Yoneda lemma that if \(u \colon (\C,E_{\C}) \to (\D,E_{\D})\) is a final 2-functor and \(F\colon \D \to \E\) is a 2-functor then  
	\(F\) admits an \(E_{\D}\)-bicolimit if and only if \(Fu\) admits an \(E_{\C}\)-bicolimit, and when these two equivalent conditions hold, the canonical 1-cell \(\varinjlim Fu \to \varinjlim F\) is an equivalence.
	\end{rem}

\begin{rem}
The definition of final 2-functor given above is equivalent to the one studied in~\cite{GarciaSternThmA} by Abell\'{a}n Garc\'{i}a and Stern, as can be verified by comparing the explicit description of marked cocones in~\cite[Definition 6.1.3]{GarciaSternThmA} with the explicit description of \(E\)-lax cocones in Paragraph~\ref{parag:lax_cone} above.
\end{rem}

A terminal object \(c\) in a category \(C\) can be characterized
by the fact that the functor \(c \colon \Di_0 \to C\) is final.
Indeed, \(c\) is terminal if and only if
the identity functor \(1_C \colon C \to C\) admits \(c\) as its colimit.
Our goal in this section is discuss the analogous situation in the setting of final 2-functors \(\Di_0 \to \C\), for \(\C\) a 2-category. 
For this, we will need to extend the notion of contraction to the framework of marked 2-categories.

\begin{define}
	Let \((\C, M)\) be a marked 2-category and \(\gamma\) a contraction on \(\C\)
	with center \(c\). We say that \(\gamma\) is an \emph{\(M\)-contraction} if the following two conditions hold:
	\begin{enumerate}
	\item
	For every \(x \in \Ob(\C)\) the edge \(\gamma_x\) of \(\C\) is marked.
	\item
	For every marked edge \(f\) of \((\C,M)\), the 2-cell \(\gamma_f\) of \(\C\)
	is invertible.
	\end{enumerate} 
	If such an \(M\)-contraction exists, we will say that \(c\)
	is \emph{\(M\)-bifinal}. 
\end{define}

\begin{rem}\label{r:minimal-marking}
Any contraction \(\gamma\) with center \(c\) is automatically an \(M_{\gamma}\)-contraction, where \(M_{\gamma}\) denotes the collection of edges \(\gamma_a\) for \(a \in \Ob(\C)\) (together with the identities).
\end{rem}

\begin{rem}\label{rem:homotopy-sound}
	Let \((\C, M)\) be a marked 2-category, and \(\gamma,\rho\) two contractions with center
	\(c\). 
	By Paragraph~\ref{parag:uniqueness_contractions}, we have that \(\gamma\) and \(\rho\) are isomorphic as contractions. It then follows that when the collection of 1-cells \(M\) is closed under isomorphism of 1-cells then \(\gamma\) is an \(M\)-contraction if and only if \(\rho\) is an \(M\)-contraction. 
\end{rem}

\begin{prop}
 \label{prop:final_iff_contraction}
	Let \(c \in \C\) be an object and for each \(a \in \Ob(\C)\) fix a morphism \(f_a \colon a \to c\) in such a way that \(f_c = 1_c\). Then \(c\) is bifinal with a contraction \(\gamma\) satisfying \(\gamma_a = f_a\) 	if and only if the inclusion \(\{c\} \hookrightarrow (\C, M_{c})\) is a final \(2\)-functor,
	where \(M_{c} = \{f_a\}_{a \in \Ob(\C)}\). 
\end{prop}

The previous proposition is a particular case of a more general result
characterizing the final 2-functors of the form \(\{c\} \to (\C, M)\),
for a given marking \(M\) of \(\C\) containing \(M_c\).

\begin{define}
\label{def:E-final}
 	Let \((\C, M)\) be a marked 2-category and \(c\) an object of \(\C\).
 	We say that \(c\) is \emph{pre-final} if
 	\begin{enumerate}
 		\item\label{item:final1} for every \(a\) in \(\Ob(\C)\) the category \(\C(a, c)\) has a terminal object.
 		\item\label{item:final2} the identity 1-cell of \(c\) is terminal in \(\C(c, c)\). 
	\end{enumerate}
 	Moreover, we say that the pre-final object \(c\) is \emph{\(M\)-final} if 	  		
 	\begin{enumerate}[resume]
 		\item\label{item:final2a} for every \(a\) in \(\Ob(\C)\) there exists a marked edge \(a \to c\) in \((\C,M)\); 
 		\item\label{item:final3} for every marked edge \(f \colon a \to b\) in \(M\), the induced functor \(f^\ast\colon \C(b, c) \to \C(a, c)\) preserves terminal objects.
 	\end{enumerate}
 \end{define}

\begin{rem}\label{rem:terminal-is-marked}
If \(c\) is \(M\)-final in \((\C,M)\) then every marked edge \(a \to c\) with target \(c\) is terminal in \(\C(a,c)\).  
To see this, note that since \(1_c\) is terminal in \(\C(c,c)\) and pre-composition with marked edges preserves terminal objects it follows that every marked edges \(a \to c\) is terminal in \(\C(a,c)\). On the other hand, such a marked terminal edge exists in \(\C(a,c)\) by condition~\eqref{item:final2a}, and hence any other terminal edge in \(\C(a,c)\) is isomorphic to it. In particular, if the collection \(M\) is closed under isomorphisms of 1-cells then the marked edges in \(\C(a,c)\) are exactly the terminal ones. 
\end{rem}

 \begin{rem}
 	In~\cite[\S3.3]{DescotteDubucSzyldSigmaLimits} Descotte, Dubuc and Szyld work out 
	a theory of \(M\)-final 2-functor in the case where the source is an \(M\)-filtered 2-category.
 	For a marked 2-functor \(\{c\} \to (\C, M)\) the definition of loc.~cit.~is equivalent to \(c\) being \(M\)-final in the sense of Definition~\ref{def:E-final} above, at least when the collection of marked edges \(M\) is closed under composition of 1-cells (as it is assumed in loc.~cit.). Indeed, their definition adapted to the case where
 	the source is \(\Di_0\) states that the following conditions are satisfied:
 	\begin{enumerate}[label=C\arabic*, start=0]
 		\item\label{item:finalDDS-0} for every \(a\) in \(\Ob(\C)\) there exists a 1-cell
 		\(a \to c\) in $M$;
 		\item\label{item:finalDDS-1} for every object $a$ of $\C$ and every pair of parallel 1-cells
 		\(f, g \colon a \to c\) with \(g\) in \(M\), there is a 2-cell \(\alpha \colon f \to g\);
 		\item\label{item:finalDDS-2} for every object \(a\) of \(\C\), every pair of parallel 1-cells
 		\(f, g\colon a \to c\) with \(g\) in \(M\), the 2-cell \(\alpha \colon f \to g\) is unique.
 	\end{enumerate}
 	These conditions all hold when \(c\) is \(M\)-final by Remark~\ref{rem:terminal-is-marked}. On the other hand, the combination of the above conditions is clearly equivalent to the statement that for each \(a \in \Ob(\C)\) there exist marked edges \(a \to c\), and that these are all terminal in their respective hom categories. In particular, conditions~\eqref{item:final1} and~\eqref{item:final2a} hold in this case. In addition, since all identities are marked~\eqref{item:final2} is implied, and when the collection of marked edges is closed under composition we also get that~\eqref{item:final3} holds. 
 \end{rem}

The following result shows that the notions of finality for an object
that we have introduced so far coincide.

\begin{prop}
	\label{prop:E-final-iff-E-contraction}
	Let \((\C, M)\) be a marked 2-category, \(c\) an object of \(\C\), and \(f_a\colon a \to c\) a choice of a marked 1-cell for every \(a \in \Ob(\C)\) such that \(f_c = 1_c\). 
	Then the following statements are equivalent:
	\begin{enumerate}[label=(\roman*)]
		\item \(c\) is \(M\)-final; 
		\item \(c\) is \(M\)-bifinal with contraction \(\gamma\) such that \(\gamma_a = f_a\);
		\item the inclusion \(\{c\} \hookrightarrow (\C, M)\) is final. 	\end{enumerate}
\end{prop}

\begin{warning}
In Proposition~\ref{prop:E-final-iff-E-contraction} we implicitly assume in advance that a marked edge \(a \to c\) exists for every \(a \in \Ob(\C)\). This assumption automatically holds if \(c\) is \(M\)-final or \(M\)-bifinal, but not necessarily if one only assumes that the inclusion \(\{c\} \hookrightarrow (\C, M)\) is final. In particular, while Proposition~\ref{prop:E-final-iff-E-contraction} shows that for a given object \(c \in \Ob(\C)\) the properties of being \(M\)-final and \(M\)-bifinal are equivalent, these conditions are only shown to imply the finality of the inclusion \(\{c\} \hookrightarrow (\C, M)\), and to be implied by it if \((\C,M)\) possesses sufficiently many marked edges. 
\end{warning}

\begin{proof}[Proof of Proposition~\ref{prop:E-final-iff-E-contraction}]
 	\((i) \Rightarrow (ii)\).
 	Assume that \(c\) is \(M\)-final. By Remark~\ref{rem:terminal-is-marked} we have that \(f_a\) is terminal in \(\C(a,c)\) for every \(a \in \Ob(\C)\).  
 	By Proposition~\ref{prop:contraction-terminality} there exists a contraction \(\gamma\) such that \(\gamma_a = f_a\) for every \(a \in \Ob(\C)\). 
 	To see that \(\gamma\) is an \(M\)-contraction we now note that \(\gamma_a\) is marked by construction, 
	and for \(g \colon a \to b\) a marked 1-cell, condition~\eqref{item:final3} entails
	that \(\gamma_g \colon \gamma_b \comp_0 g \to \gamma_a\) is a 2-cell between two terminal 1-cells, and is hence invertible.
	This shows that \(\gamma\) is an \(M\)-contraction.

	\((ii) \Rightarrow (iii)\).
	Let \(c\) be an \(M\)-bifinal object and \(\gamma\) an \(M\)-contraction
	with center \(c\) such that \(\gamma_a = f_a\) for every \(a \in \Ob(\C)\). 
	We need to show that for any 2-functor \(F \colon \C \to \D\), and any \(d \in \D\), the evaluation at \(c\) map
	\[
		\bullet_c \colon [\C, \D]_{M}(F, \Delta d) \to \D(Fc, d),
	\]
	is an equivalence of categories. We construct an explicit inverse 
	\[
		\gamma^{\ast} \colon \D(Fc, d) \to [\C, \D]_{M}(F, \Delta d)
	\]
	to this functor by means of post-composition with \(\gamma\).
	More precisely, given \(f \colon Fc \to d\), we can consider the family of \(1\)-cells
	\[
	  \bigl(fF(\gamma_x)\colon Fx \to d\bigr)_{x \in \Ob(\C)},
	\]
	and that of \(2\)-cells
	\[
		\bigl(fF(\gamma_h) \colon fF(\gamma_x) \to f F(\gamma_y) F(h)\bigr)_{h \in \C( x ,y)}.
	\]
	It is clear that these data all organize into an \(M\)-lax natural transformation
	\(\gamma^{\ast}(f)\) from \(F\) to \(\Delta d\) and one can similarly
	obtain a modification \(\gamma^{\ast}(\alpha)\colon \gamma^{\ast}(f) \to \gamma^{\ast}(g)\)
	from a \(2\)-cell \(\alpha \colon f \to g\). We now claim that \(\gamma^{\ast}\) is inverse to \(\bullet_c\).
	Indeed, on the one hand, the composition \(\bullet_c \circ \gamma^\ast\)
	is the identity on \(\D(Fc, d)\). In the other direction, 
	we need to show that the \(M\)-lax transformations \(\rho\) and \(\gamma^\ast(\rho_c)\)
	are isomorphic objects in \([\C, \D]_{M}(F, \Delta d)\),
	naturally in \(\rho\).
	For \(x\) an object of \(\C\), we get \(\gamma^\ast(\rho_c)_x = \rho_c F\gamma_x\).
	This gives us a triangle
	\[
		\begin{tikzcd}[column sep=small]
			Fx \ar[rr, "F\gamma_x"] \ar[rd, "\rho_x"{swap, name=t}]	&	& Fc \ar[ld, "\rho_c"]	\\
																	& d &
				\ar[Rightarrow, from=1-3, to=t, shorten <= 4mm, shorten >=4mm, "\beta_x"{swap}]
		\end{tikzcd}
	\]
	in \(\D\), where the 2-cell \(\beta_x\) is \(\rho_{\gamma_x}\).
	As \(\gamma_x\) is in \(M\) 
	the 2-cell \(\rho_{\gamma_x}\) is 
	invertible by assumption. We need to check that the collection of
	the 2-cells \(\beta_x\), for all objects \(x\) of \(\C\), organizes
	into an invertible modification between \(\rho\) and \(\gamma^\ast(\rho_c)\).
	For this we must check that given any 1-cell \(f \colon x \to y\) of \(\C\),
	the equation
	\[
		\rho_f \comp_1 (\rho_{\gamma_y} \comp_0 Ff) =
		\rho_{\gamma_x} \comp_1 (\rho_c \comp_0 F\gamma_f).
	\]
	holds. This is precisely the relation satisfied by the action of \(\rho\) on
	the 2-cell
	\[
		\begin{tikzcd}[column sep=small]
			x \ar[rr, "f"] \ar[rd, "\gamma_x"{swap, name=t}]	&	& y \ar[ld, "\gamma_y"]	\\
																& c &
				\ar[Rightarrow, from=1-3, to=t, shorten <= 3.5mm, shorten >=3.5mm, "\gamma_f"{swap}]
		\end{tikzcd}
	\]
	of \(\C\). Hence, the functor \(\bullet_c\) is an equivalence.

	\((iii) \Rightarrow (i)\).
	Suppose that \(c \colon \Di_0 \to (\C, M)\) is a final 2-functor. The object \(c\)
	with the cocone \(1_c\) is clearly the bicolimit of the 2-functor \(c\).
	In light of Remark~\ref{rem:final-2-functor} it then follows that the identity 2-functor \(1_{\C} \colon (\C, M) \to \C\) admits an \(M\)-bicolimit cocone \(\gamma\) 
	such that \(\gamma_c = 1_c\).
Now by assumption every object \(a \in \Ob(\C)\) admits a marked edge \(f_a\colon a \to c\). Since \(\gamma\) is an \(M\)-lax cocone the 2-cell \(\gamma_f\colon f_a \Rightarrow \gamma_a\)
is invertible, and hence \(f_a\) and \(\gamma_a\) are isomorphic 1-cells in \(\C(a,c)\). We now observe that for a general morphism \(g\colon a \to c\) the condition that \(\gamma_g\) is invertible is closed under isomorphism in \(\C(a,c)\). It then follows that \(\gamma_{\gamma_a}\) is invertible as well, so that \(\gamma\) constitutes a contraction with center \(c\). By Proposition~\ref{prop:contraction-terminality} we now get that each \(\gamma_a\) is terminal in \(\C(a,c)\), and since \(\gamma_c = \Id_c\) this means that that \(c\) is pre-final. In addition, Condition~\eqref{item:final2a} holds by assumption, and since \(\gamma\) is an \(M\)-lax cocone we have that for every marked edge \(g\colon a \to b\) the 2-cell \(\gamma_g\) is invertible, and so \(\gamma_a \cong g \comp_0 \gamma_b\). We thus have that pre-composition \(\C(b,c) \to \C(a,c)\) with any marked edge \(g\colon a \to b\) preserves terminal objects, which shows~\eqref{item:final3}.
\end{proof}

\begin{proof}[Proof of Proposition~\ref{prop:final_iff_contraction}]
	This is a particular case of the implication \((ii) \Leftrightarrow (iii)\) of Proposition~\ref{prop:E-final-iff-E-contraction} applied to \(M_{c}\). 
\end{proof}

\section{Slice fibrations}
In this section we further study the slice fibrations \(p\colon \mCF \to \C\) associated to a marked 2-category \(\mJ = (\J, E)\) and a \(2\)-functor \(F \colon \J \to \C\) as in Paragraph~\ref{parag:marked-slice}. We begin in \S\ref{sec:car-edges} by giving an explicit description of the \(p\)-cartesian edges in \(\mCF\), and show that they coincide with the marked edges with respect to the marking of Paragraph~\ref{parag:marked_join}. In \S\ref{sec:fibrations} we then focus on the particular case where \(\J=\Di_0\), so that \(\mCF = \trbis{\C}{F(\ast)} \) is a representable fibration, and show that in this case the object \(1_{F(\ast)}\) is the center of a contraction relative to the collection of \(p\)-cartesian edges. Finally, in \S\ref{sec:special_cones} we construct a modified 2-category of cones which projects to \(\mCF\), and show that this projection is a biequivalence 
if and only if \(F\) admits an \(E\)-bilimit.

\subsection{Cartesian edges in slice fibrations}
\label{sec:car-edges}
	\begin{parag}
		Let us fix a marked 2-category \(\mJ = (\J, E)\), a \(2\)-functor \(F \colon \J \to \C\),
		and consider the
		associated projection \(2\)-functor \(p \colon \mCF \to \C\).
		We wish to describe the \emph{\(p\)-cartesian morphisms} of \(\mCF\).
		Consider an object \((x, \lambda) \colon \final \ast \mJ \to \C\) of \(\mCF\),
		that we can represent by
		\[
			\begin{tikzcd}[column sep=tiny]
									& x \ar[ld, "\lambda_i"'] \ar[rd, "\lambda_j"{name=t}]	&	\\
				Fi \ar[rr, "Fh"']	&														& Fj
				\ar[Rightarrow, from=2-1, to=t, shorten <= 2mm, shorten >= 3mm, "\lambda_h"{swap}]
			\end{tikzcd}
			\quad  \quad
			\text{for $h \colon i \to j$ in $\J$,}
		\]
		with \(\lambda_h\) invertible whenever \(h\) is in \(E\).
		We claim that every \(1\)-cell \((f, \sigma) \colon \Di_1 \ast \mJ \to \C\) from \((z, \alpha)\) to \((x, \lambda)\):
		\[
			\begin{tikzcd}[column sep=tiny]
				z \ar[rr, "f"] \ar[rd, "\alpha_j"'{name=t}]	&		& x \ar[ld, "\lambda_j"] \\
															&	F_j &
				\ar[Rightarrow, from=1-3, to=t, shorten <=3mm, shorten >=3mm, "\sigma_j"{swap}, "\cong"]
			\end{tikzcd}
		\]
		such that the \(2\)-cell \(\sigma_j\) of \(\C\) is invertible for every object \(j\) of \(\J\),
		is a \(p\)-cartesian lift of \(f \colon z \to x\).
		In the notations of Paragraph~\ref{parag:marked_join}, we are claiming that \((f, \sigma)\) is a \(p\)-cartesian
		lift whenever it is represented by a 2-functor \(\Di_1^\sharp \ast \mJ \to \C\).
		One possible such choice is given by the precomposition of \((x, \lambda)\) with \(f\),
		\ie where \(\sigma_j\) is the identity of \(\lambda_j f\) for all \(j\) in \(\Ob(\J)\).
		
		Indeed, let \((f', \mu) \colon (z', \alpha') \to (x, \lambda)\) be a \(1\)-cell of \(\mCF\)
		and \(g \colon z' \to z\) a \(1\)-cell of \(\C\) such that \(fg = f'\).
		We wish to find an edge \((g, \bar{\mu}) \colon (z', \alpha')\to (z, \alpha)\)
		such that
		\begin{equation}
		\label{eq:cartesian_edge_slice_CF}
			(f, \sigma) \comp_0 (g, \bar{\mu}) = (f', \mu).
		\end{equation}
		Now \((f, \sigma) \comp_0 (g, \bar{\mu}) = (fg, \bar{\mu} \bullet \sigma)\),
		where we set \((\bar{\mu} \bullet \sigma)_j = \bar{\mu}_j \comp_1 (\sigma_j \comp_0 g) \)
		so that necessarily \(f \comp_0 g = f'\) and \(\mu = \bar{\mu} \bullet \sigma\).
		As \(p(g, \bar{\mu}) = g\), we must have \((g, \bar{\mu}) = (g, \mu \bullet \sigma^{-1})\),
		where we set
		\[
			(\mu \bullet \sigma^{-1})_j = \mu_j \comp_1 (\sigma^{-1}_j \comp_0 g)
		\]
		for all objects \(j\) in \(\J\). Said otherwise, if such a 1-cell \((g, \bar{\mu})\)
		of \(\mCF\) exists, then for any object \(j\) of \(\J\) the 2-cell
		\(\bar{\mu}\) of \(\C\) must be equal to \(\mu_j \comp_1 (\sigma^{-1}_j \comp_0 g)\).
		This shows that if \((g, \bar{\mu})\) satisfying \eqref{eq:cartesian_edge_slice_CF} 
		exists, then it is unique.
		As for its existence, it is easy to check that the following diagram
		\begin{nscenter}
		 \begin{tikzpicture}
			\square{
				/square/label/.cd,
					0={$z'$}, 1={$z$}, 2={$Fi$}, 3={$Fj$},
					01={$g$}, 12={$\alpha_i$}, 23={$Fh$}, 02={$\alpha'_i$},
					13={$\alpha_j$}, 03={$\alpha'_j$},
					012={$\bar{\mu}_i$}, 023={$\alpha'_h$},
					013={$\bar{\mu}_j$}, 123={$\alpha_h$},
				/square/arrowstyle/.cd,
					012={Leftarrow}, 023={Leftarrow},
					013={Leftarrow}, 123={Leftarrow},
					0123 = {equal}
			}
		 \end{tikzpicture}
		\end{nscenter}
		in \(\C\) commutes for every edge \(h \colon i \to j\) in \(\J\), so that \(\bar{\mu} \colon \alpha' \to \alpha\) defines
		a \(1\)-cell in \(\mCF\).
		This concludes the existence and uniqueness property for \(1\)-cells.

		Let \(\Xi \colon (f', \mu) \to (f'', \mu')\) be a \(2\)-cell of \(\mCF\).
		In particular, for any object \(i\) of \(\J\) we have a commutative diagram
		\[
				\begin{tikzcd}[column sep=small, row sep=3.2em]
					{z'} \ar[rr, bend left, "{f'}", ""{swap, name=s}]
						\ar[rr, bend right, ""{name=t}, "{f''}"{above right = 0pt and 3pt}]
						\ar[rd, "{\alpha'_i}"{swap, name=a}]	&		&
									x \ar[ld, "\lambda_i"{pos=0.4}, ""{swap, pos=0.4, name=b}]	\\
																& Fi 	&
					\ar[Rightarrow, from=s, to=t, "\Xi"{swap}]
					\ar[Rightarrow, from=b, to=a, shorten <=4mm, shorten >=4mm, "{\mu'_i}"]
				\end{tikzcd}
				=
				\begin{tikzcd}[column sep=small, row sep=3.2em]
					{z'} \ar[rr, bend left, "{f'}"] \ar[rd, "{\alpha'_i}"{swap, name=a}]
						&	& x \ar[ld, "\lambda_i"{pos=0.4}, ""{swap, pos=0.1, name=b}]	\\
						& Fi &
					\ar[Rightarrow, from=b, to=a, shorten <=4mm, shorten >=4mm, "\mu_i"']
				\end{tikzcd}
			\]
		in \(\C\). Fix \(\tau \colon g \to g'\) a \(2\)-cell of \(\C\)
		such that \(p(\Xi) = \Xi = f \comp_0 p(\tau)\).
		We know from the argument above that there are unique lifts
		\[
			(g, \bar{\mu}) \colon (z', \alpha') \to (z, \alpha)
			\quad , \quad
			(g', \ovl{\mu'}) \colon (z', \alpha') \to (z, \alpha)
		\]
		such that \((f, \sigma) \comp_0 (g, \bar\mu) = (f', \mu)\)
		and \((f, \sigma) \comp_0 (g', \overline{\mu'}) = (f'', \mu')\).
		We wish to find a unique \(2\)-cell \(\bar{\Xi} \colon (g, \bar\mu) \to (g', \overline{\mu'})\)
		of \(\mCF\)
		such that \(p(\bar{\Xi}) = \tau\) and
		\[
			1_{(f, \sigma)} \comp_0 \tau = \Xi.
		\]
		in \(\mCF\). The condition \(p(\bar{\Xi}) = \tau\) forces the uniqueness. Since \(1_{(f, \sigma)} \comp_0 \tau = 1_f \comp_0 \tau = \Xi\),
		the existence part follows once we show that~\(\tau\) is a \(2\)-cell
		from \((g, \bar\mu)\) to \((g', \overline{\mu'})\). For every \(i\) in \(\Ob(\J)\) we have
		\[
		 \begin{split}
		 	\ovl{\mu_i'} \comp_1 (\alpha_i \comp_0 \tau)
		 		& = \mu_i' \comp_1 (\sigma^{-1}_i \comp_0 g) \comp_1 (\alpha_i \comp_0 \tau)				\\
		 		& = \mu_i' \comp_1 (\lambda_i \comp_0 f \comp_0 \tau) \comp_1 (\sigma^{-1}_i \comp_0 g')  \\
		 		& = \mu_i' \comp_1 (\lambda_i \comp_0 \Xi) \comp_1  (\sigma^{-1}_i \comp_0 g')  		\\
		 		& = \mu_i \comp_1 (\sigma^{-1}_i \comp_0 g') = \bar{\mu}_i,
		 \end{split}
		\]
		where the second equality is given by the interchange law, the third and fourth
		by the assumptions and the first and the last one are the definitions of
		\(\ovl{\mu'}\) and \(\bar{\mu}\), respectively.

		We have thus shown that the 1-cell \((f, \sigma)\) satisfies the properties of \(p\)-cartesian
		1-cell detailed in paragraph~\ref{parag:cartesian_edge}.
	\end{parag}

	\begin{prop}\label{prop:cartesian_edges}
		The \(p\)-cartesian edges \((f, \sigma) \colon (z, \beta) \to (x, \alpha)\) of \(\mCF\) are all of the form
		\[
		 \begin{tikzcd}[column sep=tiny]
			 z \ar[rr, "f"] \ar[rd, "{\beta_i}"{swap, name=a}] 	&	& x \ar[ld, "\alpha_i"] \\
													& Fi &
			 \ar[Rightarrow, from=1-3, to=a, shorten <=3mm, shorten >=3mm, "\sigma_i"{swap}, "\simeq"]
		 \end{tikzcd}
		\]
		with \(\sigma_i\) an invertible \(2\)-cell of \(\C\) for all \(i\) in \(\Ob(\J)\). In particular, the \(p\)-cartesian edges of \(\mCF\)
		are precisely its marked edges, as detailed in Paragraph~\ref{parag:marked_join}.
	\end{prop}

	\begin{proof}
		We have shown in the paragraphs above that the edges of \(\mCF\)
		of this form are indeed \(p\)-cartesian.
		If \((g, \mu) \colon (z, \beta) \to (x, \alpha)\) is
		a \(p\)-cartesian edge, then there exists a unique \(1\)-cell
		\((1_z, \mu') \colon (z, \alpha \cdot g) \to (z, \beta)\)
		such that
		\[
			(g, \mu) \comp_0 (1_z, \mu') = (g, 1_{\alpha \cdot g}),
		\]
		which is equivalent to the condition \(\mu_i' \comp_0 \mu_i = 1_{\alpha_i\comp_0 g}\)
		for all \(i\) in \(\Ob(\J)\).
		But \(\mu'_i\) is an invertible \(2\)-cell of \(\C\),
		since by~\cite[Proposition~2.1.4(2)]{BuckleyFibred} the \(1\)-cell
		\((1_z, \mu')\) is an isomorphism of \(\CF\) and so in \(\mCF\).
		Hence, \(\mu_i\) is an invertible \(2\)-cell of \(\C\) for all objects \(i\) of \(\J\). 
	\end{proof}

	\begin{parag}
	 \label{parag:CFcart}
		Proposition~\ref{prop:cartesian_edges} yields an explicit description of the sub-2-category \(\mCFcart\) of \(\mCF\) spanned by all objects and the \(p\)-cartesian edges between them (see Definition~\ref{def:sub-car}).
		Indeed, the \(1\)-cells of \(\mCFcart\) are given by
		\[
		 \begin{tikzcd}[column sep=tiny]
			 z \ar[rr, "{g}"] \ar[rd, "{\alpha}"{swap, name=a}] 	&	& z' \ar[ld, "\alpha'"{pos=0.35}] \\
													& Fi &
			 \ar[Rightarrow, from=1-3, to=a, shorten <=3mm, shorten >=3mm, "\sigma_i"{swap}, "\simeq"]
		 \end{tikzcd}
		\]
		with \(\sigma_i\) an invertible \(2\)-cell of \(\C\) for every \(i\) in \(\Ob(\J)\).
	\end{parag}

\subsection{Representable fibrations}
\label{sec:fibrations}

In this section we focus on the particular case of slice fibrations where \(J=\Di_0\), that is, a 2-functor \(F\colon J \to \C\) simply corresponds to an object \(x \in \C\). In this case \(\trbis{\C}{x}\) is equipped with a designated object \(1_x\), whose finality properties we wish to understand. 
We start by recalling the classical \(1\)-categorical scenario.

\smallskip

\textbf{The \texorpdfstring{$1$}{1}-categorical case.}
Let \(p \colon E \to B\) be a cartesian fibration of 1-categories and suppose that
it is represented by an object \(x\) of \(B\).
By this we mean that the functor
\[
	B(-, x) \colon B^\op \to \Set
\]
is classified by \(p \colon E \to B\), \ie we have
a pullback square
\[
	\begin{tikzcd}
		E \ar[r] \ar[d, "p"'] 	& \bigl(\overslice{\Set}{\{\ast\}}\bigr)^\op \ar[d] \\
		B \ar[r] 				& \Set^\op
	\end{tikzcd}
\]
of categories.
In this case, \(E\) is isomorphic to \(\trbis{B}{x}\), and we may consider the object \(\bar{x} \in \Ob(E)\) corresponding under this isomorphism to \(1_x \in \trbis{B}{x}\). The object \(\bar{x} \in \Ob(E)\) can then be internally characterized inside \(E\) as being a \emph{terminal object}.

\smallskip

\textbf{The \pdftwo-categorical case.}
Let \(p \colon \E \to \B\) be a cartesian \(2\)-fibration
and suppose that it is represented by an object \(x\) of \(\B\).
That is, the \(2\)-functor
\[
	\B(-, x) \colon \B^\op \to \Cat
\]
is classified by \(p \colon \E \to \B\), meaning that by Proposition~\ref{prop:universal_fibration}
we have a pullback square
\[
	\begin{tikzcd}
		\E \ar[r] \ar[d, "p"'] 	& \bigl(\overslice{\Cat}{\final}\bigr)^\op \ar[d] \\
		\B \ar[r] 				& \Cat^\op
	\end{tikzcd}
\]
of \(2\)-categories. In this case, \(\E\) is isomorphic to \(\trbis{\B}{x}\), and we may consider the object \(\bar{x} \in \Ob(\E)\) corresponding under this isomorphism to \(1_x \in \trbis{\B}{x}\). 

\medskip

	Unlike the 1-categorical case, the object \(\bar{x}\) of \(\E\) 
	is \emph{not} biterminal in general (\(\bar{x}\) biterminal means that for every object
	\(z\) in \(\E\) the category \(\E(z, \bar{x})\) is equivalent to \(\Di_0\)). Instead, we have the following:
\begin{prop}\label{prop:pre-final}
Let \(p\colon\E \to \B\) be a cartesian 2-fibration represented by \(x \in \Ob(\B)\) and let \(\bar{x} \in \E\) be the associated lift of \(x\). Let \(M_{\car}\) denote the collection of \(p\)-cartesian edges. Then \(\bar{x}\) is \(M_{\car}\)-final in \(\E\).
\end{prop}
\begin{proof}
Without loss of generality we may assume that \(\E=\trbis{\B}{x}\) and \(\bar{x} = 1_{x}\).
	The objects of the category \(\trbis{\B}{x}(\alpha, 1_x)\) are of the form \((\beta, \zeta)\)
	\[
		\begin{tikzcd}[column sep=tiny]
			y \ar[rr, "\beta"] \ar[rd, "\alpha"{swap, name=a}] 	&	& x \ar[ld, equal, "1_x"] \\
													& x &
			\ar[Rightarrow, from=1-3, to=a, shorten <=3mm, shorten >=2mm, "\zeta"{swap}]
		\end{tikzcd}
		\quad , \quad\text{that is}\quad
		\begin{tikzcd}
			y \ar[r, bend left, "\beta"{name=b}] \ar[r, bend right, "\alpha"{swap, name=a}] & x
			\ar[Rightarrow, from=b, to=a, shorten <=1mm, shorten >= 1mm, "\zeta"]
		\end{tikzcd} ,
	\]
	and its morphisms \(\Xi \colon (\beta, \zeta) \to (\beta', \zeta')\) are of the form
	\[
				\begin{tikzcd}[column sep=small, row sep=3.2em]
					y \ar[rr, bend left, "\beta", ""{swap, name=s}]
						\ar[rr, bend right, ""{name=t}, "{\beta'}"{above right = 0pt and 3pt}]
						\ar[rd, "\alpha"{swap, name=a}]	&	& x \ar[ld, equal, "1_x"{pos=0.4}, ""{swap, pos=0.1, name=b}]	\\
												& x &
					\ar[Rightarrow, from=s, to=t, "\Xi"{swap}]
					\ar[Rightarrow, from=b, to=a, shorten <=3mm, shorten >=3mm, "{\zeta'}"]
				\end{tikzcd}
				=
				\begin{tikzcd}[column sep=small, row sep=3.2em]
					y \ar[rr, bend left, "\beta"] \ar[rd, "\alpha"{swap, name=a}]
						&	& x \ar[ld, equal, "1_x"{pos=0.4}, ""{swap, pos=0.1, name=b}]	\\
						& x &
					\ar[Rightarrow, from=b, to=a, shorten <=3mm, shorten >=3mm, "\zeta"']
				\end{tikzcd}
		\quad , \quad\text{that is}\quad
		\zeta' \comp_1 \Xi = \zeta.
	\]
	We claim that the object \((\alpha, 1_\alpha)\) of \(\trbis{\B}{x}(\alpha, 1_x)\) is terminal.
	Indeed, for every other object \((\beta, \zeta)\) in \(\trbis{\B}{x}(\alpha, 1_x)\) we have that
	\(\zeta \colon (\beta, \zeta) \to (\alpha, 1_\alpha)\) is a morphism in \(\trbis{\B}{x}(\alpha, 1_x)\).
	Moreover, it is clearly the unique going from \((\beta, \zeta)\) to \((\alpha, 1_\alpha)\),
	as any such morphism \(\Xi\) must satisfy
	\(1_x \comp_1 \Xi = \Xi = \zeta\). In addition, taking \(\alp = 1_{x}\) we also get that the identity on \(1_{x}\) is terminal in \(\trbis{\B}{x}(1_x,1_x)\). 
	
	We have thus established that \(1_x\) is a pre-final object of \(\trbis{\B}{x}\). To show that it is \(M_{\car}\)-final we now point out that by Proposition~\ref{prop:cartesian_edges} each of the arrows \((\alpha, 1_\alpha)\colon \alpha \to 1_x\) just contracted are \(p\)-cartesian, so that~\eqref{item:final2a} holds, and that any other \(p\)-cartesian edge \(\alpha \to 1_x\) is of the form
	\[
		\begin{tikzcd}[column sep=tiny]
			y \ar[rr, "\beta"] \ar[rd, "\alpha"{swap, name=a}] 	&	& x \ar[ld, equal, "1_x"] \\
													& x &
			\ar[Rightarrow, from=1-3, to=a, shorten <=3mm, shorten >=2mm, "\simeq"{swap}]
		\end{tikzcd}
	\]
	and hence isomorphic in \(\trbis{\B}{x}(\alpha,1_x)\) to \((\alpha,1_\alpha)\). This means that all \(p\)-cartesian edges \(\alpha \to 1_x\) are in fact terminal, and since the collection of \(p\)-cartesian edges is also closed under isomorphism we have that these are exactly the terminal edges in \(\trbis{\B}{x}(\alpha,1_x)\). Condition~\eqref{item:final3} is then a consequence of the fact that \(p\)-cartesian edges are closed under composition.
	\end{proof}

	\begin{parag}\label{paragr:Bx_contraction}
	Combining Proposition~\ref{prop:pre-final} and Proposition~\ref{prop:contraction-terminality}
	we obtain a contraction~\(\gamma\) on \(\trbis{\B}{x}\) with center \(1_x\).
	This contraction can be explicitly described as follows:
		\begin{itemize}
			\item To a given object \((y, \alpha)\) we assign the \(1\)-cell \((\alpha, 1_\alpha)\),
			that we can depict by
				\[
					\begin{tikzcd}[column sep=tiny]
						y \ar[rr, "\alpha"] \ar[rd, "\alpha"{swap, name=a}]	&		& x \ar[ld, equal] \\
																		& x 	&
						\ar[phantom, from=a, to=1-3, "="{description}]
					\end{tikzcd}
				\]
			\item For every \(1\)-cell \((f, \zeta) \colon (y, \alpha) \to (z, \beta)\)
			we have to provide a \(2\)-cell going from \(\gamma(z, \beta) \comp_0 (f, \zeta)\) to \(\gamma(y, \alpha)\).
			By definition, \(\gamma(z, \beta) \comp_0 (f, \zeta) = (\beta f, \zeta)\) and
			\(\gamma(y, \alpha) = (\alpha, 1_\alpha)\). We set this \(2\)-cell \(\gamma(f, \zeta)\) to be \(\zeta\). This assignment can be depicted as follows:
			\begin{nscenter}
		 		\begin{tikzpicture}
					\square{
						/square/label/.cd,
							0={$y$}, 1={$z$}, 2={$x$}, 3={$x$},
							01={$f$}, 12={$\beta$}, 23={$$}, 02={$\alpha$},
							13={$\beta$}, 03={$\alpha$},
							012={${\zeta}$}, 023={$$},
							013={${\zeta}$}, 123={$$},
						/square/arrowstyle/.cd,
							23 = {equal},
							012={Leftarrow}, 023={equal},
							013={Leftarrow}, 123={equal},
							0123 = {equal}
					}
			 	\end{tikzpicture}
			\end{nscenter}
		\end{itemize}
		The assignment on \(2\)-cells is given by the identity on the 2-cells of \(\B\)
		joint with a coherence, which is trivially satisfied.
		Notice that here we are identifying the 2-cells of \(\trbis{\B}{x}\)
		(and therefore \(\trbis{(\trbis{\B}{x})}{1_x}\), too)
		with some 2-cells of \(\B\) and we are saying that the assignment of 2-cells
		of the contraction is the identity, but with the 2-cells seen in different 2-categories.
	\end{parag}

	\begin{parag}
		We give an explicit description of the \(2\)-category \(\Bxcart\) (see Definition~\ref{def:sub-car}).
		The objects are the elements of \(\Ob(\Bx)\).
		For every pair of objects \(\alpha \colon y \to x\) and \(\alpha' \colon y' \to x\),
		the hom-category \(\Bxcart\bigl((y, \alpha), (y', \alpha')\bigr)\)
		is the full subcategory of \(\Bx((y, \alpha), (y', \alpha')\bigr)\)
		spanned by \(p\)-cartesian edges. By Proposition~\ref{prop:cartesian_edges},
		these are the \(1\)-cells \((\beta, \sigma)\) of \(\Bx\) of the form
		\[
		 \begin{tikzcd}[column sep=tiny]
			 z \ar[rr, "{\beta}"] \ar[rd, "{\alpha'}"{swap, pos=0.4, name=a}] 	&	& y \ar[ld, "\alpha"] \\
													& x &
			 \ar[Rightarrow, from=1-3, to=a, shorten <=3mm, shorten >=3mm, "\sigma"{swap}, "\simeq"]
		 \end{tikzcd}
		\]
		with \(\sigma\) an invertible \(2\)-cell of \(\B\).
	\end{parag}

	\begin{prop}
	 \label{prop:biterminal_Bx}
		The object \((x, 1_x)\) is biterminal in \(\Bxcart\).
	\end{prop}

	\begin{proof}
		Consider any object \(\alpha \colon y \to x\) of \(\Bxcart\).
		We wish to show that the category \(\Bxcart\bigl((y, \alpha), (x, 1_x)\bigr)\) is equivalent to
		the terminal category \(\final\).
		First of all, this category is non-empty, since it has \(\alpha, 1_\alpha\) as an object.
		Given a \(1\)-cell \((\beta, \zeta) \colon (y, \alpha) \to (x, 1_x)\),
		the composition \(1_x \comp_0 \beta = \beta\) must be the source of
		\(\zeta\) and moreover \(\zeta \colon \beta \to \alpha\) must be an invertible
		\(2\)-cell of \(\B\) by Proposition~\ref{prop:cartesian_edges}.
		For every pair of \(1\)-cells \((\beta, \zeta), (\beta', \zeta') \colon (y, \alpha) \to (x, 1_x)\),
		a \(2\)-cell \(\tau \colon (\beta, \zeta) \to (\beta', \zeta')\) must satisfy
		the relation \(\zeta' \comp_1 \tau = \zeta\). It follows that
		such a \(2\)-cell is invertible and in fact it always exists and it is unique,
		namely it is given by \(\tau = (\zeta')^{-1} \comp_1 \zeta\).
		This shows that for every object \((y, \alpha)\) in \(\Bxcart\) the
		category \(\Bxcart\bigl((y, \alpha), (x, 1_x)\bigr)\) is
		the chaotic category on the objects of the form \((\beta, \zeta)\),
		with \(\zeta \colon \beta \to \alpha\) invertible, and it is thus
		equivalent to the terminal category \(\final\), \ie \((x, 1_x)\) is biterminal.
	\end{proof}

	\begin{rem}
		In the general case of a \(2\)-functor \(F \colon \J \to \C\)
		and a marking \(E\)	on \(\J\),
		the \(2\)-category \(\mCF\) does not necessarily have a quasi-terminal object
		and similarly \(\mCFcart\) does not necessarily have a biterminal object.
	\end{rem}

\subsection{The modified \pdftwo-category of cones}
\label{sec:special_cones}
	
	Let \(\mJ = (\J, E)\) be a marked 2-category, \(F \colon \J \to \C\) be a \(2\)-functor,
	\(\ell\) an object of \(\C\),
	\(\lambda \colon \Delta\ell \to F\) an \(E\)-cone over \(F\) and
	\(F^\triangleleft\colon \mJ^\triangleleft \to \C\) the corresponding \(2\)-functor
	(cf.~Remark~\ref{rem:cones}).
	In this section we introduce an auxiliary \(2\)-category of cones
	that is always biequivalent to \(\C^{/\ell} = \trbis{\C}{\ell}\) and that
	 is biequivalent to \(\mCF\) if and only if \((\ell, \lambda)\)
	is an \(E\)-bilimit of \(F\).

	\begin{define}
		We define \(\CFmod\) as the \(2\)-full sub-\(2\)-category of \(\trbis{\C}{F^\triangleleft}\)
		whose objects are the \(2\)-functors \(\final \ast \final \ast \J \to \C\)
		such that their restrictions to \(\final \ast \final \ast \{j\} \to \C\), for any object \(j\) of \(\J\),
		determine diagrams
		\[
			\begin{tikzcd}[column sep=tiny]
				x \ar[rr,"h"] \ar[rd, "\alpha_j"{swap,name=t}]	&		& \ell \ar[ld, "\lambda_j"{pos=0.4}] \\
																& F(j)	&
				\ar[Rightarrow, from=1-3, to=t, shorten <=5mm, shorten >=4mm, "\sigma_j"{swap}, "\simeq"{pos=.6}]
			\end{tikzcd}
		\]
		in \(\C\) with \(\sigma_j \colon \lambda_j h \to \alpha_j\) an \emph{invertible} \(2\)-cell of \(\C\).
	\end{define}

	\begin{parag}
	\label{par:structure_modified_slice}
		We denote an object of \(\CFmod\) by \((x, h, \alpha, \sigma)\),
		where we mean that:
		\begin{itemize}
			\item \(h \colon x \to \ell\) is a \(1\)-cell of \(\C\);

			\item the pair \((x, \alpha)\) is an \(E\)-cone over \(F\), so that
			in particular for every \(1\)-cell \(\kappa \colon i \to j\) in \(\J\)
			we have a triangle
			\[
				\begin{tikzcd}[column sep=tiny]
											& x \ar[dl, "\alpha_i"'] \ar[dr, "\alpha_j"{name=t}]	&		\\
					Fi \ar[rr, "F\kappa"{swap}] 	&															& Fj
					\ar[Rightarrow, from=2-1, to=t, shorten <=3mm, shorten >=3mm, "\alpha_h"{swap}]
				\end{tikzcd},
			\]
			with \(\alpha_\kappa\) invertible whenever \(\kappa \in E\).

			\item for every \(i\) in \(\Ob(\J)\),
			\(\sigma_i \colon \lambda_i h \to \alpha_i\) is an invertible \(2\)-cell  of \(\C\),
			that we can depict as
			\[
				\begin{tikzcd}[column sep=tiny]
				x \ar[rr,"h"] \ar[rd, "\alpha_i"{swap,name=t}]	&		& \ell \ar[ld, "\lambda_i"] \\
																	& F(i)	&
				\ar[Rightarrow, from=1-3, to=t, shorten <=5mm, shorten >=4mm, "\sigma_i"{swap}, "\simeq"{pos=.6}]
				\end{tikzcd}.
			\]
			Notice that by Remark~\ref{rem:lax_cones_are_lax_transformations} we have that \(\sigma \colon \lambda \cdot h \to \alpha\)
			is an invertible modification.
		\end{itemize}
		The objects of the form \((x, h, \lambda\cdot h, \iota)\), with \(\iota_i\)
		the identity \(2\)-cell of \(\lambda_i h\) for all \(i\) in \(\Ob(\J)\), are particularly simple
		and we will make use of them in the following proofs in order to simplify the coherences appearing
		in explicit form of the cells of \(\CFmod\).
		We shall also commit the notational abuse of denoting by \(\iota\) the appropriate trivial modification,
		regardless of the 1-cell of \(\C\) involved;
		for instance, for \(g \colon y \to \ell\) another 1-cell of \(\C\) we shall write \((y, g, \lambda \cdot g, \iota)\)
		where we mean that here \(\iota_i\) is the identity of \(\lambda_i g\) for all \(i \in \Ob(\J)\).

		A 1-cell from \((x, h, \alpha, \sigma)\) to \((y, g, \beta, \tau)\) is given by a triple \((f, \zeta, \mu)\),
		where \(f \colon x \to y\) is a 1-cell in \(\C\), \(\zeta \colon g \comp_0 f \to h\) a 2-cell of \(\C\)
		and \(\mu\) is the modification satisfying \(\mu \comp_1 (\tau\comp_0 f) = \sigma \comp_1 \lambda \cdot \zeta\).
		Since \(\tau\) is an invertible modification, we actually have that \(\mu\) is defined as
		\(\sigma \comp_1 \lambda\cdot \zeta \comp_1 (\tau^{-1} \comp_0 f)\).
		Notice that a 1-cell from \((x, h, \lambda\cdot h, \iota)\) to \((y, g, \lambda\cdot h, \iota)\)
		must be of the form \((f, \zeta, \lambda \cdot \zeta)\).

		In light of Lemma~\ref{lemma:simpler_object} below it will suffice to describe 
		the 2-cells of \(\CFmod\) whose 0-dimensional source
		and target are of the form \((x, h, \lambda\cdot h, \iota)\) and \((y, g, \lambda\cdot h, \iota)\).
		Here and in what follows, we will always commit the abuse of denoting by \(\iota\) the
		appropriate trivial modification. In this case, if we are given 1-cells \((f, \zeta, \lambda \cdot \zeta)\)
		and \((f', \zeta', \lambda \cdot \zeta')\) from \((x, h, \lambda\cdot h, \iota)\) to \((y, g, \lambda\cdot g, \iota)\),
		then we have that a 2-cell of \(\CFmod\) is precisely a 2-cell of \(\Cl\), that is a 2-cell \(\Xi \colon f \to f'\colon x \to y\) of \(\C\)
		satisfying \(\zeta' \comp_1 (g \comp_0 \Xi) = \zeta\). Indeed, if for instance \(k \colon i \to j\) is a 1-cell of \(\J\),
		then by the coherences imposed by the 1-cells of \(\CFmod\) we have
		\begin{equation*}
		 \begin{split}
			\bigl(\lambda_j \zeta' \comp_1 (\lambda_k g \comp_0 f') \bigr) \comp_1 (Fk \comp_0 \lambda_i g \comp_0 \Xi)
			& =  \bigl(\lambda_k h \comp_1 (Fk \comp_0 \lambda_i \zeta') \bigr) \comp_1 (Fk \comp_0 \lambda_i g \comp_0 \Xi) \\
			& = \lambda_k h \comp_1 (Fk \comp_0 \lambda_i \zeta),
		 \end{split}
		\end{equation*}
		and similarly for a 2-cell of \(\J\). Morally, the coherences of a 2-cell of \(\CFmod\)
		are encoded by higher cells of \(D_2 \ast D_0 \ast D_m\), with \(m = 0, 1, 2\), but since we are
		truncating at the 2-dimensional level, many of these higher cells are trivial (cf.~Remark~\ref{rem:1-cells_of_CF})
		and moreover we are choosing specific objects (and 1-cells) of \(\CFmod\) which further simplify
		some of the coherences involved with identities.
	\end{parag}

	We shall now prove that the \(2\)-functor \(q \colon \CFmod \to \Cl\) is a biequivalence, following the strategy outlined in \ref{par: bieq}.
	The proof is subdivided in few steps and a preliminary auxiliary lemma, that we shall use to simplify some of the coherences involved.

	\begin{lemma}\label{lemma:simpler_object}
		Let \((x, h, \alpha, \sigma)\) be an object of \(\CFmod\). Then it is isomorphic to the object
		\((x, h, \lambda \cdot h, \iota)\), where \(\iota_i\)
		is the identity \(2\)-cell of \(\lambda_i h\) for all \(i\) in \(\Ob(\J)\).
	\end{lemma}

	\begin{proof}
		The \(1\)-cell \((1_x, \sigma)\) of \(\CF\) is an isomorphism and it is
		such that \((h, \iota) \comp_0 (1_x, \sigma) = (h, \sigma)\).
		This is equivalent to say that \((1_x, 1_h, \sigma^{-1})\) is an isomorphism
		from \((x, h, h \cdot \lambda, \iota)\) to \((x, h, \alpha, \sigma)\) in \(\CFmod\).
	\end{proof}

	\begin{parag}[Surjectivity on objects]
		Consider an object \(h \colon x \to \ell\) of \(\Cl\).
		Then the object  \((x, h, h \cdot \lambda, \iota)\) of \(\CFmod\) maps to \((x, h)\).
	\end{parag}

	\begin{parag}[Fullness on \(1\)-cells]
		Let \((f, \zeta) \colon (x, h) \to (y, g)\) be a \(1\)-cell in \(\Cl\),
		that we can depict as
		\[
			\begin{tikzcd}[column sep=tiny]
				x \ar[rr,"f"] \ar[rd, "h"{swap,name=t}]	&		& y \ar[ld, "g"] \\
												& \ell	&
				\ar[Rightarrow, from=1-3, to=t, shorten <=3mm, shorten >=2mm, "\zeta"{swap}]
			\end{tikzcd}
		\]
		Given two objects \((x, h, \alpha, \sigma)\) and \((y, g, \beta, \tau)\) of \(\CFmod\) that lift the source and target (respectively) of \((f,\zeta)\),
		we want to find a \(1\)-cell \(\Di_1 \ast \Di_0 \ast \J \to \C\)
		from \((x, h, \alpha, \sigma)\) to \((y, g, \beta, \tau)\) such that its projection via \(q\)
		is the \(1\)-cell \((f, \zeta)\) of \(\Cl\). 
		According to paragraph~\ref{par:structure_modified_slice}, the triple \((f, \zeta, \mu)\)
		is a 1-cell of \(\CFmod\) with the correct boundary, with \(\mu = \sigma \comp_1 \lambda\cdot \zeta \comp_1 (\tau^{-1} \comp_0 f)\),
		and it is clearly mapped to \((f, \zeta)\). Notice that this lifting requires no additional data, but just
		commutativity conditions, so that such a lift must be unique.
		We also remark that if we choose \((x, h, \lambda\cdot h, \iota)\) and \((y, g, \lambda\cdot g, \iota)\)
		as objects lifting the source and target of \((f, \zeta)\), then
		the 1-cell \((f, \zeta, \lambda\cdot \zeta)\) between these two objects of \(\CFmod\)
		maps to \((f, \zeta)\). A simple verification shows that the square
		\begin{equation}
		 \label{dia:replacing_objects}
			\begin{tikzcd}[column sep=large]
				(x, h, \alpha, \sigma) \ar[r, "{(f, \zeta, \mu)}"] \ar[d, "{(1_x, 1_h, \sigma)}"']
						& (y, g, \beta, \tau) \ar[d, "{(1_y, 1_g, \tau)}"] 	\\
				(x, h, \lambda\cdot h, \iota) \ar[r, "{(f, \zeta, \lambda \cdot \zeta)}"']
						& (y, g, \lambda \cdot \zeta, \iota)
			\end{tikzcd}
		\end{equation}
		of \(\CFmod\) is commutative, where the vertical 1-cells are (the inverses of) those
		of Lemma~\ref{lemma:simpler_object}. In particular, this means that it was actually
		enough to find the (unique) lift \((f, \zeta, \lambda\cdot \zeta)\) of \((f, \zeta)\) with \((x, h, \lambda\cdot h, \iota)\) and \((y, g, \lambda\cdot g, \iota)\)
		as source and target, since then the composition
		\[
			(1_y, 1_g, \tau^{-1}) \comp_0 (f, \zeta, \lambda\cdot \zeta) \comp_0 (1_x, 1_h, \sigma)
		\]
		is a lift of \((f, \zeta)\) with \((x, h, \alpha, \sigma)\) and \((y, g, \beta, \tau)\)
		as source and target.		
	\end{parag}

	\begin{parag}[Fullness on \(2\)-cells]
		Let \(\Xi \colon (f, \zeta) \to (f', \zeta') \colon (x, h) \to (y, g)\) be a \(2\)-cell in \(\Cl\),
		that we can depict as
		\[
			\begin{tikzcd}[column sep=small, row sep=3.2em]
					x \ar[rr, bend left, "f", ""{swap, name=s}]
						\ar[rr, bend right, ""{name=t}, "{f'}"{above right = 0pt and 3pt}]
						\ar[rd, "h"{swap, name=a}]	&	& y \ar[ld, "g", ""{swap, pos=0.1, name=b}]	\\
												& \ell &
					\ar[Rightarrow, from=s, to=t, "\Xi"{swap}]
					\ar[Rightarrow, from=b, to=a, shorten <=3mm, shorten >=3mm, "{\zeta'}"]
			\end{tikzcd}
				=
			\begin{tikzcd}[column sep=small, row sep=3.2em]
					x \ar[rr, bend left, "f"] \ar[rd, "h"{swap, name=a}]
						&	& y \ar[ld, "g", ""{swap, pos=0.1, name=b}]	\\
						& \ell &
					\ar[Rightarrow, from=b, to=a, shorten <=3mm, shorten >=3mm, "\zeta"']
			\end{tikzcd}
			\]
		Given two objects \((x, h, \alpha, \sigma)\) and \((y, g, \beta, \tau)\) of \(\CFmod\), and lifts of \((f, \zeta, \mu)\) and \((f', \zeta', \mu')\) to \(\CFmod\),
		we wish to find a \(2\)-cell \(\Di_2 \ast \Di_0 \ast \J \to \C\)
		from \((f, \zeta, \mu)\) to \((f', \zeta', \mu')\) such that its projection via \(q\)
		is the \(2\)-cell \(\Xi\) of \(\Cl\).
		Thanks to Lemma~\ref{lemma:simpler_object} and the observation at the end of the previous point, we can again assume
		that \(\alpha = \lambda \cdot h\), \(\beta = \lambda \cdot g\) and \(\sigma_i\) and \(\tau_i\)
		are the identity \(2\)-cells for all objects \(i\) of \(\J\). Indeed, if we find a lift \(\Xi'\) of \(\Xi\)
		with \((x, h, \lambda\cdot h, \iota)\) and \((y, g, \lambda\cdot g, \iota)\) as source and target,
		then the 2-cell \((1_y, 1_g, \tau^{-1}) \comp_0 \Xi' \comp_0 (1_x, 1_h, \sigma)\) is a lift of \(\Xi\)
		having \((x, h, \alpha, \sigma)\) and \((y, g, \beta, \tau)\) as source and target.
		By the discussion in paragraph~\ref{par:structure_modified_slice}, we know that \(\Xi\)
		is a 2-cell of \(\CFmod\) with correct source and target and that it clearly lifts \(\Xi\).
	\end{parag}

	\begin{parag}[Faithfulness on \(2\)-cells]
		By the observation in paragraph~\ref{par:structure_modified_slice}.
		a \(2\)-cell of \(\CFmod\)
		from \((f, \zeta, \lambda\cdot \zeta)\) to \((f', \zeta', \lambda\cdot \zeta')\)
		simply amounts to a \(2\)-cell \(\Xi \colon f \to f'\) of \(\C\)
		satisfying \(\zeta' \comp_1 (g \ast_0 \Xi) = \zeta\),
		which is the same thing as a \(2\)-cell of \(\Cl\) from
		\((f, \zeta)\) to \((f', \zeta')\). Therefore the map
		\[
			\CFmod\bigl((f, \zeta, \lambda\cdot \zeta), (f', \zeta', \lambda\cdot \zeta')\bigr) \to \Cl\bigl((f, \zeta), (f', \zeta')\bigr)
		\]
		is a bijection.
		Consider now two generic parallel 1-cells \((f, \zeta, \mu)\) and \((f', \zeta', \mu')\) of \(\CFmod\),
		say from \((x, h, \alpha, \sigma)\) to \((y, g, \beta, \tau)\).
		The commutativity of diagram~\eqref{dia:replacing_objects}
		implies that we have an isomorphism of categories
		\[
			\CFmod\bigl((x, h, \alpha, \sigma), (y, g, \beta, \tau)\bigr) \cong
			\CFmod\bigl((x, h, \lambda \cdot h, \iota), (y, g, \lambda \cdot g, \iota)\bigr)
		\]
		mapping \((f, \zeta, \mu)\) to \((f, \zeta, \lambda\cdot \zeta)\),
		so that in particular we get a bijection
		\[
			\CFmod\bigl((f, \zeta, \mu), (f', \zeta', \mu')\bigr) \cong
				\CFmod\bigl((f, \zeta, \lambda\cdot \zeta), (f', \zeta', \lambda\cdot \zeta')\bigr).
		\]
		Hence, the \(2\)-functor \(q \colon \CFmod \to \Cl\) is faithful on \(2\)-cells.
	\end{parag}

	Putting together the previous four points, we get the following result.

	\begin{prop}
	\label{prop:mod_slice_trivial_fibration_over_point}
		The canonical projection \(q \colon \CFmod \to \Cl\) is a biequivalence.
	\end{prop}

\begin{cor}
\label{cor:characterization-of-limits}
Let \((\J,E)\) be a marked 2-category, \(F\colon \J \to \C\) a 2-functor,  and \((\ell,\lambda) \in \CF\) an \(E\)-lax cone over \(F\). Then \((\ell,\lambda)\) is an \(E\)-bilimit cone if and only if the projection \(\CFmod \to \mCF\) is a biequivalence.
\end{cor}
\begin{proof}
		By definition, \((\ell,\lambda)\) is an \(E\)-bilimit cone if and only if the canonical 2-natural transformation
		\[
			\lambda \cdot (-) \colon \C(-, \ell) \to [\J, \C]_E(\Delta -, F)  
		\]
		of functors \(\C^\op \to \Cat\), is an equivalence.
		Since the fibrations classifying these two 2-functors are
		\(\Cl\) and \(\mCF\), respectively (cf.~Example~\ref{example:slice_classifies_cones}),
		by virtue of Theorem~\ref{S/U} (see also Remark~\ref{rem:S/U}), this 2-natural transformation is an equivalence
		if and only if the corresponding 2-functor \(\Cl \to \mCF\) depicted below is a biequivalence.
		\[
			\begin{tikzcd}[column sep=tiny]
				\Cl \ar[rr] \ar[rd]	&		& \mCF \ar[ld]	\\
									& \C 	&
			\end{tikzcd}
		\]
		Now, the triangle
		\[
			\begin{tikzcd}[column sep=tiny]
							& \CFmod \ar[ld, "\simeq"'] \ar[rd, ""{swap, name=s}]	&		\\
				\Cl \ar[rr]	&									& \mCF
				\ar[from=s, to=2-1, Rightarrow, shorten >= 10pt, shorten <= 5 pt, "\cong"{swap, pos=0.35}]
			\end{tikzcd}
		\]
		of 2-categories commutes up-to an invertible 2-natural transformation
		given by Lemma~\ref{lemma:simpler_object} (see also diagram~\eqref{dia:replacing_objects}).
		Hence, applying Proposition~\ref{prop:mod_slice_trivial_fibration_over_point}
		together with the 2-out-of-3 property of biequivalences
		and the stability of biequivalences by 2-natural isomorphisms
		to the previous triangle
		we deduce that the 2-functor \(\Cl \to \mCF\) is a biequivalence
		if and only if \(\CFmod \to \mCF\) is one.
\end{proof}

\section{Bilimits and bifinal cones}

	\begin{parag}
		We fix a marked 2-category \(\mJ = (\J, E)\) and a 2-functor \(F \colon \J \to \C\),
		and we denote by \(p \colon \mCF \to \C\) the projection \(2\)-functor.
		In this section we show that a cone \((\ell, \lambda)\) over \(F\) is an \(E\)-bilimit cone
		if and only if it is the center of a limiting contraction \(H\) on \(\mCF\), as defined below.
		\begin{notate}
			In what follows, we will use the letter \(H\) to denote a contraction, in contrast with the use of \(\gamma\) and other greek letters in previous sections. In fact, we will confine our use of greek letters to denote cones on diagrams.  
		\end{notate}
	\end{parag}

		\begin{define}
		\label{def: limit contraction}
		 	Let \(H\) be a contraction on \(\mCF\). We say that \(H\) is a \ndef{limiting contraction} if it is an \(M_{\car}\)-contraction, where \(M_{\car}\) is the collection of \(p\)-cartesian arrows in \(\mCF\).
		A cone \((\ell, \lambda)\) over \(\mCF\) will be called \ndef{limiting bifinal} if
		it is the center of a limiting contraction on \(\mCF\).
		 \end{define} 
		
	\begin{parag}
	 \label{parag:limit_contraction}
		More explicitly, the condition of being an \(M_{\car}\)-contraction means that
		 \begin{enumerate}
			\item\label{item:cartesian_cones} For every cone \((x, \alpha)\) over \(F\), the \(1\)-cell \(H(x, \alpha)\)
			is \(p\)-cartesian.
			\item\label{item:image_cartesian_edges} For every \(p\)-cartesian edge \((f, \mu) \colon (x, \alpha) \to (y, \beta)\),
			the 2-cell \(H(f, \mu)\) of \(\mCF\) is invertible.
		\end{enumerate}
		The first condition then means that 
		for all objects \(i\) of \(\J\) the 2-cell
			\[
				\begin{tikzcd}[column sep=tiny]
					x \ar[rd, "\alpha_i"{swap, name=a}] \ar[rr]	&		& \ell \ar[ld, "\lambda_i"]	\\
																& Fi 	&
					\ar[Rightarrow, from=1-3, to=a, shorten >=2.5mm, shorten <=3.5mm]
				\end{tikzcd}
			\]
		of \(H(x, \alpha)\) is invertible. The second condition
		states that in the commutative diagram
		\begin{equation}
		\label{condition 2  of limiting contraction}
		\vcenter{\hbox{
			\begin{tikzpicture}[scale=1.6]
				\square{
					/square/label/.cd,
						0={$x$}, 1={$y$}, 2={$\ell$}, 3={$Fi$},
						01={$f$}, 12={$h(y, \beta)$}, 23={$\lambda_i$},
						02={$h(x, \alpha)$}, 13={$\beta_i$}, 03={$\alpha_i$},
						012={$H(f, \mu)$}, 023={$\sigma(x, \alpha)_i$},
						013={$\mu_i$}, 123={$\sigma(y, \beta)_i$},
					/square/arrowstyle/.cd,
						012 = {Leftarrow}, 013 = {Leftarrow},
						023 = {Leftarrow}, 123 = {Leftarrow},
						0123={equal},
					/square/labelstyle/.cd,
						012={below right}, 123={below left}
				}
			\end{tikzpicture}
			}}
		\end{equation}
		of \(\C\), where \(H(x, \alpha) = (h(x, \alpha), \sigma(x, \alpha))\) and
		\(H(y, \beta) = (h(y, \beta), \sigma(y, \beta))\), whenever the 2-cells \(\mu_i\)
		are invertible for all \(i \in \Ob(\J)\), the 2\nbd-cell
		\(H(f, \mu)\) is invertible.
		Equivalently, the 2-cell \(H(f, \mu)\) is invertible whenever the 2-cells \(\lambda_i \comp_0 H(f, \mu)\) are invertible
		for all \(i \in \Ob(\J)\).
		
	\end{parag}

	\begin{rem}
	It follows from Remark~\ref{rem:homotopy-sound} and the fact that \(M_{\car}\) is closed under isomorphisms of 1-cells that if \(H\) and \(K\) are two contractions on \(\mCF\) with center \((\ell,\lambda)\) then \(H\) is limiting if and only if \(K\) is limiting.
	\end{rem}

Applying Proposition~\ref{prop:E-final-iff-E-contraction} to the present case yields:
	\begin{prop}
			\label{prop: limiting iff final}
			\(\mCF\) admits a limiting contraction with center \((\ell, \lambda)\) if and only if every cone \((\ell',\lambda' )\) admits a \(p\)-cartesian edge \((\ell',\lambda') \to (\ell,\lambda)\) and the 2-functor
		\(\{(\ell, \lambda)\} \to (\mCF, M_{\car})\) is final.
	\end{prop}
		
Given a contraction \(H\) on \(\mCF\) with center \((\ell,\lambda)\), we may consider the canonical projection
	\(\CFmod \to \mCF\), where \(F^{\triangleleft}\colon \Di_0 \ast \mJ \to \C\) is the cone determined by \((\ell,\lambda)\). 
	\begin{prop}
	 \label{prop:contraction_to_biequivalence}
		If \((\ell, \lambda)\) is limiting bifinal 
		then the associated projection \(2\)\nbd-func\-tor \(\CFmod \to \mCF\) is a biequivalence.
	\end{prop}
	\begin{proof}
		Call \(H\) the $M_{\car}$-contraction. We will prove that the projection \(\CFmod \to \mCF\)
		is a trivial fibration.
		\begin{description}
			\item[Surjectivity on objects]
				For a cone \((x, \alpha)\) over \(F\), the \(1\)-cell
				\(H(x, \alpha)\) is by definition an object of \(\CFmod\) that lifts \((x, \alpha)\).

			\item[Fullness on \(1\)-cells]
				Let \((x, h, \alpha, \sigma)\) and \((y, g, \beta, \tau)\) be objects of \(\CFmod\)
				and fix a \(1\)-cell \((f, \mu) \colon (x, \alpha) \to (y, \beta)\) of \(\mCF\).
				We wish to lift \((f, \mu)\) to \(\CFmod\).
				Applying the contraction to \((f, \mu)\) gives a triangle
				\[
					\begin{tikzcd}[column sep=tiny]
						{(x, \alpha)} \ar[rr, "{(f, \mu)}"] \ar[rd, "{H(x, \alpha)}"{swap, name=t}]
							&					& {(y, \beta)} \ar[ld, "{H(y, \beta)}"]  \\
							& {(\ell, \lambda)}	&
						\ar[Rightarrow, from=1-3, to=t, shorten <=4mm, shorten >=5mm, "\Gamma"{swap}]
					\end{tikzcd}
				\]
				in \(\mCF\), where we have set \(\Gamma = H(f, \mu)\). 
				By assumption, the 1-cells \((h, \sigma) \colon (x, \alpha) \to (\ell, \lambda)\)
				and \((g, \tau) \colon (y, \beta) \to (\ell, \lambda)\) 
				of \(\mCF\)
				are \(p\)-cartesian. 
				Applying to these 1-cells the limiting contraction \(H\),
				by point~\eqref{item:image_cartesian_edges} we get two invertible
				\(2\)-cells
				\[
					H(h, \sigma) \colon (h, \sigma) \to H(x, \alpha)
					\quad\text{and}\quad
					H(g, \tau) \colon (g, \tau) \to H(y, \beta)
				\]
				of \(\mCF\).
				The composite 2-cell
				\(\zeta = H(h, \sigma)^{-1} \comp_1 H(f, \mu) \comp_1 \bigl(H(g, \tau) \comp_0 (f, \mu)\bigr) \)
				of \(\mCF\) fills the triangle
				\[
					\begin{tikzcd}[column sep=tiny]
						{(x, \alpha)} \ar[rr, "{(f, \mu)}"] \ar[rd, "{(h, \sigma)}"{swap, name=t}]
							&					& {(y, \beta)} \ar[ld, "{(g, \tau)}"]  \\
							& {(\ell, \lambda)}	&
						\ar[Rightarrow, from=1-3, to=t, shorten <=4mm, shorten >=5mm, "\zeta"{swap}]
					\end{tikzcd}
				\]
				therefore providing a 1-cell
				\((f, \zeta,\mu,) \colon (x, h, \alpha, \sigma) \to (y, g, \beta, \tau)\)
				of \(\CFmod\) lifting the 1-cell \((f, \mu)\) of \(\mCF\).

			\item[Fullness on \(2\)-cells]
				Let \((f, \zeta, \mu)\) and \((f', \zeta',\mu')\) be two parallel \(1\)-cells of \(\CFmod\) from
				\((x, h, \alpha, \sigma)\) to \((y, g, \beta, \tau)\). We wish to find a lift to \(\CFmod\)
				for any \(2\)-cell \(\Xi \colon (f, \mu) \to (f', \mu')\) of \(\mCF\).
				Using the description of~\ref{par:structure_modified_slice}, we may view \(\Xi\) itself as such a lift: for this, we have to check that the identity
				\begin{equation}
				 \label{eq:fullness_2-cells}
					\zeta' \comp_1 (g \comp_0 \Xi) = \zeta
				\end{equation}
				is satisfied in \(\mCF\). Applying the constraint of the contraction \(H\)
				to the 2-cell \(\zeta \colon (gf, \tau \cdot \mu) \to (h, \sigma) \colon (x, \alpha) \to (\ell, \lambda)\),
				where \((gf, \tau \cdot \mu) = (g, \tau) \comp_0 (f, \mu)\), of \(\mCF\) we get the relation
				\[
					H(gf, \tau \cdot \mu) = H(h, \sigma) \comp_1 \zeta.
				\]
				By the functoriality of the contraction, the left-hand side of this equation is equal to
				\(H(f, \mu) \comp_1 \bigl(H(g, \tau) \comp_0 (f, \mu)\bigr)\) and moreover the 2-cell
				\(H(h, \sigma)\) is invertible, as \((h, \sigma)\) is \(p\)-cartesian by assumption.
				This implies that we can write
				\[
					\zeta = H(h, \sigma)^{-1} \comp_1 H(f, \mu) \comp_1 \bigl(H(g, \tau) \comp_0 (f, \mu)\bigr)
				\]
				and similarly
				\[
					\zeta' = H(h, \sigma)^{-1} \comp_1 H(f', \mu') \comp_1 \bigl(H(g, \tau) \comp_0 (f', \mu')\bigr).
				\]
				Hence, equation~\eqref{eq:fullness_2-cells} is satisfied if and only if we have
				\[
					H(f', \mu') \comp_1 \bigl(H(g, \tau) \comp_0 (f', \mu')\bigr) \comp_1 \bigl((g, \tau) \comp_0 \Xi\bigr) =
						H(f, \mu) \comp_1 \bigl(H(g, \tau) \comp_0 (f, \mu)\bigr).
				\]
				Notice that by the interchange rule for
				\[
					\begin{tikzcd}[column sep=large]
					  (x, \alpha) \ar[r, bend left, "{(f, \mu)}", ""{swap, name=si}] \ar[r, bend right, ""{name=ti}, "{(f', \mu')}"']	&
					  (y, \beta) \ar[r, bend left, "{(g, \tau)}", ""{swap, name=sii}] \ar[r, bend right, ""{name=tii}, "{H(y, \beta)}"']	&
					  (\ell, \lambda)
					  \ar[Rightarrow, from=si, to=ti, "{\Xi}"]
					  \ar[Rightarrow, from=sii, to=tii, "{H(g, \tau)}"{description}]
					\end{tikzcd}
				\]
				we actually have
				\[
					\bigl(H(g, \tau) \comp_0 (f', \mu')\bigr) \comp_1 \bigl((g, \tau) \comp_0 \Xi\bigr) =
						\bigl(H(y, \beta) \comp_0 \Xi\bigr) \comp_1 \bigl(H(g, \tau) \comp_0 (f, \mu)\bigr)
				\]
				and the 2-cell \(H(g, \tau)\) is invertible, since \((g, \tau)\) is \(p\)-cartesian by assumption,
				so that equation~\eqref{eq:fullness_2-cells} is satisfied if and only if
				the following equation is:
				\begin{equation}
				 \label{eq:fullness_2-cell-reduced}
				 	H(f', \mu') \comp_1 \bigl(H(y, \beta) \comp_0 \Xi\bigr) = H(f, \mu).
				\end{equation}
				Now, using the coherence for 2-cells given  by the contraction \(H\)
				applied to the 2-cell on the left-hand side of the previous equation
				and using the constraints for which \(H(H(x, \alpha)) = 1_{H(x, \alpha)}\)
				and \(H(\ell, \lambda) = 1_{(\ell, \lambda)}\), we finally get
				that equation~\eqref{eq:fullness_2-cell-reduced} is indeed satisfied.

			\item[Faithfulness on \(2\)-cells]
				This is clear by the explicit description of the \(2\)-cells
				of \(\CFmod\) and \(\mCF\), which are just \(2\)-cells of \(\C\) satisfying some coherence
				conditions. \qedhere
		\end{description}
	\end{proof}

	We now prove the converse of Proposition~\ref{prop:contraction_to_biequivalence}.
	For convenience, we first record the following observation about bilimits.

	\begin{prop}
	 \label{prop:bilimit_to_contraction}
		If \((\ell, \lambda)\) is an \(E\)-bilimit of the \(2\)-functor \(F \colon \J \to \C\),
		then \(\mCF\) admits a limiting contraction with center \((\ell, \lambda)\). 
	\end{prop}

\begin{proof}
	Let \((\ell,\lambda)\) be an \(E\)-bilimit of \(F\). Our goal is to construct a limiting contraction \(H\) with center \((\ell, \lambda)\).

	Since \((\ell,\lambda)\) is an \(E\)-bilimit the functor 
	\[
		\lambda \cdot (-) \colon \C(x, \ell) \to [\J, \C]_E(\Delta x, F)
	\]
	is an equivalence for every \(x \in \Ob(\C)\). 
	In particular, \(\lambda \cdot (-)\) is essentially surjective, so that we can find, for any \(\alp \in [\J, \C]_E(\Delta x, F)\), a 1-cell \(h \colon x \to \ell\) of \(\C\) and an isomorphism \(\sigma\colon \alp \xrightarrow{\cong} \lambda \cdot h\). 
	The data of \(\alp,h\) and \(\sig\) then determine a \(p\)-cartesian 1-cell of \(\mCF\) depicted by the triangle 
	\[
		\begin{tikzcd}[column sep=tiny]
			(x, \alpha) \ar[rr, "{(1_x, \sigma)}"] \ar[dr, "{(h, \sigma)}"']	& & (x, \lambda \cdot h) \ar[dl, "{(h, 1_{\lambda \cdot h})}"] \\
					&	(\ell, \lambda)		&
		\end{tikzcd}
	\]
	We define our contraction on the level of objects by associating to \((x,\alp)\) the \(p\)-cartesian edge \(H(x, \alpha) := (h, \sigma)\), where we may assume without loss of generality that we have picked \(H(\ell, \lambda)\) to be \(1_{(\ell, \lambda)} = (1_\ell, 1_\lambda)\).
	
	Now consider a 1-cell \((f, \mu) \colon (x, \alpha) \to (y, \beta)\) in \(\mCF\), that can be seen as a 1-cell \(\mu\)
	of \([\J, \C]_E(\Delta x, F)\) from \(\alpha\) to \(\beta \cdot f\) (that is, a modification between
	the \(E\)-lax natural transformations \(\alpha\) and \(\beta \cdot f\)).  
	Let \((h,\sig) = H(x,\alp)\) and \((g,\tau) = H(y,\bet)\) be the 1-cells constructed above. 	The composite \(\rho = \sigma^{-1} \comp_1 \mu \comp_1 (\tau \comp_0 \Delta f)\) then gives a morphism from 
	\(\lambda \cdot(gf)\) to \(\lambda \cdot h\) in the category \([\J, \C]_E(\Delta x, F)\).
	Since \(\lambda\cdot(-)\) is fully faithful this morphism lifts to a unique morphism \(\Gamma\colon gf \to h\) in \(\C(x,\ell)\), which we can write as 
	a 2-cell  
	\[
		\begin{tikzcd}[column sep=tiny]
			x \ar[rr, "f"] \ar[rd, "h"', ""{name=t}]	&&	y \ar[ld, "g"] \\
						& \ell &
			\ar[from=1-3, to=t, Rightarrow, shorten <= 10pt, "\Gamma"']
		\end{tikzcd},
	\]
	in \(\C\).  
	One immediately checks, by applying \(\lambda\cdot (-)\), that we have
	\[
		\sigma_i \comp_1 (\lambda_i \comp_0 \Gamma) = \mu_i \comp_1 (\tau_i \comp_0 f)
	\]
	for all \(i \in \Ob(\J)\). 
	We may thus consider \(\Gamma\)
	as a 2-cell of \(\mCF\) with source \(H(y, \beta) \comp_0 (f, \mu)\) and target \(H(x, \alpha)\). We then extend our contraction to the level of 1-morphisms by setting \(H(f, \mu) = \Gamma\).
	Explicitly, the 2-cell \(H(f, \mu)\) is given by the pasting
	\begin{equation}
	 \label{dia:bilimits_to_contraction-action_1-cell}
		\begin{tikzcd}[column sep=small]
			(x, \alpha) \ar[rr, "{(1_x, \sigma)}"] \ar[rrrd, "{(h, \sigma)}"']	&&
			(x, \lambda \cdot h) \ar[rr, "{(f, \lambda \cdot \Gamma)}"] \ar[rd, "{(h, 1_{\lambda\cdot h})}"{description, name=t}]	&&
			(y, \lambda \cdot g) \ar[rr, "{(1_y, \tau^{-1})}"]\ar[ld, "{(g, 1_{\lambda\cdot g})}"{description}]		&&
			(y, \beta)	\ar[llld, "{(g, \tau)}"]	\\
				&&&		(\ell, \lambda)		&&&		
				\ar[Rightarrow, from=1-5, to=t, shorten <= 10pt, "\Gamma"{description, pos=0.65}]	
		\end{tikzcd}
	\end{equation}
	in $\mCF$, where the left-most and the right-most triangles are commutative by definition.
	With this description at hand, and given that \(\Gamma\) is uniquely determined by the fully faithfullness of \(\lambda\cdot(-)\), it is easy to verify that the assignment \(H\) is compatible with composition of 1-cells, that is 
	\[H((f',\mu') \circ (f,\mu)) = H(f,\mu) \ast_1 (H(f',\mu') \ast_0 (f,\mu))\]
	for any pair of composable 1-cells \((f,\mu),(f',\mu')\)  of \(\mCF\).
	In addition, since \(\lambda \cdot\Gamma =  \sigma^{-1} \comp_1 \mu \comp_1 (\tau \comp_0 \Delta f)\) with \(\tau\) and \(\sig\) invertible and \(\lambda \cdot (-)\) is an equivalence we have that \(\Gamma\) is invertible whenever \(\mu\) is invertible. 
		We conclude that \(H\) sends \(p\)-cartesian 1-cells to invertible 2-cells, and consequently that \(H(H(x, \alpha)) = 1_{H(x, \alpha)}\) by Remark~\ref{parag:apriori-differ}.
	
	To show that \(H\) constitutes a limiting contraction it will now suffice to show that for any 2-cell \(\Xi \colon (f, \mu) \to (f', \mu') \colon (x, \alpha) \to (g, \beta)\) of \(\mCF\),
	the relation
	\begin{equation}
	 \label{eq:bilimits_to_contraction-constraint_2-cells}
		H(f, \mu) = H(f', \mu') \comp_1 (H(y, \beta) \comp_0 \Xi)
	\end{equation}
	is satisfied. 
We begin by noticing that since \(\Xi\) is a 2-cell of \(\mCF\),
	by definition we have \(\mu_i = \mu_i' \comp_1 (\beta_i \comp_0 \Xi)\) for every object \(i\) of \(\J\).
	Using the previous part of the proof, we know
	that there exist 1-cells \(h \colon x \to \ell\) and \(g \colon y \to \ell\),
	isomorphic modifications \(\sigma \colon \lambda \cdot h \to \alpha\) and \(\tau \colon \lambda\cdot g \to \beta\)
	and a unique 2-cells \(\Gamma \colon gf \to h\) and \(\Gamma' \colon gf' \to h\) of \(\C\) such that
	\begin{align*}
		\mu & = \sigma \comp_1 \lambda\cdot \Gamma \comp_1 (\tau^{-1} \comp_0 \Delta f), \\
		\mu' & = \sigma \comp_1 \lambda\cdot \Gamma' \comp_1 (\tau^{-1} \comp_0 \Delta f').
	\end{align*}
	So the modification \(\mu\) is equal to
	\[
		\sigma \comp_1 \lambda\cdot \Gamma' \comp_1 (\tau^{-1} \comp_0 \Delta f') \comp_1 (\beta \comp_0 \Delta\Xi),
	\]
	which using the interchange law can be rewritten as
	\[
		\sigma \comp_1 \lambda\cdot \Gamma' \comp_1 (\lambda\cdot g \comp_0 \Delta\Xi) \comp_1 (\tau^{-1} \comp_0 \Delta f).
	\]
	Since \(\sigma\) and \(\tau\) are invertible by construction, we obtain the relation
	\begin{equation}
	 \label{eq:bilimits_to_contraction-data_2-cell}
		\lambda\cdot \Gamma = \lambda\cdot \Gamma' \comp_1 (\lambda\cdot g \comp_0 \Delta\Xi).
	\end{equation}
	Observe that since \((1_x, \sigma^{-1}) \colon (x, \lambda\cdot h) \to (x, \alpha)\) and \((1_y, \tau) \colon (y, \beta) \to (y, \lambda\cdot g)\)
	are invertible 1-cells of \(\mCF\), equation~\eqref{eq:bilimits_to_contraction-constraint_2-cells} holds if and only if
	it is whiskered by these two cells. Namely, we must show that
	\[
		\begin{tikzcd}[column sep=small]
			(x, \lambda \cdot h) \ar[rr, "{(1_x, \sigma^{-1})}"] \ar[rrrd, "{(h, 1_{\lambda\cdot h})}"']	&&
			(x, \alpha) \ar[rr, "{(f, \mu)}"] \ar[rd, "{(h, \sigma)}"{description, name=t}]	&&
			(y, \beta) \ar[rr, "{(1_y, \tau)}"]\ar[ld, "{(g, \tau)}"{description}]		&&
			(y, \lambda \cdot g)	\ar[llld, "{(g, 1_{\lambda\cdot g})}"]	\\
				&&&		(\ell, \lambda)		&&&		
				\ar[Rightarrow, from=1-5, to=t, shorten <= 10pt, "\Gamma"{description, pos=0.65}]	
		\end{tikzcd}
	\]
	is equal to
	\[
		\begin{tikzcd}[row sep=huge, column sep=small]
			(x, \lambda \cdot h)
				\ar[rr, "{(1_x, \sigma^{-1})}"] \ar[rrrd, "{(h, 1_{\lambda\cdot h})}"']	&&
			(x, \alpha)
				\ar[rr, bend left=20, "{(f, \mu)}"{description, name=sXi}] \ar[rr, bend right=20, "{(f', \mu')}"{description, name=tXi}]  \ar[rd, "{(h, \sigma)}"{description, name=t}]	&&
			(y, \beta) \ar[rr, "{(1_y, \tau)}"]\ar[ld, "{(g, \tau)}"{description}]		&&
			(y, \lambda \cdot g)	\ar[llld, "{(g, 1_{\lambda\cdot g})}"]	\\
				&&&		(\ell, \lambda)		&&&		
				\ar[Rightarrow, from=1-5, to=t, shift left=5pt, shorten <= 10pt, "{\Gamma'}"{description, pos=0.65}]
				\ar[Rightarrow, from=sXi, to=tXi, "{\Xi}"]
		\end{tikzcd}
	\]
	We can rewrite this as
	\[
		\begin{tikzcd}[row sep=large, column sep=tiny]
			(x, \lambda \cdot h) 
				\ar[rr, "{(f, \lambda\cdot \Gamma)}"] \ar[rd, "{(h, 1_{\lambda\cdot h})}"', ""{name=t}]	&&
			(y, \lambda \cdot g)	\ar[ld, "{(g, 1_{\lambda\cdot g})}"]	\\
				&		(\ell, \lambda)		&
				\ar[Rightarrow, from=1-3, to=t, shorten <= 10pt, shorten >=5pt, "\Gamma"{description, pos=0.55}]	
		\end{tikzcd}
		\quad = \quad
		\begin{tikzcd}[row sep=huge, column sep=tiny]
			(x, \lambda \cdot h)
				\ar[rr, bend left=20, "{(f, \lambda\cdot \Gamma)}"{description, name=sXi}] \ar[rr, bend right=20, "{(f', \lambda\cdot\Gamma')}"{description, name=tXi}] 
				\ar[rd, "{(h, 1_{\lambda\cdot h})}"', ""{name=t}]	&&
			(y, \lambda \cdot g)	\ar[ld, "{(g, 1_{\lambda\cdot g})}"]	\\
				&		(\ell, \lambda)		&	
				\ar[Rightarrow, from=1-3, to=t, shift left=7pt, shorten <= 10pt, shorten >=5pt, "{\Gamma'}"{description, pos=0.55}]
				\ar[Rightarrow, from=sXi, to=tXi, "{\Xi}"]
		\end{tikzcd}
	\]
	But the equality between these 2-cells of \(\mCF\) is expressed precisely by equation~\eqref{eq:bilimits_to_contraction-data_2-cell},
	which is satisfied by assumption. This complete the definition of a limiting contraction \(H\) with center \((\ell, \lambda)\),
	thereby finishing the proof.
\end{proof}

We are now ready to prove the converse of Proposition~\ref{prop:contraction_to_biequivalence}.

	\begin{cor}
	 \label{prop:biequivalence_to_contraction}
		Let \((\ell, \lambda)\) be an \(E\)-lax cone over
		the 2-functor \(F \colon \J \to \C\).
		If the projection \(2\)-functor \(\CFmod \to \mCF\) is a biequivalence then \((\ell,\lambda)\) is limiting bifinal.
	\end{cor}

	\begin{proof}
		This follows from Corollary~\ref{cor:characterization-of-limits} and Proposition~\ref{prop:bilimit_to_contraction}.
	\end{proof}

	We have all the elements to finally state and prove our main theorem.

	\begin{thm}
	 \label{thm:bilimits_are_bifinal}
		Let \((\ell, \lambda)\) be an \(E\)-lax cone
		of the 2-functor \(F \colon \J \to \C\). The following statements are equivalent:
		\begin{enumerate}
			\item\label{item:thm_bilimit} the \(E\)-lax \(F\)-cone \((\ell, \lambda)\) is an \(E\)-bilimit of \(F\);

			\item\label{item:thm_biequivalence} the induced 2-functor \(\CFmod \to \mCF\) is a biequivalence;

			\item\label{item:thm_bifinal} the object \((\ell, \lambda)\) of \(\mCF\) is limiting bifinal.
		\end{enumerate}
	\end{thm}

	\begin{proof}
		The equivalence between statements~\eqref{item:thm_bilimit} and~\eqref{item:thm_biequivalence}
		is provided 
		by Corollary~\ref{cor:characterization-of-limits}.
		By virtue of Proposition~\ref{prop:contraction_to_biequivalence} and its converse
		Corollary~\ref{prop:biequivalence_to_contraction}, \(\CFmod \to \mCF\)
		is a biequivalence if and only if \((\ell, \lambda)\) is limiting bifinal in \(\mCF\).
		This proves the equivalence between statements~\eqref{item:thm_biequivalence}
		and~\eqref{item:thm_bifinal}, thereby concluding the proof of the theorem.
	\end{proof}

	\begin{rem}
		Given a 2-functor \(F \colon \J \to \C\), a marking \(E\) on \(\J\)
		and a weight \(W \colon \J \to \Cat\), we observed in Remark~\ref{rem:weighted}
		that the \(W\)-weighted \(E\)-bilimit can be expressed as the \(E_W\)-bilimit
		of the functor \(\mathcal{E}l_F \xrightarrow{p} \J \xrightarrow{F} \C\),
		where \(p \colon \mathcal{E}l_F \to \J\) is the fibration classifying the weight
		2-functor \(W\) and \(E_W\) is a canonical marking on \(\mathcal{E}l_F\)
		detailed in~\cite[Definition~2.1.4]{DescotteDubucSzyldSigmaLimits} by Descotte, Dubuc and Szyld.
		In particular, a \(W\)-weighted \(E\)-bilimit of \(F\) is equivalent
		to a limiting bifinal object \((\ell, \lambda)\) in \(\C^{/Fp}\).
	\end{rem}

	Restricting to the cartesian edges of \(\mCF\) as a fibration over \(\C\) (see~\ref{parag:CFcart})
	and using the analysis on cartesian edges performed in \S\ref{sec:car-edges}, 
	we can deduce the following statement, which already appears
	as Proposition~5.4 in~\cite{ClingmanMoserLimitsDifferent} by clingman and Moser.

	\begin{cor}
	\label{cor:bilimit_cart_biterminal}
		An \(E\)-bilimit \((\ell, \lambda)\) of the \(2\)-functor \(F \colon \J \to \C\)
		is a biterminal object in \(\mCFcart\) .
	\end{cor}

	\begin{proof}
		It follows from the previous theorem that the 2-functor
		\(\Cl \to \mCF\) induced by \((\ell, \lambda)\) is a biequivalence,
		that we think of as a morphism of fibrations over \(\C\).
		This biequivalence of fibrations induces a 2-functor \((\Cl)_{\text{cart}} \to \mCFcart\),
		since biequivalences preserve and create cartesian edges,
		and it is clear that such a 2-functor is still a biequivalence.
		By virtue of Proposition~\ref{prop:biterminal_Bx} the object
		\((\ell, 1_\ell)\) is a biterminal object in \(\Cl\) and
		therefore its image \((\ell, \lambda)\) is a biterminal object
		in~\(\mCFcart\).
	\end{proof}
\bibliographystyle{amsplain}
\bibliography{biblio}
\end{document}

%% file: macros.tex
\usepackage[utf8]{inputenc}
\usepackage{eucal}
\usepackage[english]{babel}
\usepackage{amsmath}
\usepackage{amssymb}
\usepackage{amsopn}                 % defines new operators
\usepackage{amsbsy}
\usepackage{amsfonts}
\usepackage{MnSymbol}
\usepackage{amsthm}
\usepackage{newlfont}
\usepackage{amscd}
\usepackage{mathtools}              % \MoveEqLeft, \intertext
\usepackage{xpatch}
\usepackage{mathrsfs}
\usepackage{relsize}
\usepackage{tensor}
\usepackage{enumerate}
\usepackage{enumitem}
\usepackage{comment}
\usepackage{float}
\usepackage[abbreviations]{foreign} % i.e., e.g.
\usepackage{xy}
\xyoption{arrow}
\xyoption{matrix}
\SilentMatrices
\xyoption{tips}
\SelectTips{cm}{10}
\xyoption{2cell}
\UseAllTwocells
\usepackage{tikz}
\usetikzlibrary{cd}
\usetikzlibrary{calc}
\usetikzlibrary{math}
\usetikzlibrary{arrows}
\usetikzlibrary{fit}
\usetikzlibrary{positioning}

\usetikzlibrary{decorations.pathmorphing, arrows.meta}

\tikzset{Rightarrow/.style={double equal sign distance,>={Implies},->},
triple/.style={-,preaction={draw,Rightarrow}},
quadruple/.style={preaction={draw,Rightarrow,shorten >=0pt},shorten >=1pt,-,double,double
distance=0.2pt}}
\usepackage{xcolor}
\definecolor{darkblue}{rgb}{0,0,0.3}
\usepackage[colorlinks=true, citecolor=darkblue, filecolor=darkblue, linkcolor=darkblue, urlcolor=darkblue]{hyperref}
\newtheorem{thm}{Theorem}[section]
\newtheorem{cor}[thm]{Corollary}

\newtheorem{lemma}[thm]{Lemma}
\newtheorem{prop}[thm]{Proposition}

\theoremstyle{definition}
\newtheorem{define}[thm]{Definition}
\newtheorem{notate}[thm]{Notation}

\theoremstyle{remark}
\newtheorem{rem}[thm]{Remark}
\newtheorem{example}[thm]{Example}

\newtheorem{warning}[thm]{Warning}

\newtheorem{parag}[thm]{}
\makeatletter
\newif\ifsection
\preto\section{\sectiontrue}
\preto\subsection{\sectionfalse}
\xapptocmd\@sect{%
  \ifsection
    \numberwithin{thm}{section}
  \else
    \numberwithin{thm}{subsection}
  \fi
  \setcounter{thm}{0}\relax}
  {}{}
\makeatother
\newenvironment{nscenter}%
{\parskip=0pt\par\nopagebreak\centering}
{\par\noindent\ignorespacesafterend}        % center tikzpictures
\newcommand\nbd\nobreakdash
\newcommand{\ndef}{\emph}
\newcommand{\op}{\text{op}}

\newcommand{\Set}{\mathcal{S}\mspace{-2.mu}\text{et}}
\newcommand{\Cat}{{\mathcal{C}\mspace{-2.mu}\mathit{at}}}
\newcommand{\nCat}[1]{{#1}\hbox{\protect\nbd-}\kern1pt\Cat}	    % n-categories
\newcommand{\s}{\mathcal{S}\mspace{-2.mu}\text{et}_{\Delta}}

\newcommand\pdftwo{\texorpdfstring{$2$}{2}}
\DeclareMathOperator{\Id}{Id}
\DeclareMathOperator{\Ob}{Ob}

\DeclareMathOperator{\co}{co}

\def\alp{{\alpha}}
\def\bet{{\beta}}
\def\gam{{\gamma}}

\def\sig{{\sigma}}

\def\ovl{\overline}

\newcommand{\tr}[2]{\mathchoice
	{#1\raise -1.8pt\vbox{\hbox{$\kern -.8pt/\mathsmaller{#2} $}}}
	{#1\raise -1.8pt\vbox{\hbox{$\kern -.8pt/#2$}}\kern .8pt}
	{#1\raise -1.8pt\vbox{\hbox{$\scriptstyle\kern -.8pt /#2$}}}
	{#1\raise -1.8pt\vbox{\hbox{$\scriptscriptstyle\kern -.8pt /#2$}}}}

\newcommand{\trbis}[2]{\mathchoice
	{#1\raise -1.8pt\vbox{\hbox{$\kern -.8pt\mathsmaller{/#2} $}}}
	{#1\raise -1.8pt\vbox{\hbox{$\kern -.8pt\mathsmaller{/#2}$}}\kern .8pt}
	{#1\raise -1.8pt\vbox{\hbox{$\scriptstyle\kern -.8pt /#2$}}}
	{#1\raise -1.8pt\vbox{\hbox{$\scriptscriptstyle\kern -.8pt /#2$}}}}
\newcommand{\overslice}[2]{\mathchoice
	{#1\raise -1.8pt\vbox{\hbox{$\kern -.8pt\mathsmaller{#2/} $}}}
	{#1\raise -1.8pt\vbox{\hbox{$\kern -.8pt\mathsmaller{#2/}$}}\kern .8pt}
	{#1\raise -1.8pt\vbox{\hbox{$\scriptstyle\kern -.8pt #2/$}}}
	{#1\raise -1.8pt\vbox{\hbox{$\scriptscriptstyle\kern -.8pt #2/$}}}}

	\makeatletter

\def\labelstylecode@triangle#1{%
	\pgfkeys@split@path%
	\edef\label@key{/triangle/label/\pgfkeyscurrentname}%
	\edef\style@key{\pgfkeyscurrentkey/.@val}%
	\def\temp@a{#1}%
	\def\temp@b{\pgfkeysnovalue}%
	\ifx\temp@a\temp@b
	\pgfkeysgetvalue{\label@key}\temp@a
	\ifx\temp@a\temp@b\else
	\pgfkeysalso{commutative diagrams/.cd, \style@key}%
	\fi
	\else
	\pgfkeys{\style@key/.code = \pgfkeysalso{#1}}%
	\fi}
\def\arrowstylecode@triangle#1{%
	\edef\style@key{\pgfkeyscurrentkey/.@val}%
	\def\temp@a{#1}%
	\def\temp@b{\pgfkeysnovalue}%
	\ifx\temp@a\temp@b
	\pgfkeysalso{commutative diagrams/.cd, \style@key}%
	\else
	\pgfkeys{\style@key/.code = \pgfkeysalso{#1}}%
	\fi}

\pgfkeys{
	/triangle/label/.cd,
	0/.initial = {$\bullet$}, 1/.initial = {$\bullet$}, 2/.initial = {$\bullet$},
	01/.initial, 12/.initial, 02/.initial,
	012/.initial,
	/triangle/labelstyle/.cd,
	01/.@val/.initial=\pgfkeysalso{swap}, 12/.@val/.initial=\pgfkeysalso{swap},
	02/.@val/.initial,
	012/.@val/.initial=\pgfkeysalso{near start},
	01/.code=\labelstylecode@triangle{#1}, 02/.code=\labelstylecode@triangle{#1},
	12/.code=\labelstylecode@triangle{#1}, 012/.code=\labelstylecode@triangle{#1},
	/triangle/arrowstyle/.cd,
	01/.@val/.initial, 12/.@val/.initial,
	02/.@val/.initial,
	012/.@val/.initial=\pgfkeysalso{Rightarrow},
	01/.code=\arrowstylecode@triangle{#1}, 02/.code=\arrowstylecode@triangle{#1},
	12/.code=\arrowstylecode@triangle{#1}, 012/.code=\arrowstylecode@triangle{#1},
}

\def\tr@abc{%
	\draw [/triangle/arrowstyle/012] (90:0.20) --
	node [/triangle/labelstyle/012] {
		\pgfkeysvalueof{/triangle/label/012}} (270:0.10);
}

\def\tr@#1#2{
	\begin{scope}[shift=#2, commutative diagrams/every diagram]
		
		\node (n{#1}0) at (150:1) {
			\pgfkeysvalueof{/triangle/label/0}};
		\node (n{#1}1) at (270:0.6) {
			\pgfkeysvalueof{/triangle/label/1}};
		\node (n{#1}2) at (30:1) {
			\pgfkeysvalueof{/triangle/label/2}};			
		
		\node (s#1) at (0,0) [circle, inner sep = 0pt,
		fit = (n{#1}0.center)(n{#1}1.center)(n{#1}2.center)] {};
		
		\begin{scope}[commutative diagrams/.cd, every arrow, every label]
			\ifcase #1
			\def\list{0/1, 1/2, 0/2}\or
			\def\list{0/1, 1/2, 0/2}\else
			\def\list{}\fi
			
			\foreach \s / \e in \list {
				\draw [/triangle/arrowstyle/\s\e] (n{#1}\s) --
				node [/triangle/labelstyle/\s\e] {
					\pgfkeysvalueof{/triangle/label/\s\e}} (n{#1}\e);
			}
			
			\ifcase #1
			\tr@abc\or
			\tr@abc
			\else\fi
			
		\end{scope}
	\end{scope}
}

\def\triangle#1{
	\pgfkeys{#1}
	\tr@{0}{(0:0)}
}

\makeatother

\makeatletter

\def\labelstylecode@square#1{%
	\pgfkeys@split@path%
	\edef\label@key{/square/label/\pgfkeyscurrentname}%
	\edef\style@key{\pgfkeyscurrentkey/.@val}%
	\def\temp@a{#1}%
	\def\temp@b{\pgfkeysnovalue}%
	\ifx\temp@a\temp@b
	\pgfkeysgetvalue{\label@key}\temp@a
	\ifx\temp@a\temp@b\else
	\pgfkeysalso{commutative diagrams/.cd, \style@key}%
	\fi
	\else
	\pgfkeys{\style@key/.code = \pgfkeysalso{#1}}%
	\fi}
\def\arrowstylecode@square#1{%
	\edef\style@key{\pgfkeyscurrentkey/.@val}%
	\def\temp@a{#1}%
	\def\temp@b{\pgfkeysnovalue}%
	\ifx\temp@a\temp@b
	\pgfkeysalso{commutative diagrams/.cd, \style@key}%
	\else
	\pgfkeys{\style@key/.code = \pgfkeysalso{#1}}%
	\fi}

\pgfkeys{
	/square/label/.cd,
	0/.initial = {$\bullet$}, 1/.initial = {$\bullet$}, 2/.initial = {$\bullet$},
	3/.initial = {$\bullet$},
	01/.initial, 12/.initial, 23/.initial,
	02/.initial, 03/.initial, 13/.initial,
	012/.initial, 013/.initial, 023/.initial, 123/.initial,
	0123/.initial,
	/square/labelstyle/.cd,
	01/.@val/.initial=\pgfkeysalso{swap},
	12/.@val/.initial=\pgfkeysalso{swap},
	23/.@val/.initial=\pgfkeysalso{swap},
	03/.@val/.initial,
	02/.@val/.initial=\pgfkeysalso{description},
	13/.@val/.initial=\pgfkeysalso{description},
	012/.@val/.initial=\pgfkeysalso{above left = -1pt and -1pt},
	013/.@val/.initial=\pgfkeysalso{swap, right},
	023/.@val/.initial=\pgfkeysalso{left},
	123/.@val/.initial=\pgfkeysalso{swap, above right = -1pt and -1pt},
	0123/.@val/.initial,
	01/.code=\labelstylecode@square{#1}, 02/.code=\labelstylecode@square{#1},
	03/.code=\labelstylecode@square{#1}, 12/.code=\labelstylecode@square{#1},
	13/.code=\labelstylecode@square{#1}, 23/.code=\labelstylecode@square{#1},
	012/.code=\labelstylecode@square{#1}, 013/.code=\labelstylecode@square{#1},
	023/.code=\labelstylecode@square{#1}, 123/.code=\labelstylecode@square{#1},
	0123/.code=\labelstylecode@square{#1},
	/square/arrowstyle/.cd,
	01/.@val/.initial, 12/.@val/.initial, 23/.@val/.initial,
	02/.@val/.initial, 03/.@val/.initial, 13/.@val/.initial,
	012/.@val/.initial=\pgfkeysalso{Rightarrow},
	013/.@val/.initial=\pgfkeysalso{Rightarrow},
	023/.@val/.initial=\pgfkeysalso{Rightarrow},
	123/.@val/.initial=\pgfkeysalso{Rightarrow},
	0123/.@val/.initial=\pgfkeysalso{Rightarrow},
	01/.code=\arrowstylecode@square{#1}, 02/.code=\arrowstylecode@square{#1},
	03/.code=\arrowstylecode@square{#1}, 12/.code=\arrowstylecode@square{#1},
	13/.code=\arrowstylecode@square{#1}, 23/.code=\arrowstylecode@square{#1},
	012/.code=\arrowstylecode@square{#1}, 013/.code=\arrowstylecode@square{#1},
	023/.code=\arrowstylecode@square{#1}, 123/.code=\arrowstylecode@square{#1},
	0123/.code=\arrowstylecode@square{#1}
}

\def\sq@abc{%
	\draw [/square/arrowstyle/012] (235:0.25) --
	node [/square/labelstyle/012] {
		\pgfkeysvalueof{/square/label/012}} (235:0.6);
}
\def\sq@bcd{%
	\draw [/square/arrowstyle/123] (-54:0.25) --
	node [/square/labelstyle/123] {
		\pgfkeysvalueof{/square/label/123}} (-54:0.6);
}
\def\sq@acd{%
	\draw [/square/arrowstyle/023] (55:0.55) --
	node [/square/labelstyle/023] {
		\pgfkeysvalueof{/square/label/023}} (15:0.45);
}
\def\sq@abd{%
	\draw [/square/arrowstyle/013] (125:0.55) --
	node [/square/labelstyle/013] {
		\pgfkeysvalueof{/square/label/013}} (165:0.45);
}

\def\sq@#1#2{
	\begin{scope}[shift=#2, commutative diagrams/every diagram]
		
		\foreach \i in {0,1,2,3} {
			\tikzmath{\a = 135 + (90 * \i);}
			\node (n{#1}\i) at (\a:1) {
				\pgfkeysvalueof{/square/label/\i}};
		}
		
		\node (s#1) at (0,0) [circle, inner sep = 0pt,
		fit = (n{#1}0.center)(n{#1}1.center)(n{#1}2.center)
		(n{#1}3.center)] {};
		
		\begin{scope}[commutative diagrams/.cd, every arrow, every label]
			\ifcase #1
			\def\list{0/1, 1/2, 2/3, 0/2, 0/3}\or
			\def\list{0/1, 1/2, 2/3, 1/3, 0/3}\else
			\def\list{}\fi
			
			\foreach \s / \e in \list {
				\draw [/square/arrowstyle/\s\e] (n{#1}\s) --
				node [/square/labelstyle/\s\e] {
					\pgfkeysvalueof{/square/label/\s\e}} (n{#1}\e);
			}
			
			\ifcase #1
			\sq@abc\sq@acd\or
			\sq@abd\sq@bcd
			\else\fi
			
		\end{scope}
	\end{scope}
}

\def\square#1{
	\pgfkeys{#1}
	\sq@{0}{(180:1.5)}\sq@{1}{(0:1.5)}
	
	\begin{scope}[commutative diagrams/.cd, every arrow, every label]
		\draw[->] [shorten >=10pt, shorten <=10pt, /square/arrowstyle/0123] (s0) --
		node [/square/labelstyle/0123] {%
			\pgfkeysvalueof{/square/label/0123}} (s1);
		
	\end{scope}
}

\makeatother

\makeatletter

\def\labelstylecode@pent#1{%
	\pgfkeys@split@path%
	\edef\label@key{/pentagon/label/\pgfkeyscurrentname}%
	\edef\style@key{\pgfkeyscurrentkey/.@val}%
	\def\temp@a{#1}%
	\def\temp@b{\pgfkeysnovalue}%
	\ifx\temp@a\temp@b
	\pgfkeysgetvalue{\label@key}\temp@a
	\ifx\temp@a\temp@b\else
	\pgfkeysalso{commutative diagrams/.cd, \style@key}%
	\fi
	\else
	\pgfkeys{\style@key/.code = \pgfkeysalso{#1}}%
	\fi}
\def\arrowstylecode@pent#1{%
	\edef\style@key{\pgfkeyscurrentkey/.@val}%
	\def\temp@a{#1}%
	\def\temp@b{\pgfkeysnovalue}%
	\ifx\temp@a\temp@b
	\pgfkeysalso{commutative diagrams/.cd, \style@key}%
	\else
	\pgfkeys{\style@key/.code = \pgfkeysalso{#1}}%
	\fi}

\pgfkeys{
	/pentagon/label/.cd,
	0/.initial = {$\bullet$}, 1/.initial = {$\bullet$}, 2/.initial = {$\bullet$},
	3/.initial = {$\bullet$}, 4/.initial = {$\bullet$},
	01/.initial, 12/.initial, 23/.initial, 34/.initial, 04/.initial,
	02/.initial, 03/.initial, 13/.initial, 14/.initial, 24/.initial,
	012/.initial, 013/.initial, 014/.initial, 023/.initial, 024/.initial,
	034/.initial, 123/.initial, 124/.initial, 134/.initial, 234/.initial,
	0123/.initial, 0124/.initial, 0134/.initial, 0234/.initial, 1234/.initial,
	01234/.initial,
	/pentagon/labelstyle/.cd,
	01/.@val/.initial, 12/.@val/.initial, 23/.@val/.initial, 34/.@val/.initial,
	04/.@val/.initial=\pgfkeysalso{swap},
	02/.@val/.initial=\pgfkeysalso{description},
	03/.@val/.initial=\pgfkeysalso{description},
	13/.@val/.initial=\pgfkeysalso{description},
	14/.@val/.initial=\pgfkeysalso{description},
	24/.@val/.initial=\pgfkeysalso{description},
	012/.@val/.initial=\pgfkeysalso{swap, above},
	013/.@val/.initial=\pgfkeysalso{below = 1pt},
	014/.@val/.initial=\pgfkeysalso{below},
	023/.@val/.initial=\pgfkeysalso{swap, above = 1pt},
	024/.@val/.initial=\pgfkeysalso{below left = -1pt and -1pt},
	034/.@val/.initial=\pgfkeysalso{swap, right},
	123/.@val/.initial=\pgfkeysalso{below left = -1pt and -1pt},
	124/.@val/.initial=\pgfkeysalso{swap, right},
	134/.@val/.initial=\pgfkeysalso{left},
	234/.@val/.initial=\pgfkeysalso{swap, below right = -1pt and -1pt},
	0123/.@val/.initial,
	0124/.@val/.initial=\pgfkeysalso{swap},
	0134/.@val/.initial,
	0234/.@val/.initial=\pgfkeysalso{swap},
	1234/.@val/.initial,
	01234/.@val/.initial,
	01/.code=\labelstylecode@pent{#1}, 02/.code=\labelstylecode@pent{#1},
	03/.code=\labelstylecode@pent{#1}, 04/.code=\labelstylecode@pent{#1},
	12/.code=\labelstylecode@pent{#1}, 13/.code=\labelstylecode@pent{#1},
	14/.code=\labelstylecode@pent{#1}, 23/.code=\labelstylecode@pent{#1},
	24/.code=\labelstylecode@pent{#1}, 34/.code=\labelstylecode@pent{#1},
	012/.code=\labelstylecode@pent{#1}, 013/.code=\labelstylecode@pent{#1},
	014/.code=\labelstylecode@pent{#1}, 023/.code=\labelstylecode@pent{#1},
	024/.code=\labelstylecode@pent{#1}, 034/.code=\labelstylecode@pent{#1},
	123/.code=\labelstylecode@pent{#1}, 124/.code=\labelstylecode@pent{#1},
	134/.code=\labelstylecode@pent{#1}, 234/.code=\labelstylecode@pent{#1},
	0123/.code=\labelstylecode@pent{#1}, 0124/.code=\labelstylecode@pent{#1},
	0134/.code=\labelstylecode@pent{#1}, 0234/.code=\labelstylecode@pent{#1},
	1234/.code=\labelstylecode@pent{#1}, 01234/.code=\labelstylecode@pent{#1},
	/pentagon/arrowstyle/.cd,
	01/.@val/.initial, 12/.@val/.initial, 23/.@val/.initial, 34/.@val/.initial,
	04/.@val/.initial, 02/.@val/.initial, 03/.@val/.initial, 13/.@val/.initial,
	14/.@val/.initial, 24/.@val/.initial,
	012/.@val/.initial=\pgfkeysalso{Rightarrow},
	013/.@val/.initial=\pgfkeysalso{Rightarrow},
	014/.@val/.initial=\pgfkeysalso{Rightarrow},
	023/.@val/.initial=\pgfkeysalso{Rightarrow},
	024/.@val/.initial=\pgfkeysalso{Rightarrow},
	034/.@val/.initial=\pgfkeysalso{Rightarrow},
	123/.@val/.initial=\pgfkeysalso{Rightarrow},
	124/.@val/.initial=\pgfkeysalso{Rightarrow},
	134/.@val/.initial=\pgfkeysalso{Rightarrow},
	234/.@val/.initial=\pgfkeysalso{Rightarrow},
	0123/.@val/.initial=\pgfkeysalso{triple},
	0124/.@val/.initial=\pgfkeysalso{triple},
	0134/.@val/.initial=\pgfkeysalso{triple},
	0234/.@val/.initial=\pgfkeysalso{triple},
	1234/.@val/.initial=\pgfkeysalso{triple},
	01234/.@val/.initial=\pgfkeysalso{quadruple},
	01/.code=\arrowstylecode@pent{#1}, 02/.code=\arrowstylecode@pent{#1},
	03/.code=\arrowstylecode@pent{#1}, 04/.code=\arrowstylecode@pent{#1},
	12/.code=\arrowstylecode@pent{#1}, 13/.code=\arrowstylecode@pent{#1},
	14/.code=\arrowstylecode@pent{#1}, 23/.code=\arrowstylecode@pent{#1},
	24/.code=\arrowstylecode@pent{#1}, 34/.code=\arrowstylecode@pent{#1},
	012/.code=\arrowstylecode@pent{#1}, 013/.code=\arrowstylecode@pent{#1},
	014/.code=\arrowstylecode@pent{#1}, 023/.code=\arrowstylecode@pent{#1},
	024/.code=\arrowstylecode@pent{#1}, 034/.code=\arrowstylecode@pent{#1},
	123/.code=\arrowstylecode@pent{#1}, 124/.code=\arrowstylecode@pent{#1},
	134/.code=\arrowstylecode@pent{#1}, 234/.code=\arrowstylecode@pent{#1},
	0123/.code=\arrowstylecode@pent{#1}, 0124/.code=\arrowstylecode@pent{#1},
	0134/.code=\arrowstylecode@pent{#1}, 0234/.code=\arrowstylecode@pent{#1},
	1234/.code=\arrowstylecode@pent{#1}, 01234/.code=\arrowstylecode@pent{#1}
}

\def\pent@abc{%
	\draw [/pentagon/arrowstyle/012] (198:0.45) --
	node [/pentagon/labelstyle/012] {
		\pgfkeysvalueof{/pentagon/label/012}} (198:0.8);
}
\def\pent@bcd{%
	\draw [/pentagon/arrowstyle/123] (126:0.45) --
	node [/pentagon/labelstyle/123] {
		\pgfkeysvalueof{/pentagon/label/123}} (126:0.8);
}
\def\pent@cde{%
	\draw [/pentagon/arrowstyle/234] (54:0.45) --
	node [/pentagon/labelstyle/234] {
		\pgfkeysvalueof{/pentagon/label/234}} (54:0.8);
}
\def\pent@ade{%
	\draw [/pentagon/arrowstyle/034] (-40:0.6) --
	node [/pentagon/labelstyle/034] {
		\pgfkeysvalueof{/pentagon/label/034}} (-5:0.5);
}
\def\pent@abe{%014
	\draw [/pentagon/arrowstyle/014] (-70:0.55) --
	node [/pentagon/labelstyle/014] {
		\pgfkeysvalueof{/pentagon/label/014}} (-110:0.55);
}
\def\pent@acd{%
	\draw [/pentagon/arrowstyle/023] (55:0.3) --
	node [/pentagon/labelstyle/023] {
		\pgfkeysvalueof{/pentagon/label/023}} (125:0.3);
}
\def\pent@bde{%
	\draw [/pentagon/arrowstyle/134] (-5:0.4) --
	node [/pentagon/labelstyle/134] {
		\pgfkeysvalueof{/pentagon/label/134}} (35:0.5);
}
\def\pent@ace{%
	\draw [/pentagon/arrowstyle/024] (-45:0.45) --
	node [/pentagon/labelstyle/024] {
		\pgfkeysvalueof{/pentagon/label/024}} (-45:0.1);
}
\def\pent@abd{%
	\draw [/pentagon/arrowstyle/013] (-90:0.22) --
	node [/pentagon/labelstyle/013] {
		\pgfkeysvalueof{/pentagon/label/013}} (-150:0.46);
}
\def\pent@bce{%
	\draw [/pentagon/arrowstyle/124] (188:0.4) --
	node [/pentagon/labelstyle/124] {
		\pgfkeysvalueof{/pentagon/label/124}} (150:0.55);
}

\def\pent@#1#2{
	\begin{scope}[shift=#2, commutative diagrams/every diagram]
		
		\foreach \i in {0,1,2,3,4} {
			\tikzmath{\a = 270 - (72 * \i);}
			\node (n{#1}\i) at (\a:1) {
				\pgfkeysvalueof{/pentagon/label/\i}};
		}
		
		\node (p#1) at (0,0) [circle, inner sep = 0pt,
		fit = (n{#1}0.center)(n{#1}1.center)(n{#1}2.center)
		(n{#1}3.center)(n{#1}4.center)] {};
		
		\begin{scope}[commutative diagrams/.cd, every arrow, every label]
			\ifcase #1
			\def\list{0/1, 1/2, 2/3, 3/4, 0/4, 0/2, 0/3}\or
			\def\list{0/1, 1/2, 2/3, 3/4, 0/4, 1/3, 1/4}\or
			\def\list{0/1, 1/2, 2/3, 3/4, 0/4, 0/2, 2/4}\or
			\def\list{0/1, 1/2, 2/3, 3/4, 0/4, 0/3, 1/3}\or
			\def\list{0/1, 1/2, 2/3, 3/4, 0/4, 1/4, 2/4}\else
			\def\list{}\fi
			
			\foreach \s / \e in \list {
				\draw [/pentagon/arrowstyle/\s\e] (n{#1}\s) --
				node [/pentagon/labelstyle/\s\e] {
					\pgfkeysvalueof{/pentagon/label/\s\e}} (n{#1}\e);
			}
			
			\ifcase #1
			\pent@abc\pent@acd\pent@ade\or
			\pent@bcd\pent@bde\pent@abe\or
			\pent@cde\pent@ace\pent@abc\or
			\pent@ade\pent@abd\pent@bcd\or
			\pent@abe\pent@bce\pent@cde
			\else\fi
			
		\end{scope}
	\end{scope}
}

\def\pentagon#1{
  \pgfkeys{#1}
  \pent@{2}{(270:3)}\pent@{0}{(198:3)}\pent@{3}{(126:3)}
  \pent@{1}{(54:3)}\pent@{4}{(342:3)}

  \begin{scope}[commutative diagrams/.cd, every arrow, every label]
    \draw [/pentagon/arrowstyle/0123] (p0) --
    node [/pentagon/labelstyle/0123] {
      \pgfkeysvalueof{/pentagon/label/0123}} (p3);

    \draw [/pentagon/arrowstyle/0134] (p3) --
    node [/pentagon/labelstyle/0134] {
      \pgfkeysvalueof{/pentagon/label/0134}} (p1);

    \draw [/pentagon/arrowstyle/1234] (p1) --
    node [/pentagon/labelstyle/1234] {
      \pgfkeysvalueof{/pentagon/label/1234}} (p4);

    \draw [/pentagon/arrowstyle/0234] (p0) --
    node [/pentagon/labelstyle/0234] {
      \pgfkeysvalueof{/pentagon/label/0234}} (p2);

    \draw [/pentagon/arrowstyle/0124] (p2) --
    node [/pentagon/labelstyle/0124] {
      \pgfkeysvalueof{/pentagon/label/0124}} (p4);

    \draw [/pentagon/arrowstyle/01234] (270:0.75) --
    node [/pentagon/labelstyle/01234] {
      \pgfkeysvalueof{/pentagon/label/01234}} (90:0.75);
  \end{scope}
}

\makeatother